\documentclass{article}

\usepackage[T2A]{fontenc}
\usepackage[utf8]{inputenc}
\usepackage[russian]{babel}
\usepackage{amsfonts}
\usepackage{amsthm}
\usepackage{amsmath}
\usepackage{amssymb}
\usepackage{cancel}
\usepackage{tabu}
\usepackage{array}
\usepackage{tikz}
\usepackage{hyperref}
\usepackage{makecell}
\usepackage{xcolor}
\usepackage{graphicx}
\usepackage[final]{pdfpages}

\graphicspath{ {./images/} }

\title{Перечисление путей в графе Юнга--Фибоначчи}

\author{В. Ю. Евтушевский}

\begin{document}

\maketitle


\tableofcontents

\newpage

\newtheorem{Lemma}{Лемма}

\newtheorem{Alg}{Алгоритм}

\newtheorem{Col}{Следствие}

\newtheorem{theorem}{Теорема}

\newtheorem{Def}{Определение}

\newtheorem{Prop}{Утверждение}

\newtheorem{Problem}{Задача}

\newtheorem{Zam}{Замечание}

\newtheorem{Oboz}{Обозначение}

\newtheorem{Ex}{Пример}

\newtheorem{Nab}{Наблюдение}

\renewcommand{\labelenumi}{\theenumi)}
\renewcommand{\labelenumii}{\arabic{enumii}$^\circ$}
\renewcommand{\labelenumiii}{\arabic{enumii}.\arabic{enumiii}$^\circ$}

\section{Введение}

Рассмотрим слова над алфавитом $\{1,2\}$ с данной суммой цифр $n$.
Как известно, их количество есть число Фибоначчи $F_{n+1}$
($F_0=0,F_1=1,F_{k+2}=F_{k+1}+F_k$), и это самая распространённая
комбинаторная
интерпретация чисел Фибоначчи. Также можно думать 
о разбиениях полосы $2\times n$ на домино $1\times 2$ и $2\times 1$,
сопоставляя двойки парам горизонтальных домино, а 
единицы вертикальным домино.

Введём на этом множестве слов частичный порядок: будем говорить,
что слово $x$ предшествует слову $y$, если после удаления 
общего суффикса в слове $y$ остаётся не меньше
двоек, чем в слове $x$ остаётся цифр. 

Это действительно частичный порядок, более того, 
соответствующее частично упорядоченное множество
является модулярной
решёткой, известной как решётка Юнга -- Фибоначчи.

\begin{center}
\includegraphics[width=12cm, height=10cm]{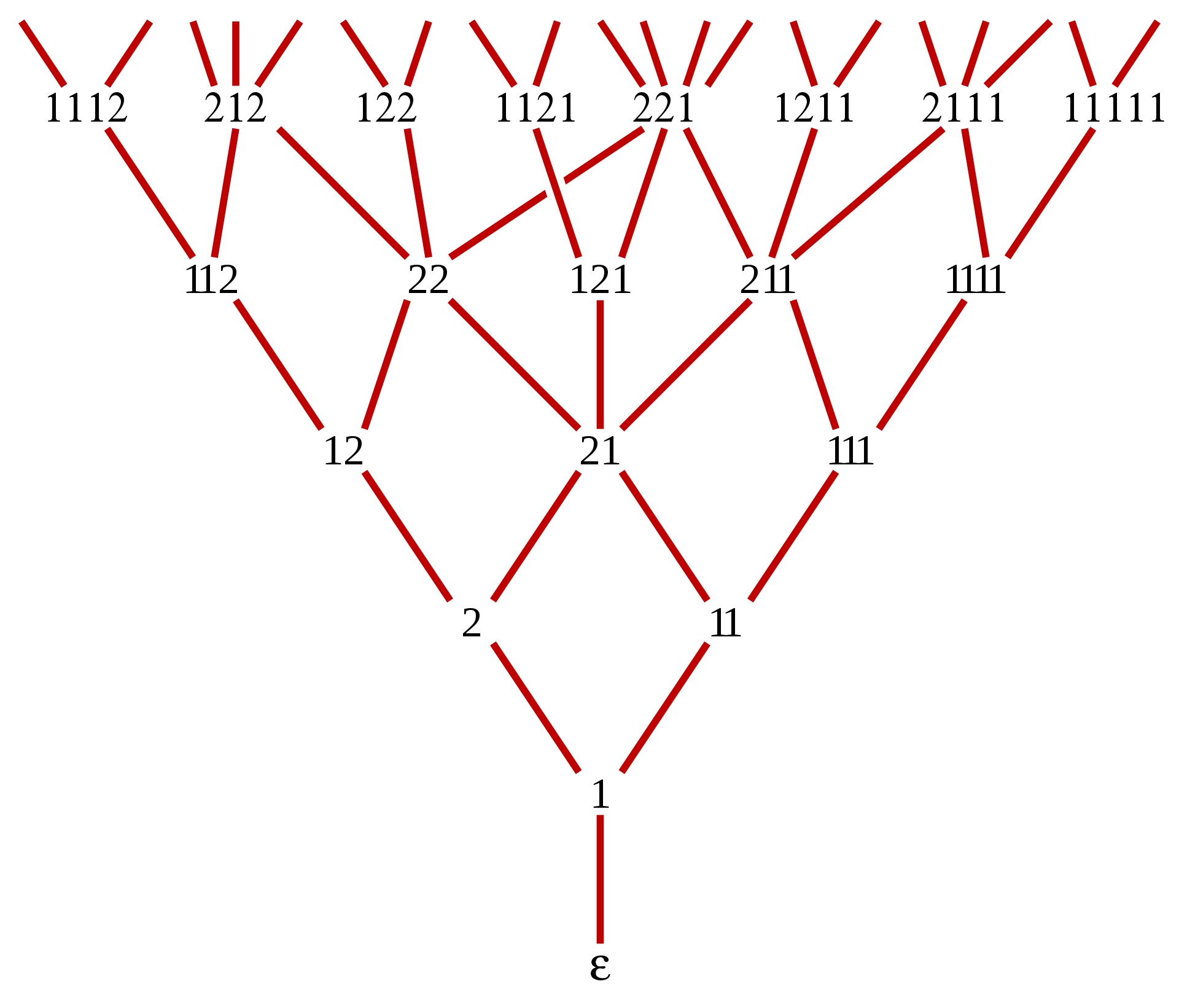}
\end{center}

Графом Юнга -- Фибоначчи (он изображён на рисунке выше) называют диаграмму Хассе этой
решётки. Это градуированный граф, который мы представляем
растущим снизу вверх начиная с пустого слова. 
Градуировкой служит функция суммы цифр. Опишем явно,
как устроены ориентированные
рёбра. Рёбра из данного
слова $x$ ведут в слова, получаемые из $x$ одной из двух операций:
\begin{enumerate}
    \item  заменить самую левую единицу на двойку;

    \item вставить единицу левее чем самая левая единица.
\end{enumerate}

Этот граф помимо модулярности является 1-дифференциальным, то есть
для каждой вершины исходящая степень на 1 превосходит 
входящую степень.

Изучение градуированного графа Юнга -- Фибоначчи было инициировано
в 1988 году одновременно и независимо такими математиками, как
Ричард Стенли \cite{St} и Сергей Владимирович Фомин \cite{Fo}.

Причина интереса к нему в том, что существует всего две 1-дифференциальных
модулярных решётки, вторая --- это решётка диаграмм Юнга, 
имеющая ключевое значение в теории представлений симметрической
группы. 

Центральные вопросы о градуированных графах касаются центральных мер на 
пространстве (бесконечных) путей в графе. Эта точка зрения
последовательно развивалась в работах Анатолия Моисеевича Вершика,
к недавнему обзору которого
\cite{Ver} и приводимой там литературе мы отсылаем читателя.

Среди центральных мер выделяют те, которые являются пределами 
мер, индуцированных путями в далёкие вершины --- так называемую
границу Мартина графа.

Граница пространства путей графа Юнга -- Фибоначчи изучалась в работе
Фредерика Гудмана и Сергея Васильевича Керова (2000) \cite{GK}.

Они использовали алгебраический формализм Окады \cite{Okada}.

Как следует из самого определения, 
асимптотический вопрос о границе напрямую связан с перечислительным
вопросом о числе путей между двумя вершинами графа. 
Отметим важную общую работу С. В. Фомина \cite{Fo1} о перечислении
путей в градуированных графах, в которой приводится
ряд общих тождеств и указывается связь помимо прочего с обобщением
алгоритма Робинсона -- Шенстеда -- Кнута .

Гудман и Керов обходятся без явных формул для числа путей, хотя, как
указал автору Павел Павлович Никитин, из
их рассуждений и можно их извлечь --- но количество слагаемых оказывается
экспоненциальным по длине меньшего из слов. 

В разделе 2 мы предлагаем явную
формулу с квадратичным (в худшем случае) числом слагаемых. 
В некотором смысле это также лучше, чем известная детерминантная формула
Фейта \cite{Feit}
для графа Юнга: в определителе, если его раскрыть очень много слагаемых.
Это обстоятельство позволяет рассчитывать на применимость
предлагаемых формул в асимптотических вопросах.

В разделе 3 приводится формула для числа путей между двумя вершинами,
если ходить можно в обоих направлениях. В ней кубическое (в худшем случае) число слагаемых. 

Подход, используемый в настоящей работе совершенно элементарен, большинство формулировок основаны
на изучении конкретных случаев и догадках, а доказательств --- на вложенной
индукции. 

\newpage
\section{Теорема о количестве путей ``вниз'' между двумя вершинами в графе Юнга-Фибоначчи}

\subsection{Подготовка к формулировке теоремы и её доказательству}

\begin{Oboz}
Пусть $\mathbb{YF}$ -- это граф Юнга-Фибоначчи.
\end{Oboz}

\begin{Def}
Пусть $x \in \mathbb{YF}$. Тогда номером $x$ будем называть слово из единиц и двоек, соответствующее вершине $x$.
\end{Def}

\begin{Oboz}
Если $x\in \mathbb{YF}$ и номер $x$ -- это $\alpha_1\alpha_2...\alpha_n$, где $\alpha_i \in \{1,2 \}$, то будем писать, что $x=\alpha_1\alpha_2...\alpha_n$.
\end{Oboz}

\begin{Oboz}
Если $x \in \mathbb{YF}$, то сумму цифр в номере $x$ обозначим за $|x|$.
\end{Oboz}

\begin{Zam}
$|x|$ -- это ранг вершины $x$ в графе Юнга-Фибоначчи.
\end{Zam}

\begin{Def}\label{vniz}
Пусть $x,y \in \mathbb{YF}: |y|\ge |x|$. Тогда путь 
$$y=y_0y_1y_2...y_n=x$$
в графе Юнга-Фибоначчи назовём $yx$-путём ``вниз'', если $\forall i \in \{0,...,n\} \;|y_i|=|y|-i$. Количество $yx$-путей ``вниз'' в $\mathbb{YF}$ будем обозначать как $d(x,y)$.
\end{Def}

Наша задача в этом параграфе заключается в нахождении $d(x,y)$ для любой пары $x,y\in\mathbb{YF}$.

\begin{Zam}
В Определении \ref{vniz} $n=|y|-|x|$.
\end{Zam}

\begin{Oboz}
Пусть $x \in \mathbb{YF}$. Тогда:
\begin{itemize} 
    \item Количество цифр в номере $x$ обозначим за $\#x$;
    \item Количество двоек в номере $x$ обозначим за $d(x)$;
    \item Количество двоек до первой единицы в номере $x$ обозначим за $d'(x)$;
    \item Множество предков $x$ обозначим за $R(x)$, множество потомков $x$ обозначим за $L(x)$, количество предков $x$ обозначим за $l(x)=|L(x)|$, количество потомков $x$ обозначим за $r(x)=|R(x)|$. 
\end{itemize}
\end{Oboz}    
    
\begin{Zam}
При $x \in \mathbb{YF}$ выполняется одно из двух условий:
\begin{enumerate}
    \item $x$ содержит хотя бы одну единицу, и $r(x)=d'(x)+1$;
    \item $x$ не содержит ни одной единицы, и $r(x)=d'(x)$.
\end{enumerate}
\end{Zam}

\begin{Oboz}
При $x \in \mathbb{YF}$ вершины из $R(x)$ обозначим следующим образом: 
\begin{itemize}
    \item Вершину с номером, получающимся заменой первой двойки на единицу, обозначим за $x_1$;
    \item Вершину с номером, получающимся заменой второй двойки на единицу, обозначим за $x_2$;
    \item ...
    \item Вершину с номером, получающимся заменой $d'(x)$-ой двойки на единицу, обозначим за $x_{d'(x)}$;
    \item Вершину с номером, получающимся удалением первой единицы (если она есть), обозначим за $x_{d'(x)+1}$.
\end{itemize}
\end{Oboz}

\begin{Def}
При $x \in \mathbb{YF}$ нижней функцией для $x$ назовём функцию
$$f(x,y,z) \quad (y \in \{0,...,|x|\}, \; z\in \{0,...,\#x\}),$$
определённую следующим образом:

При $z=0$:
\begin{itemize}
\item Если $x\in \mathbb{YF}$ представляется в виде $x=\alpha_1...\alpha_m\alpha_{m+1}...\alpha_n$, где $|\alpha_{m+1}...\alpha_n|=y, \alpha_i \in \{1,2 \}$, то
$$f(x,y,0):=\frac{1}{(\alpha_{m+1})(\alpha_{m+1}+\alpha_{m+2})...(\alpha_{m+1}+...+\alpha_{n})}\cdot(-1)^{n-m}\cdot$$
$$\cdot\frac{1}{\alpha_m(\alpha_m+\alpha_{m-1})(\alpha_m+\alpha_{m-1}+\alpha_{m-2})...(\alpha_m+...+\alpha_1)}=$$
$$=\frac{1}{(-\alpha_{m+1})(-\alpha_{m+1}-\alpha_{m+2})...(-\alpha_{m+1}-...-\alpha_{n})}\cdot$$
$$\cdot\frac{1}{(\alpha_m)(\alpha_m+\alpha_{m-1})(\alpha_m+\alpha_{m-1}+\alpha_{m-2})...(\alpha_m+...+\alpha_1)};$$
\item Если $x\in \mathbb{YF}$ не представляется в виде $x=\alpha_1...\alpha_m\alpha_{m+1}...\alpha_n$, где  $|\alpha_{m+1}...\alpha_n|=y,$ $\alpha_i \in \{1,2 \}$, то
$$f(x,y,0)=0.$$
\end{itemize}

При $z>0$ (рекурсивное определение):
\begin{itemize}
    \item Если $y=0$, то
$$f(x1,0,z)=f(x1,0,0);$$
    \item Если $y>0$, то
$$f(x1,y,z)=f(x1,y,0)+f(x,y-1,z-1);$$
    \item 
\begin{equation*}
f(x2,y,z)= 
 \begin{cases}
   $$\frac{f(x11,y,z+1)}{1-y}$$ &\text {если $y \ne 1$}\\
   0 &\text{если $y=1$.}
 \end{cases}
\end{equation*}
\end{itemize}
\end{Def}

\begin{Ex} \label{exf1}
Построение нижней функции $f(x,y,z)$ для $x=21221$ при $z=0$:

\begin{itemize}
\item $$f(21221,0,0)=\frac{1}{1\cdot3\cdot5\cdot6\cdot8}=\frac{1}{720};$$
\item $$f(21221,1,0)=\frac{1}{(-1)\cdot2\cdot4\cdot5\cdot7}=-\frac{1}{280};$$
\item $$f(21221,2,0)=0;$$
\item $$f(21221,3,0)=\frac{1}{(-2)\cdot(-3)\cdot2\cdot3\cdot5}=\frac{1}{180};$$
\item $$f(21221,4,0)=0;$$
\item $$f(21221,5,0)=\frac{1}{(-2)\cdot(-4)\cdot(-5)\cdot1\cdot3}=-\frac{1}{120};$$
\item $$f(21221,6,0)=\frac{1}{(-1)\cdot(-3)\cdot(-5)\cdot(-6)\cdot2}=\frac{1}{180};$$
\item $$f(21221,7,0)=0;$$
\item $$f(21221,8,0)=\frac{1}{(-2)\cdot(-3)\cdot(-5)\cdot(-7)\cdot(-8)}=-\frac{1}{1680}.$$
\end{itemize}
\end{Ex}

\begin{Ex} \label{exf2}
Значения $f(x,y,z)$ при $|x| \le 4$:

\begin{itemize}
    \item $x=\varepsilon$
\begin{center}
\begin{tabular}{ | m{0.5cm} || m{0.5cm} | } 
  \hline
 & $y=0$ \\
  \hline \hline
$z=0$ & $1$ \\ 
  \hline
\end{tabular}
\end{center}
    \item $x=1$
\begin{center}
\begin{tabular}{ | m{0.5cm} || m{0.5cm} | m{0.5cm} | } 
  \hline
 & $y=0$ & $y=1$ \\
  \hline \hline
$z=0$ & $1$ & $-1$ \\
  \hline
$z=1$ & $1$ & $0$ \\ 
  \hline
\end{tabular}
\end{center}
    \item $x=2$
\begin{center}
\begin{tabular}{ | m{0.5cm} || m{0.5cm} | m{0.5cm} | m{0.5cm} | } 
  \hline
 & $y=0$ & $y=1$ & $y=2$ \\
  \hline \hline
$z=0$ & $\frac{1}{2}$ & $0$ & $-\frac{1}{2}$\\
  \hline
$z=1$ & $\frac{1}{2}$ & $0$ &$-\frac{1}{2}$ \\ 
  \hline
\end{tabular}
\end{center}
    \item $x=11$
\begin{center}
\begin{tabular}{ | m{0.5cm} || m{0.5cm} | m{0.5cm} | m{0.5cm} | } 
  \hline
 & $y=0$ & $y=1$ & $y=2$ \\
  \hline \hline
$z=0$ & $\frac{1}{2}$ & $-1$ & $\frac{1}{2}$\\
  \hline
$z=1$ & $\frac{1}{2}$ & $0$ &$-\frac{1}{2}$ \\ 
  \hline
$z=2$ & $\frac{1}{2}$ & $0$ & $\frac{1}{2}$\\ 
  \hline
\end{tabular}
\end{center}
    \item $x=12$
\begin{center}
\begin{tabular}{ | m{0.5cm} || m{0.5cm} | m{0.5cm} | m{0.5cm} | m{0.5cm} | } 
  \hline
 & $y=0$ & $y=1$ & $y=2$ & $y=3$\\
  \hline \hline
$z=0$ & $\frac{1}{6}$ & $0$ & $-\frac{1}{2}$ &$\frac{1}{3}$\\
  \hline
$z=1$ & $\frac{1}{6}$ & $0$ &$-\frac{1}{2}$ &$\frac{1}{3}$\\ 
  \hline
$z=2$ & $\frac{1}{6}$ & $0$ & $-\frac{1}{2}$ &$-\frac{1}{6}$\\ 
  \hline
\end{tabular}
\end{center}
    \item $x=21$
\begin{center}
\begin{tabular}{ | m{0.5cm} || m{0.5cm} | m{0.5cm} | m{0.5cm} | m{0.5cm} | } 
  \hline
 & $y=0$ & $y=1$ & $y=2$ & $y=3$\\
  \hline \hline
$z=0$ & $\frac{1}{3}$ & $-\frac{1}{2}$ & $0$ &$\frac{1}{6}$\\
  \hline
$z=1$ & $\frac{1}{3}$ & $0$ &$0$ &$-\frac{1}{3}$\\ 
  \hline
$z=2$ & $\frac{1}{3}$ & $0$ & $0$ &$-\frac{1}{3}$\\ 
  \hline
\end{tabular}
\end{center}
    \item $x=111$
\begin{center}
\begin{tabular}{ | m{0.5cm} || m{0.5cm} | m{0.5cm} | m{0.5cm} | m{0.5cm} | } 
  \hline
 & $y=0$ & $y=1$ & $y=2$ & $y=3$\\
  \hline \hline
$z=0$ & $\frac{1}{6}$ & $-\frac{1}{2}$ & $\frac{1}{2}$ &$-\frac{1}{6}$\\
  \hline
$z=1$ & $\frac{1}{6}$ & $0$ &$-\frac{1}{2}$ &$\frac{1}{3}$\\ 
  \hline
$z=2$ & $\frac{1}{6}$ & $0$ & $\frac{1}{2}$ &$-\frac{2}{3}$\\ 
  \hline
$z=3$ & $\frac{1}{6}$ & $0$ & $\frac{1}{2}$ &$\frac{1}{3}$\\ 
  \hline
\end{tabular}
\end{center}
    \item $x=112$
\begin{center}
\begin{tabular}{ | m{0.5cm} || m{0.5cm} | m{0.5cm} | m{0.5cm} | m{0.5cm} | m{0.5cm} | } 
  \hline
 & $y=0$ & $y=1$ & $y=2$ & $y=3$ & $y=4$\\
  \hline \hline
$z=0$ & $\frac{1}{24}$ & $0$ & $-\frac{1}{4}$ &$\frac{1}{3}$ &$-\frac{1}{8}$\\
  \hline
$z=1$ & $\frac{1}{24}$ & $0$ &$-\frac{1}{4}$ &$\frac{1}{3}$ &$-\frac{1}{8}$\\ 
  \hline
$z=2$ & $\frac{1}{24}$ & $0$ & $-\frac{1}{4}$ &$-\frac{1}{6}$ &$\frac{5}{24}$\\ 
  \hline
$z=3$ & $\frac{1}{24}$ & $0$ & $-\frac{1}{4}$ &$-\frac{1}{6}$ &$-\frac{1}{8}$\\ 
  \hline
\end{tabular}
\end{center}
    \item $x=22$
\begin{center}
\begin{tabular}{ | m{0.5cm} || m{0.5cm} | m{0.5cm} | m{0.5cm} | m{0.5cm} | m{0.5cm} | } 
  \hline
 & $y=0$ & $y=1$ & $y=2$ & $y=3$ & $y=4$\\
  \hline \hline
$z=0$ & $\frac{1}{8}$ & $0$ & $-\frac{1}{4}$ & $0$ &$\frac{1}{8}$\\
  \hline
$z=1$ & $\frac{1}{8}$ & $0$ & $-\frac{1}{4}$ & $0$ &$\frac{1}{8}$\\
  \hline
$z=2$ & $\frac{1}{8}$ & $0$ & $-\frac{1}{4}$ & $0$ &$\frac{1}{8}$\\
  \hline
\end{tabular}
\end{center}
\item $x=121$
\begin{center}
\begin{tabular}{ | m{0.5cm} || m{0.5cm} | m{0.5cm} | m{0.5cm} | m{0.5cm} | m{0.5cm} | } 
  \hline
 & $y=0$ & $y=1$ & $y=2$ & $y=3$ & $y=4$\\
  \hline \hline
$z=0$ & $\frac{1}{12}$ & $-\frac{1}{6}$ & $0$ &$\frac{1}{6}$ &$-\frac{1}{12}$\\
  \hline
$z=1$ & $\frac{1}{12}$ & $0$ &$0$ &$-\frac{1}{3}$ &$\frac{1}{4}$\\ 
  \hline
$z=2$ & $\frac{1}{12}$ & $0$ & $0$ &$-\frac{1}{3}$ &$\frac{1}{4}$\\
  \hline
$z=3$ & $\frac{1}{12}$ & $0$ & $0$ &$-\frac{1}{3}$ &$-\frac{1}{4}$\\
  \hline
\end{tabular}
\end{center}
\item $x=211$
\begin{center}
\begin{tabular}{ | m{0.5cm} || m{0.5cm} | m{0.5cm} | m{0.5cm} | m{0.5cm} | m{0.5cm} | } 
  \hline
 & $y=0$ & $y=1$ & $y=2$ & $y=3$ & $y=4$\\
  \hline \hline
$z=0$ & $\frac{1}{8}$ & $-\frac{1}{3}$ & $\frac{1}{4}$ & $0$ &$-\frac{1}{24}$\\
  \hline
$z=1$ & $\frac{1}{8}$ & $0$ &$-\frac{1}{4}$ &$0$ &$\frac{1}{8}$\\ 
  \hline
$z=2$ & $\frac{1}{8}$ & $0$ & $\frac{1}{4}$ &$0$ &$-\frac{3}{8}$\\ 
  \hline
$z=3$ & $\frac{1}{8}$ & $0$ & $\frac{1}{4}$ &$0$ &$-\frac{3}{8}$\\ 
  \hline
\end{tabular}
\end{center}
\item $x=1111$
\begin{center}
\begin{tabular}{ | m{0.5cm} || m{0.5cm} | m{0.5cm} | m{0.5cm} | m{0.5cm} | m{0.5cm} | } 
  \hline
 & $y=0$ & $y=1$ & $y=2$ & $y=3$ & $y=4$\\
  \hline \hline
$z=0$ & $\frac{1}{24}$ & $-\frac{1}{6}$ & $\frac{1}{4}$ &$-\frac{1}{6}$ &$\frac{1}{24}$\\
  \hline
$z=1$ & $\frac{1}{24}$ & $0$ &$-\frac{1}{4}$ &$\frac{1}{3}$ &$-\frac{1}{8}$\\ 
  \hline
$z=2$ & $\frac{1}{24}$ & $0$ & $\frac{1}{4}$ &$-\frac{2}{3}$ &$\frac{3}{8}$\\ 
  \hline
$z=3$ & $\frac{1}{24}$ & $0$ & $\frac{1}{4}$ &$\frac{1}{3}$ &$-\frac{5}{8}$\\ 
  \hline
$z=4$ & $\frac{1}{24}$ & $0$ & $\frac{1}{4}$ &$\frac{1}{3}$ &$\frac{3}{8}$\\ 
  \hline
\end{tabular}
\end{center}
\end{itemize}
\end{Ex}

Примеры \ref{exf1} и \ref{exf2} могут помочь понять некоторые переходы в доказательстве ближайшей теоремы. 

\begin{Def}
При $x \in \mathbb{YF}$ верхней функцией для $x$ назовём функцию
$$g(x,y) \quad (y \in \{1,...,d(x)\}),$$
определённую следующим образом:

Рассмотрим представление $x$ в виде $$x=\underbrace{1...1}_{\beta_{d(x)}}2\underbrace{1...1}_{\beta_{d(x)-1}}2...2\underbrace{1...1}_{\beta_1}2\underbrace{1...1}_{\beta_0}$$
и определим:
\begin{itemize}
    \item $g(x,1)=\beta_0+1;$
    \item $g(x,2)=\beta_0+\beta_1+3;$
    \item ...
    \item $g(x,m)=\beta_0+...+\beta_{m-1}+2m-1;$
    \item ...
    \item $g(x,d(x))=\beta_0+...+\beta_{d(x)-1}+2d(x)-1.$
\end{itemize}

\end{Def}

\begin{Oboz}
Пусть $x,y \in \mathbb{YF}$. Тогда вершину, номер которой -- это конкатенация номеров $x$ и $y$, обозначим за $xy$. 
\end{Oboz}

\begin{Oboz}
Пусть $x,y \in \mathbb{YF}$. Тогда максимальное $z\in\mathbb{N}_0$: $\exists x',y',z' \in \mathbb{YF}:\quad x=x'z'$, $y=y'z'$, $\#z'=z$, обозначим за $h(x,y)$.
\end{Oboz}

\begin{Zam}
Если $x,y \in \mathbb{YF}$, то $h(x,y)$ -- это количество цифр в самом длинном общем суффиксе $x$ и $y$. Ясно, что $h(x,y)=h(y,x)$.
\end{Zam}

Итак, мы готовы к тому, чтобы сформулировать и доказать теорему о количестве путей ``вниз'' между двумя данными вершинами в графе Юнга-Фибоначчи.

\subsection{Формулировка теоремы и первые утверждения для её доказательства}

\begin{theorem} \label{TH1}
Пусть $x,y \in \mathbb{YF}$: $|y| \ge |x|$ и $z=h(x,y)$. Тогда 
$$d(x,y)=\sum_{i=0}^{|x|}\left( {f\left(x,i,z\right)}\prod_{j=1}^{d(y)}\left(g\left(y,j\right)-i\right)\right).$$ 
\end{theorem}
\begin{proof}
Обозначим предполагаемый ответ за $o(x,y)$.

Для начала докажем несколько вспомогательных утверждений:

\begin{Prop} \label{z0}
Пусть $x \in \mathbb{YF}$. Тогда:
\begin{enumerate}
    \item Если $y \in \{1,...,|x1| \}$, то $-yf(x1,y,0)=f(x,y-1,0);$
    \item Если $y \in \{0,...,|x2| \}$, то $(1-y)f(x11,y,0)=f(x2,y,0);$
    \item Если $y \in \{0,...,|x|\}, \; \alpha_0 \in\{1,2\}$, то  $(|\alpha_0x|-y)f(\alpha_0x,y,0)=f(x,y,0).$
\end{enumerate}
\end{Prop}
\begin{proof}
$\;$
\begin{enumerate}
    \item
Рассмотрим два случая:    
\begin{enumerate}
    \item Пусть $x$ представляется в виде $x=\alpha_1...\alpha_m\alpha_{m+1}...\alpha_n$, где  $|\alpha_{m+1}...\alpha_n|=y-1, \; \alpha_i \in \{1,2 \}$.

Тогда $x1$ представляется в виде
$x1=\alpha_1...\alpha_m\alpha_{m+1}...\alpha_n1$, где $|\alpha_{m+1}...\alpha_n1|=y.$ Посчитаем, пользуясь определением формулы $f$:
$$f(x,y-1,0)=\frac{1}{(-\alpha_{m+1})(-\alpha_{m+1}-\alpha_{m+2})...(-\alpha_{m+1}-...-\alpha_{n})}\cdot$$
$$\cdot\frac{1}{(\alpha_m)(\alpha_m+\alpha_{m-1})(\alpha_m+\alpha_{m-1}+\alpha_{m-2})...(\alpha_m+...+\alpha_1)}=$$
$$=\frac{1}{(-\alpha_{m+1})(-\alpha_{m+1}-\alpha_{m+2})...(-\alpha_{m+1}-...-\alpha_{n})}\cdot\frac{-y}{(-\alpha_{m+1}-...-\alpha_n-1)}\cdot$$
$$\cdot\frac{1}{(\alpha_m)(\alpha_m+\alpha_{m-1})(\alpha_m+\alpha_{m-1}+\alpha_{m-2})...(\alpha_m+...+\alpha_1)}=-yf(x1,y,0).$$

\item Пусть $x$ не представляется в виде $x=\alpha_1...\alpha_m\alpha_{m+1}...\alpha_n$, где $|\alpha_{m+1}...\alpha_n|=y-1, \; \alpha_i \in \{1,2 \}$.

Тогда $x1$ не представляется в виде
$x1=\alpha_1...\alpha_m\alpha_{m+1}...\alpha_n1$, где $|\alpha_{m+1}...\alpha_n1|=y, \; \alpha_i \in \{1,2 \}$. Очевидно, что
$$f(x,y-1,0)=0=-yf(x1,y,0).$$
\end{enumerate}

    \item
Рассмотрим четыре случая:    
\begin{enumerate}
    \item Пусть $y\ge 2$ и $x2$ представляется в виде $x2=\alpha_1...\alpha_m\alpha_{m+1}...\alpha_n2$, где  $|\alpha_{m+1}...\alpha_n2|=y, \; \alpha_i \in \{1,2 \}$.

В данном случае понятно, что $x11$ представляется в виде
$x11=\alpha_1...\alpha_m\alpha_{m+1}...\alpha_n11$, где $|\alpha_{m+1}...\alpha_n11|=y, \; \alpha_i \in \{1,2 \}$. Посчитаем, пользуясь определением формулы $f$:
$$f(x2,y,0)=$$
$$=\frac{1}{(-\alpha_{m+1})(-\alpha_{m+1}-\alpha_{m+2})...(-\alpha_{m+1}-...-\alpha_{n})(-\alpha_{m+1}-...-\alpha_{n}-2)}\cdot$$
$$\cdot\frac{1}{(\alpha_m)(\alpha_m+\alpha_{m-1})(\alpha_m+\alpha_{m-1}+\alpha_{m-2})...(\alpha_m+...+\alpha_1)}=$$
$$=\frac{1}{(-\alpha_{m+1})(-\alpha_{m+1}-\alpha_{m+2})...(-\alpha_{m+1}-...-\alpha_{n})}\cdot$$
$$\cdot\frac{1-y}{(-\alpha_{m+1}-...-\alpha_n-1)}\cdot\frac{1}{(-\alpha_{m+1}-...-\alpha_n-2)}\cdot$$
$$\cdot\frac{1}{(\alpha_m)(\alpha_m+\alpha_{m-1})(\alpha_m+\alpha_{m-1}+\alpha_{m-2})...(\alpha_m+...+\alpha_1)}=$$
$$=(1-y)f(x11,y,0).$$

\item Пусть $y \ge 2$ и $x2$ не представляется в виде $x2=\alpha_1...\alpha_m\alpha_{m+1}...\alpha_n2$, где $|\alpha_{m+1}...\alpha_n2|=y, \; \alpha_i \in \{1,2 \}$.

В данном случае понятно, что $x11$ не представляется в виде
$x11=\alpha_1...\alpha_m\alpha_{m+1}...\alpha_n1$, где $|\alpha_{m+1}...\alpha_n11|=y, \; \alpha_i \in \{1,2 \}$. Очевидно, что
$$f(x11,y,0)=0=f(x2,y,0)\Longrightarrow (1-y)f(x11,y,0)=0=f(x2,y,0).$$

\item Пусть $y = 1$.

В данном случае понятно, что $x2$ не представляется в виде $x2=\alpha_1...\alpha_m\alpha_{m+1}...\alpha_n2$, где $|\alpha_{m+1}...\alpha_n2|=1, \; \alpha_i \in \{1,2 \}$, поэтому по определению функции $f$ 
$$f(x2,1,0)=0.$$

Очевидно, что $$(1-1)f(x11,1,0)=0,$$
поэтому в данном случае равенство доказано.

\item Пусть $y=0$.
    
В данном случае рассмотрим представление $x2$ в виде $x2=\alpha_1...\alpha_n2$, где $\alpha_i \in \{1,2 \}$.

Тогда $x11$ представляется в виде
$x11=\alpha_1...\alpha_n11$. 
    
Ясно, что данном представлении формула для $f$ принимает вид
$$f(x2,0,0)=$$
$$=\frac{1}{2(2+\alpha_n)(2+\alpha_n+\alpha_{n-1})(2+\alpha_n+\alpha_{n-1}+\alpha_{n-2})...(2+\alpha_n+...+\alpha_1)}=$$
$$=\frac{1}{1\cdot2(2+\alpha_n)(2+\alpha_n+\alpha_{n-1})(2+\alpha_n+\alpha_{n-1}+\alpha_{n-2})...(2+\alpha_n+...+\alpha_1)}=$$
$$=f(x11,0,0)\cdot\frac{1}{(\alpha_m)(\alpha_m+\alpha_{m-1})(\alpha_m+\alpha_{m-1}+\alpha_{m-2})...(\alpha_m+...+\alpha_1)}=$$
$$=(1-y)f(x11,y,0).$$
\end{enumerate}

\item Рассмотрим два случая:

\begin{enumerate}
    \item Пусть $x$ представляется в виде $x=\alpha_1...\alpha_m\alpha_{m+1}...\alpha_n$, где $|\alpha_{m+1}...\alpha_n|=y, \; \alpha_i \in \{1,2 \}$.

В данном случае понятно, что $\alpha_0 x$ представляется в виде
$\alpha_0x=\alpha_0 \alpha_1...\alpha_m\alpha_{m+1}...\alpha_n$, где $|\alpha_{m+1}...\alpha_n|=y.$ Посчитаем:
$$f(x,y,0)=\frac{1}{(-\alpha_{m+1})(-\alpha_{m+1}-\alpha_{m+2})...(-\alpha_{m+1}-...-\alpha_{n})}\cdot$$
$$\cdot\frac{1}{(\alpha_m)(\alpha_m+\alpha_{m-1})(\alpha_m+\alpha_{m-1}+\alpha_{m-2})...(\alpha_m+...+\alpha_1)}=$$
$$=\frac{1}{(-\alpha_{m+1})(-\alpha_{m+1}-\alpha_{m+2})...(-\alpha_{m+1}-...-\alpha_{n})}\cdot$$
$$\cdot\frac{1}{(\alpha_m)(\alpha_m+\alpha_{m-1})(\alpha_m+\alpha_{m-1}+\alpha_{m-2})...(\alpha_m+...+\alpha_1)}\cdot\frac{|\alpha_0x|-y}{(\alpha_{m}+...+\alpha_1+\alpha_0)}=$$
$$=\left(|\alpha_0x|-y\right)f\left(\alpha_0 x,y,0\right).$$

\item Пусть $x$ не представляется в виде $x=\alpha_1...\alpha_m\alpha_{m+1}...\alpha_n$, где $|\alpha_{m+1}...\alpha_n|=y, \; \alpha_i \in \{1,2 \}$.

В данном случае понятно, что $\alpha_0x$ не представляется в виде
$\alpha_0x=\alpha_0\alpha_1...\alpha_m\alpha_{m+1}...\alpha_n$, где $|\alpha_{m+1}...\alpha_n|=y, \; \alpha_i \in \{1,2 \}$. Очевидно,
$$f(x,y,0)=0=\left(|\alpha_0x|-y\right)f\left(\alpha_0 x,y,0\right).$$
\end{enumerate}
\end{enumerate}
\end{proof}

\begin{Lemma}\label{dim}
Пусть $a\in \mathbb{Z}$, $b,c\in \mathbb{N}_0$: $c \le b$. Тогда
$$\prod_{d=0}^{b}\left( a+2d\right)-\sum_{e = 0}^{c}\left(\prod_{d=0}^{b-e-1}\left(a+2d\right)\prod_{d=b-e}^{b-1}\left(a+2d+1\right)\right)=\prod_{d=0}^{b-c-1}\left(a+2d\right)\prod_{d=b-c-1}^{b-1}\left(a+2d+1\right).$$
\end{Lemma}
\begin{proof}
Зафиксируем $a$ и $b$, после чего будем доказывать лемму по индукции по $c$.

\underline{\textbf{База}}: $c=0$.

Просто посчитаем:
$$\prod_{d=0}^{b}\left( a+2d\right)-\prod_{d=0}^{b-1}\left(a+2d\right)=\left(\prod_{d=0}^{b-1}\left(a+2d\right)\right)\left(a+2b-1 \right)=\prod_{d=0}^{b-1}\left(a+2d\right)\prod_{d=b-1}^{b-1}\left(a+2d+1\right).$$
База доказана.

\underline{\textbf{Переход}} к $c+1\ge1$.

Посчитаем, помня, что $c+1 \le b$:
$$\prod_{d=0}^{b}\left( a+2d\right)-\sum_{e = 0}^{c+1}\left(\prod_{d=0}^{b-e-1}\left(a+2d\right)\prod_{d=b-e}^{b-1}\left(a+2d+1\right)\right)=$$
$$\prod_{d=0}^{b}\left( a+2d\right)-\sum_{e = 0}^{c}\left(\prod_{d=0}^{b-e-1}\left(a+2d\right)\prod_{d=b-e}^{b-1}\left(a+2d+1\right)\right)-\prod_{d=0}^{b-c-2}\left(a+2d\right)\prod_{d=b-c-1}^{b-1}\left(a+2d+1\right)=$$
$$=\text{(по предположению индукции)}=$$
$$=\prod_{d=0}^{b-c-1}\left(a+2d\right)\prod_{d=b-c-1}^{b-1}\left(a+2d+1\right)-\prod_{d=0}^{b-c-2}\left(a+2d\right)\prod_{d=b-c-1}^{b-1}\left(a+2d+1\right)=$$
$$=\left(\prod_{d=0}^{b-c-2}\left(a+2d\right)\right)\left(a+2b-2c-3 \right)\prod_{d=b-c-1}^{b-1}\left(a+2d+1\right)=\prod_{d=0}^{b-c-2}\left(a+2d\right)\prod_{d=b-c-2}^{b-1}\left(a+2d+1\right).$$
\end{proof} 
\begin{Col}\label{prom}
Пусть $a\in \mathbb{Z}$, $b\in \mathbb{N}_0$. Тогда
$$\prod_{d=0}^{b}\left( a+2d\right)=\sum_{e = 0}^{b+1}\left(\prod_{i=0}^{b-e-1}\left(a+2d\right)\prod_{d=b-e}^{b-1}\left(a+2d+1\right)\right).$$
\end{Col}
\begin{proof}
$$\prod_{d=0}^{b}\left( a+2d\right)-\sum_{e = 0}^{b+1}\left(\prod_{d=0}^{b-e-1}\left(a+2d\right)\prod_{d=b-e}^{b-1}\left(a+2d+1\right)\right)=$$
$$=\prod_{d=0}^{b}\left( a+2d\right)-\sum_{e = 0}^{b}\left(\prod_{d=0}^{b-e-1}\left(a+2d\right)\prod_{d=b-e}^{b-1}\left(a+2d+1\right)\right)-\prod_{d=0}^{-2}\left(a+2d\right)\prod_{d=-1}^{b-1}\left(a+2d+1\right)=$$
$$=\text{(по Лемме \ref{dim})}=$$
$$ =\prod_{d=0}^{-1}\left(a+2d\right)\prod_{d=-1}^{b-1}\left(a+2d+1\right)-\prod_{d=0}^{-2}\left(a+2d\right)\prod_{d=-1}^{b-1}\left(a+2d+1\right)=0.$$
\end{proof}
\begin{Col} \label{main}
Пусть $y\in \mathbb{YF}$ содержит хотя бы одну единицу. Тогда $\forall i \in \mathbb{N}_0$
$$\prod_{j=1}^{d(y)}\left(g\left(y,j\right)-i\right)=\sum_{k = 1}^{r(y)} \left( \prod_{j=1}^{d(y_k)}\left(g\left(y_k,j\right)-i\right)\right).$$
\end{Col}

\begin{proof}
Знаем, что $y$ начинается с $d'(y)$ двоек. Заметим, что в данном случае $r(y)=d'(y)+1$. 

\setlength{\unitlength}{0.20mm}
\begin{picture}(400,200)
\put(250,160){$y=2221...$}
\put(300,150){\vector(2,-1){200}}
\put(300,150){\vector(1,-2){50}}
\put(300,150){\vector(-1,-2){50}}
\put(300,150){\vector(-2,-1){200}}
\put(50,30){$y_1=1221...$}
\put(180,30){$y_2=2121...$}
\put(310,30){$y_3=2211...$}
\put(440,30){$y_4=222...$}
\end{picture}

На рисунке выше изображено, как выглядят предки вершины, номер которой содержит хотя бы одну единицу. Таким образом, замечаем, что
$$\prod_{j=1}^{d(y)}\left(g\left(y,j\right)-i\right)=\prod_{j=1}^{d(y)-d'(y)}\left(g\left(y,j\right)-i\right)\prod_{j=0}^{d'(y)-1}\left(g\left(y,d(y)-d'(y)+1\right)-i+2j\right).$$
(тут в правой части первое произведение соответствует двойкам внутри многоточия на рисунке, а второе соответствует двойкам до первой единицы).
$$\prod_{j=1}^{d(y_k)}\left(g\left(y_k,j\right)-i\right)=\prod_{j=1}^{d(y)-d'(y)}\left(g\left(y,j\right)-i\right)\cdot$$
$$\cdot\prod_{j=0}^{d'(y)-k-1}\left(g\left(y,d(y)-d'(y)+1\right)-i+2j\right)\prod_{j=d'(y)-k}^{d'(y)-2}\left(g\left(y,d(y)-d'(y)+1\right)-i+2j+1\right).$$
(тут при каждом $k$ в правой части первое произведение соответствует двойкам внутри многоточия на рисунке, второе соответствует двойкам после первой единицы, но до многоточия на рисунке, а третье соответствует двойкам до первой единицы на рисунке).

Таким образом, нам достаточно доказать, что
$$\prod_{j=0}^{d'(y)-1}\left(g\left(y,d(y)-d'(y)+1\right)-i+2j\right)=$$ 
$$=\sum_{k = 1}^{d'(y)+1}\left(\prod_{j=0}^{d'(y)-k-1}\left(g\left(y,d(y)-d'(y)+1\right)-i+2j\right)\prod_{j=d'(y)-k}^{d'(y)-2}\left(g\left(y,d(y)-d'(y)+1\right)-i+2j+1\right)\right).$$
Это равносильно следующему равенству: 
$$\prod_{d=0}^{d'(y)-1}\left(g\left(y,d(y)-d'(y)+1\right)-i+2d\right)=$$ 
$$=\sum_{e = 0}^{d'(y)}\left(\prod_{d=0}^{d'(y)-e-2}\left(g\left(y,d(y)-d'(y)+1\right)-i+2d\right)\prod_{d=d'(y)-e-1}^{d'(y)-2}\left(g\left(y,d(y)-d'(y)+1\right)-i+2d+1\right)\right).$$
А это как раз Следствие \ref{prom} при $b=d'(y)-1$ и $a=g\left(y,d(y)-d'(y)+1\right)-i.$
\end{proof}

\begin{Col} \label{y000}
Пусть $x,y \in \mathbb{YF},$ $|y|>|x|$, $z=h(x,y)$, $\forall k\in\{1,...,r(y)\}\quad z_k=h(x,y_k)$, и при этом номер $y$ содержит хотя бы одну единицу. Тогда
$${f\left(x'1,0,z\right)}\prod_{j=1}^{d(y)}g\left(y,j\right)=\sum_{k = 1}^{r(y)}\left({f\left(x'1,0,z_k\right)}\prod_{j=1}^{d(y_k)}g\left(y_k,j\right)\right).$$
\end{Col}
\begin{proof}
Посчитаем, пользуясь рекурсивной формулой для $f$:
$${f\left(x'1,0,z\right)}\prod_{j=1}^{d(y)}g\left(y,j\right)=\sum_{k = 1}^{r(y)} \left({f\left(x'1,0,z_k\right)}\prod_{j=1}^{d(y_k)}g\left(y_k,j\right)\right) \Longleftarrow$$
$$\Longleftarrow \text{(По Утверждению \ref{y01})} \Longleftarrow\prod_{j=1}^{d(y)}g\left(y,j\right)=\sum_{k = 1}^{r(y)} \prod_{j=1}^{d(y_k)}g\left(y_k,j\right).$$
А это Следствие \ref{main}.
\end{proof}

\begin{Prop}\label{y01}
Пусть $x\in\mathbb{YF}$. Тогда
\begin{enumerate}
    \item $\forall z_1,z_2\in\{0,...,\#x\} \quad f(x,0,z_1)=f(x,0,z_2);$ \label{p1}
    \item $\forall z \in\{1,...,\#x\} \quad f(x,1,z)=0;$
    \item если $\exists x':x=x'2$, то $f(x,1,0)=0.$
\end{enumerate}
\end{Prop}

\begin{proof}
$\;$

\begin{enumerate}
    \item Рассмотрим три случая:
    \begin{enumerate}
        \item Пусть $\exists x':x=x'1$.
        
        В данном случае по рекурсивному определению функции $f$
        $$f(x'1,0,z_1)=(x'1,0,0)=f(x'1,0,z_2).$$
        
        \item Пусть $\exists x':x=x'2$.
        
        В данном случае по рекурсивному определению функции $f$
        $$f(x'2,0,z_1)=f(x'11,0,z_1+1)=f(x'11,0,0)=f(x'11,0,z_2+1)=f(x'2,0,z_2).$$
        
        \item $x=\varepsilon$ 
        
        Данный случай тривиален.
        
    \end{enumerate}
    
    \item Рассмотрим два случая:
    \begin{enumerate}
        \item Пусть $\exists x':x=x'1$.
        
        В данном случае по Утверждению \ref{z0} $f(x,1,0)=-f(x',0,0)$. Пользуясь этой мыслью и рекурсивной формулой для $f$, считаем:
        $$f(x,1,z)=f(x,1,0)+f(x',0,z-1)=-f(x',0,0)+f(x',0,z-1)=$$
        $$=\text{(по пункту \ref{p1})}=-f(x',0,0)+f(x,0,0)=0.$$

        \item Пусть $\exists x':x=x'2$.
        
        В данном случае по определению функции $f$ 
        $$f(x'2,1,z)=0.$$
    \end{enumerate}
    \item В данном случае по определению функции $f$ 
        $$f(x'2,1,0)=0.$$
\end{enumerate}
\end{proof}

\begin{Prop} \label{211}
Пусть $x\in \mathbb{YF}, \;i \in \{0,...,|x'|+2\}$. Тогда
$$f(x'2,i,0)=f(x'11,i,1).$$
\end{Prop}
\begin{proof}
Рассмотрим $3$ случая:
\renewcommand{\labelenumi}{\arabic{enumi}$^\circ$}
\begin{enumerate}
    \item Пусть $i=0$.
    
    В данном случае по определению функции $f$
    $$f(x'2,0,0)=\frac{f(x'11,0,1)}{1-0}=f(x'11,0,1).$$
    \item Пусть $i=1$.
    
    В данном случае по определению функции $f$ и Утверждению \ref{y01}
    $$f(x'2,1,0)=0=f(x'11,1,1).$$
    \item Пусть $i>1$.
    
    В данном случае по определению функции $f$ 
    $$f(x'11,i,1) = f(x'11,i,0)+f(x'1,i-1,0) = \text{(По Утверждению \ref{z0})}=f(x'11,i,0)\left(1-i\right)=f(x'2,i,0).$$
\end{enumerate}
\end{proof}

\begin{Prop} \label{i1}
    Пусть $x,y \in \mathbb{YF},$ $|y|\ge|x|$, $z=h(x,y)$. Тогда
    $${f\left(x,1,z\right)}\prod_{j=1}^{d(y)}\left(g\left(y,j\right)-1\right)=0.$$
\end{Prop}
\begin{proof}
По Утверждению \ref{y01} во всех случаях, кроме случая, в котором  $z=0$ номер и $x$ оканчивается на единицу, $f(x,1,z)=0$, то есть лемма доказана. Давайте разберём единственный ``плохой'' вариант.

В данном случае из того, что $z=0$, и номер $x$ оканчивается на единицу, следует, что номер $y$ оканчивается на двойку, а значит $g(y,1)=1$, то есть в этом случае лемма тоже доказана.
\end{proof}

\subsection{Первая ключевая лемма}

\begin{Lemma} \label{1key}
Пусть $x,y \in \mathbb{YF},$ $|y|>|x|$: $\forall k$ $z_k=h(x,y_k)$. Тогда $\forall i$
$${f\left(x,i,z\right)}\prod_{j=1}^{d(y)}\left(g\left(y,j\right)-i\right)=\sum_{k = 1}^{r(y)} \left({f\left(x,i,z_k\right)}\prod_{j=1}^{d(y_k)}\left(g\left(y_k,j\right)-i\right)\right).$$ 
\end{Lemma}
\begin{proof}
Докажем Лемму по индукции по $|x|$:

\underline{\textbf{Базу}}: $|x|=0$ проверим в конце.

\underline{\textbf{Переход}}: к $|x|\ge 1$.

Рассмотрим восемь случаев:
\renewcommand{\labelenumi}{\arabic{enumi}$^\circ$}
\renewcommand{\labelenumii}{\arabic{enumi}.\arabic{enumii}$^\circ$}
\begin{enumerate}
    \item Пусть $i=1.$
    
    Заметим, что в данном случае $\forall k \in\{1,...r(y)\} \quad |y_k|\ge|x|$ (так как $|y|>|x|$ и $|y_k|=|y|-1$), а значит, по Утверждению \ref{i1}, применённому к левой части и к каждому слагаемому правой части,
    $${f\left(x,1,z\right)}\prod_{j=1}^{d(y)}\left(g\left(y,j\right)-1\right)=0=\sum_{k = 1}^{r(y)} \left({f\left(x,1,z_k\right)}\prod_{j=1}^{d(y_k)}\left(g\left(y_k,j\right)-1\right)\right).$$ 

    \item Пусть $i\ne 1$ и $\exists x',y': x=x'1 \; y=y'2$, причём номер $y$ содержит хотя бы одну единицу или $\exists x',y': x=x'2, \; y=y'1$, причём номер $y$ содержит хотя бы две единицы.

\setlength{\unitlength}{0.20mm}
\begin{picture}(400,200)
\put(250,160){$y=2221...2$}
\put(300,150){\vector(2,-1){200}}
\put(300,150){\vector(1,-2){50}}
\put(300,150){\vector(-1,-2){50}}
\put(300,150){\vector(-2,-1){200}}
\put(50,30){$y_1=1221...2$}
\put(180,30){$y_2=2121...2$}
\put(310,30){$y_3=2211...2$}
\put(440,30){$y_4=222...2$}
\end{picture}

\setlength{\unitlength}{0.20mm}
\begin{picture}(400,200)
\put(250,160){$y=2221...1$}
\put(300,150){\vector(2,-1){200}}
\put(300,150){\vector(1,-2){50}}
\put(300,150){\vector(-1,-2){50}}
\put(300,150){\vector(-2,-1){200}}
\put(50,30){$y_1=1221...1$}
\put(180,30){$y_2=2121...1$}
\put(310,30){$y_3=2211...1$}
\put(440,30){$y_4=222...1$}
\end{picture}

На рисунках выше изображено, как могут выглядеть предки вершины $y$. Замечаем, что в обоих случаях номера $x$ и $y$ оканчиваются на разные цифры, а, значит, $z=0$. Также замечаем, что в обоих случаях $\forall i\in\{1,...,r(y)\}$ номера $y_i$ и $x$ оканчиваются на разные цифры, а значит, $z_i=0$. Таким образом, мы хотим доказать, что 
$${f\left(x,i,0\right)}\prod_{j=1}^{d(y)}\left(g\left(y,j\right)-i\right)=\sum_{k = 1}^{r(y)} \left({f\left(x,i,0\right)}\prod_{j=1}^{d(y_k)}\left(g\left(y_k,j\right)-i\right)\right) \Longleftarrow$$
$$\Longleftarrow \prod_{j=1}^{d(y)}\left(g\left(y,j\right)-i\right)=\sum_{k = 1}^{r(y)} \left( \prod_{j=1}^{d(y_k)}\left(g\left(y_k,j\right)-i\right)\right).$$
А это Следствие \ref{main}.

\item Пусть $i \ne 1$ и $\exists x': x=x'1$, причём номер $y$ состоит только из двоек:
\setlength{\unitlength}{0.20mm}
\begin{picture}(400,200)
\put(250,160){$y=2222$}
\put(300,150){\vector(2,-1){200}}
\put(300,150){\vector(1,-2){50}}
\put(300,150){\vector(-1,-2){50}}
\put(300,150){\vector(-2,-1){200}}
\put(50,30){$y_1=1222$}
\put(180,30){$y_2=2122$}
\put(310,30){$y_3=2212$}
\put(440,30){$y_4=2221$}
\end{picture}

На рисунке выше изображено, как могут выглядеть предки вершины $y$. 
Замечаем, что номер $y$ оканчивается на $2$, а значит  $z=0$. Также замечаем, что $\forall k \in \{1,...,r(y)-1  \} \quad$ номер $y_k$ оканчивается на $2$, а значит  $z_k=0$. И ещё замечаем, что номер $y_{r(y)}$ оканчивается на $1$, то есть $z_{r(y)}\ge 1$. Таким образом, мы хотим доказать, что 
$${f\left(x,i,0\right)}\prod_{j=1}^{d(y)}\left(g\left(y,j\right)-i\right)=\sum_{k = 1}^{r(y)-1} \left({f\left(x,i,0\right)}\prod_{j=1}^{d(y_k)}\left(g\left(y_k,j\right)-i\right)\right)+f\left(x,i,z_{r(y)}\right)\prod_{j=1}^{d(y_{r(y)})}\left(g\left(y_{r(y)},j\right)-i\right).$$
Заметим, что если номер $y$ не содержит ни одной единицы, то $g(y,j)=2j-1$. Таким образом,
$${f\left(x,i,0\right)}\prod_{j=1}^{d(y)}\left(g\left(y,j\right)-i\right)={f\left(x,i,0\right)}\prod_{j=1}^{d(y)}\left(2j-1-i\right).$$
Также заметим, что если $y$ состоит из всех двоек, то $\forall k \in \{1,...,d(y)\} \quad r(y)=d(y)=d(y_k)+1$, и $y_k$ имеет вид 
$$\underbrace{2...2}_{k-1}1\underbrace{2...2}_{d(y)-k},$$
а значит если $j \in \{1,...,d(y)-k\}$, то $g(y_k,j)=2j-1$, а если $j \in \{d(y)-k+1,...,d(y)-1\}$, то $g(y_k,j)=2j$. Таким образом, при $1\le k \le r(y)-1$
$${f\left(x,i,0\right)}\prod_{j=1}^{d(y_k)}\left(g\left(y_k,j\right)-i\right)=f\left(x,i,0\right)\prod_{j=1}^{d(y)-k}\left(2j-1-i\right)\prod_{j=d(y)-k+1}^{d(y)-1}\left(2j-i\right).$$
Таким образом,
$${f\left(x,i,0\right)}\prod_{j=1}^{d(y)}\left(g\left(y,j\right)-i\right)-\sum_{k = 1}^{r(y)-1} \left({f\left(x,i,0\right)}\prod_{j=1}^{d(y_k)}\left(g\left(y_k,j\right)-i\right)\right)=$$
$$={f\left(x,i,0\right)}\left(\prod_{j=1}^{d(y)}\left(2j-1-i\right) - \sum_{k=1}^{d(y)-1}\left( \prod_{j=1}^{d(y)-k}\left(2j-1-i\right)\prod_{j=d(y)-k+1}^{d(y)-1}\left(2j-i\right) \right)\right)=$$
$$=(d:=j-1; e:=k-1)=$$
$$={f\left(x,i,0\right)}\left(\prod_{d=0}^{d(y)-1}\left(2d+1-i\right) - \sum_{e=0}^{d(y)-2}\left( \prod_{d=0}^{d(y)-k-2}\left(2d+1-i\right)\prod_{j=d(y)-k-1}^{d(y)-2}\left(2j+2-i\right) \right)\right)=$$
$$=\text{(по лемме \ref{dim} при $a=-i+1;$ $b=d(y)-1;$ $c=d(y)-2$)}=$$
$$={f\left(x,i,0\right)}\prod_{d=0}^{0}\left(2d+1-i\right)\prod_{j=1}^{d(y)-2}\left(2d+2-i\right)=$$
$$={f\left(x,i,0\right)}\prod_{j=1}^{1}\left(2j-1-i\right)\prod_{j=2}^{d(y)-1}\left(2j-i\right)={f\left(x,i,0\right)}\left(1-i\right)\prod_{j=1}^{d(y)-1}\left(2j-i\right).$$

Таким образом, наше равенство равносильно следующему равенству: 
$${f\left(x,i,0\right)}\left(1-i\right)\prod_{j=1}^{d(y)-1}\left(2j-i\right)=f\left(x,i,z_{r(y)}\right)\prod_{j=1}^{d(y_{r(y)})}\left(2j-i\right) \Longleftarrow \left(1-i\right){f\left(x,i,0\right)}=f\left(x,i,z_{r(y)}\right).$$

Давайте докажем, что $\left(1-i\right){f\left(x,i,0\right)}=f\left(x,i,z_{r(y)}\right)$ при $x=x'1$ по индукции по $|x'|$.

\underline{\textbf{База}}: $|x'|=0$ (то есть $x=1$).

Знаем, что при $x=1$ единственными $f(x,y,z)$ являются $f(1,0,0)=1$, $f(1,1,0)=-1$, $f(1,0,1)=1$, $f(1,1,1)=0$. Поэтому из того, что $z_{y(r)} \ge 1$ делаем вывод, что $z_{y(r)}=1$.

Рассмотрим два случая:
\renewcommand{\labelenumii}{\roman{enumii}$^\circ$}
\begin{enumerate}
    \item Пусть $i=0$.
    
    В данном случае равенство принимает вид $(1-0)f(1,0,0)=f(1,0,1)$, а это равняется единице с обеих сторон.
    \item Пусть $i=1$.
    
    В данном случае равенство принимает вид $(1-1)f(1,1,0)=f(1,1,1)$, а это равняется нулю с обеих сторон.
\end{enumerate}

\underline{\textbf{Переход}} к $x: \;|x'| \ge 1 \Longleftrightarrow |x| \ge 2$.

Рассмотрим два случая:
\begin{enumerate}
    \item Пусть $\exists x'' :x=x''11$.

Заметим, что в данном случае, $z_{y(r)}=1$. 

Рассмотрим ещё два подслучая:

\renewcommand{\labelenumiii}{\roman{enumii}.\roman{enumiii}$^\circ$}

\begin{enumerate}
    \item Пусть $i>0.$
    
Посчитаем, пользуясь рекурсивной формулой для $f$:
$$\left(1-i\right){f\left(x'1,i,0\right)}=f\left(x'1,i,1\right)\Longleftrightarrow$$$$\Longleftrightarrow \left(1-i\right){f\left(x'1,i,0\right)}=f\left(x'1,i,0\right)+f\left(x',i-1,0\right) \Longleftrightarrow$$
$$\Longleftrightarrow -i{f\left(x'1,i,0\right)}=f\left(x',i-1,0\right).$$
А это верно по Утверждению \ref{z0}.

    \item Пусть $i=0.$
    
    В данном случае равенство принимает вид $$(1-0)f(x,0,0)=f(x,0,1),$$
    а это верно по Утверждению \ref{y01}.
\end{enumerate}

\item Пусть $\exists x'': x=x''21$.

Заметим, что в данном случае $z_{y(r)} \ge 2 $.

Рассмотрим четыре подслучая:

\begin{enumerate}
    \item Пусть $i \ge 3$. 

Посчитаем, несколько раз пользуясь рекурсивной формулой для $f$:
$$ \left(1-i\right){f\left(x''21,i,0\right)}=f\left(x''21,i,z_{y(r)}\right) \Longleftrightarrow$$
$$\Longleftrightarrow \left(1-i\right){f\left(x''21,i,0\right)}=f\left(x''2,i-1,z_{y(r)}-1\right)+f\left(x''21,i,0\right) \Longleftrightarrow$$
$$\Longleftrightarrow -i{f\left(x''21,i,0\right)}=f\left(x''2,i-1,z_{y(r)}-1\right) \Longleftrightarrow$$
$$\Longleftrightarrow \text{(По Утверждению \ref{z0})}\Longleftrightarrow$$
$$\Longleftrightarrow {f\left(x''2,i-1,0\right)}=f\left(x''2,i-1,z_{y(r)}-1\right) \Longleftrightarrow$$
$$\Longleftrightarrow(2-i){f\left(x''11,i-1,0\right)}=\frac{1}{2-i}f\left(x''11,i-1,z_{y(r)}\right) \Longleftrightarrow$$
$$\Longleftrightarrow (2-i){f\left(x''11,i-1,0\right)}=\frac{1}{2-i}f\left(x''11,i-1,0\right)+\frac{1}{2-i}f\left(x''1,i-2,z_{y(r)}-1\right) \Longleftrightarrow$$
$$\Longleftrightarrow (i^2-4i+3){f\left(x''11,i-1,0\right)}=f\left(x''1,i-2,z_{y(r)}-1\right) \Longleftrightarrow$$
$$\Longleftrightarrow\text{(По Утверждению \ref{z0})}\Longleftrightarrow$$
$$\Longleftrightarrow \frac{(i^2-4i+3)}{-i+1}{f\left(x''1,i-2,0\right)}=f\left(x''1,i-2,z_{y(r)}-1\right) \Longleftrightarrow$$
$$\Longleftrightarrow (-i+3){f\left(x''1,i-2,0\right)}=f\left(x''1,i-2,z_{y(r)}-1\right).$$

\item Пусть $i=0.$

В данном случае равенство принимает вид
$$(1-0)f(x''21,0,0)=f(x''21,0,z_{y(r)}).$$ 
А это верно по Утверждению \ref{y01}.

\item Пусть $i=1.$

В данном случае равенство принимает вид
$$(1-1)f(x''21,1,0)=f(x''21,1,z_{y(r)}).$$
А это верно по Утверждению \ref{y01}.

\item Пусть $i=2$.

В данном случае по формуле для функции $f$
$$(1-2)f(x''21,2,0)=0.$$
Также по формуле для функции $f$
$$f(x''21,2,z_{y(r)})=f(x''21,2,0)+f(x''2,1,z_{y(r)}-1)=$$
$$=f(x''2,1,z_{y(r)}-1)=\text{(по Утверждению \ref{y01}})=0.$$
Таким образом, в данном случае 
$$(1-2)f(x''21,2,0)=0=f(x''21,2,z_{y(r)}).$$
Что и требовалось.
\end{enumerate}

\end{enumerate}

\item \label{3} Пусть $i \ne 1$ и $\exists x',y': x=x'1, \;y=y'1$, причём номер $y$ содержит хотя бы две единицы.

Рассмотрим два подслучая:
\renewcommand{\labelenumii}{\arabic{enumi}.\arabic{enumii}$^\circ$}
\begin{enumerate}
    \item Пусть $i>0$.

\setlength{\unitlength}{0.20mm}
\begin{picture}(400,200)
\put(250,160){$y=2221...1$}
\put(300,150){\vector(2,-1){200}}
\put(300,150){\vector(1,-2){50}}
\put(300,150){\vector(-1,-2){50}}
\put(300,150){\vector(-2,-1){200}}
\put(50,30){$y_1=1221...1$}
\put(180,30){$y_2=2121...1$}
\put(310,30){$y_3=2211...1$}
\put(440,30){$y_4=222...1$}
\end{picture}

На рисунке выше изображено, как могут выглядеть предки вершины $y$. 
Замечаем, что номер $y$ оканчивается на $1$, а значит  $z\ge 1$. Также замечаем, что $\forall k \in \{1,...,r(y)\} \quad$ номер $y_k$ оканчивается на $1$, а значит  $z_i\ge 1$. Воспользуемся рекурсивной формулой для $f$:
$${f\left(x'1,i,z\right)}\prod_{j=1}^{d(y)}\left(g\left(y,j\right)-i\right)=\sum_{k = 1}^{r(y)} \left({f\left(x'1,i,z_k\right)}\prod_{j=1}^{d(y_k)}\left(g\left(y_k,j\right)-i\right)\right) \Longleftrightarrow$$
$$\Longleftrightarrow {f\left(x'1,i,0\right)}\prod_{j=1}^{d(y)}\left(g\left(y,j\right)-i\right)+{f\left(x',i-1,z-1\right)}\prod_{j=1}^{d(y)}\left(g\left(y,j\right)-i\right)=$$
$$=\sum_{k = 1}^{r(y)} \left({f\left(x'1,i,0\right)}\prod_{j=1}^{d(y_k)}\left(g\left(y_k,j\right)-i\right)\right)+\sum_{k = 1}^{r(y)} \left({f\left(x',i-1,z_k-1\right)}\prod_{j=1}^{d(y_k)}\left(g\left(y_k,j\right)-i\right)\right).$$

Заметим, что по Следствию \ref{main}
$$ \prod_{j=1}^{d(y)}\left(g\left(y,j\right)-i\right)=\sum_{k = 1}^{r(y)} \prod_{j=1}^{d(y_k)}\left(g\left(y_k,j\right)-i\right)\Longrightarrow$$
$$\Longrightarrow{f\left(x'1,i,0\right)}\prod_{j=1}^{d(y)}\left(g\left(y,j\right)-i\right)=\sum_{k = 1}^{r(y)} \left({f\left(x'1,i,0\right)}\prod_{j=1}^{d(y_k)}\left(g\left(y_k,j\right)-i\right)\right).$$

Таким образом, наше утверждение равносильно следующему:
$${f\left(x',i-1,z-1\right)}\prod_{j=1}^{d(y)}\left(g\left(y,j\right)-i\right)=\sum_{k = 1}^{r(y)
} \left({f\left(x',i-1,z_k-1\right)}\prod_{j=1}^{d(y_k)}\left(g\left(y_k,j\right)-i\right)\right).$$

Рассмотрим $y'$ и $x'$. Пусть $z'_k=h(x',y'_k)$, а $z'=h(x',y')$. Сразу понятно, что из того, что номера $x$ и $y$ оканчиваются на $1$, следует, что $z'=z-1$.

\setlength{\unitlength}{0.20mm}
\begin{picture}(400,200)
\put(250,160){$y=2221...1$}
\put(300,150){\vector(2,-1){200}}
\put(300,150){\vector(1,-2){50}}
\put(300,150){\vector(-1,-2){50}}
\put(300,150){\vector(-2,-1){200}}
\put(50,30){$y_1=1221...1$}
\put(180,30){$y_2=2121...1$}
\put(310,30){$y_3=2211...1$}
\put(440,30){$y_4=222...1$}
\end{picture}

\setlength{\unitlength}{0.20mm}
\begin{picture}(400,200)
\put(250,160){$y'=2221...\xcancel{1}$}
\put(300,150){\vector(2,-1){200}}
\put(300,150){\vector(1,-2){50}}
\put(300,150){\vector(-1,-2){50}}
\put(300,150){\vector(-2,-1){200}}
\put(50,30){$y'_1=1221...\xcancel{1}$}
\put(180,30){$y'_2=2121...\xcancel{1}$}
\put(310,30){$y'_3=2211...\xcancel{1}$}
\put(440,30){$y'_4=222...\xcancel{1}$}
\end{picture}

На рисунках выше изображено, как могут выглядеть предки вершин $y$ и $y'$.

Заметим, что в данном случае $r(y')=r(y)$ и $\forall k \in \{1,...,r(y)\}$ $y_k=y'_k1$, а значит, что $z'_k=z_k-1$.

Теперь мы готовы воспользоваться предположением индукции при $x'$, $y'$ и $i'=i-1$. Воспользуемся:
$${f\left(x',i',z'\right)}\prod_{j=1}^{d(y')}\left(g\left(y',j\right)-i'\right)=\sum_{k = 1}^{r(y')} \left({f\left(x',i',z'_k\right)}\prod_{j=1}^{d(y'_k)}\left(g\left(y'_k,j\right)-i'\right)\right) \Longleftrightarrow$$
$$ \Longleftrightarrow {f\left(x',i-1,z-1\right)}\prod_{j=1}^{d(y)}\left(g\left(y',j\right)-i+1\right)=\sum_{k = 1}^{r(y)} \left({f\left(x',i-1,z_k-1\right)}\prod_{j=1}^{d(y_k)}\left(g\left(y'_k,j\right)-i+1\right)\right) \Longleftrightarrow$$
$$ \Longleftrightarrow {f\left(x',i-1,z-1\right)}\prod_{j=1}^{d(y)}\left(g\left(y,j\right)-i\right)=\sum_{k = 1}^{r(y)} \left({f\left(x',i-1,z_k-1\right)}\prod_{j=1}^{d(y_k)}\left(g\left(y_k,j\right)-i\right)\right).$$
Что и требовалось.

\item \label{threetwo} Пусть $i=0$.

В данном случае утверждение Леммы мгновенно следует из Следствия \ref{y000}.

\end{enumerate}

\item Пусть $i \ne 1$ и $\exists x',y': x=x'1, \;y=y'1$, причём номер $y$ содержит ровно одну единицу.

Рассмотрим два случая:
\renewcommand{\labelenumii}{\arabic{enumi}.\arabic{enumii}$^\circ$}
\begin{enumerate}
    \item Пусть $i>0.$

\begin{picture}(400,200)
\put(280,160){$y=2221$}
\put(300,150){\vector(2,-1){200}}
\put(300,150){\vector(1,-2){50}}
\put(300,150){\vector(-1,-2){50}}
\put(300,150){\vector(-2,-1){200}}
\put(50,30){$y_1=1221$}
\put(200,30){$y_2=2121$}
\put(310,30){$y_3=2211$}
\put(470,30){$y_4=222$}
\end{picture}

На рисунке выше изображено, как могут выглядеть предки вершины $y$. 
Замечаем, что номер $y$ оканчивается на $1$, а значит  $z\ge 1$. Также замечаем, что $\forall k \in \{1,...,r(y)-1  \}$ номер $y_k$ оканчивается на $1$, а значит  $z_k\ge 1$. И ещё замечаем, что номер $y_{r(y)}$ оканчивается на $2$, то есть $z_{r(y)}=0$. Посчитаем, пользуясь рекурсивной формулой для $f$:
$${f\left(x'1,i,z\right)}\prod_{j=1}^{d(y)}\left(g\left(y,j\right)-i\right)=\sum_{k = 1}^{r(y)} \left({f\left(x'1,i,z_k\right)}\prod_{j=1}^{d(y_k)}\left(g\left(y_k,j\right)-i\right)\right) \Longleftrightarrow$$
$$\Longleftrightarrow {f\left(x'1,i,0\right)}\prod_{j=1}^{d(y)}\left(g\left(y,j\right)-i\right)+{f\left(x',i-1,z-1\right)}\prod_{j=1}^{d(y)}\left(g\left(y,j\right)-i\right)=$$
$$=\sum_{k = 1}^{r(y)-1} \left({f\left(x'1,i,0\right)}\prod_{j=1}^{d(y_k)}\left(g\left(y_k,j\right)-i\right)\right)+\sum_{k = 1}^{r(y)-1} \left({f\left(x',i-1,z_k-1\right)}\prod_{j=1}^{d(y_k)}\left(g\left(y_k,j\right)-i\right)\right)+$$
$$+{f\left(x'1,i,0\right)}\prod_{j=1}^{d(y_{r(y)})}\left(g\left(y_{r(y)},j\right)-i\right)\Longleftrightarrow$$
$$\Longleftrightarrow{f\left(x'1,i,0\right)}\prod_{j=1}^{d(y)}\left(g\left(y,j\right)-i\right)+{f\left(x',i-1,z-1\right)}\prod_{j=1}^{d(y)}\left(g\left(y,j\right)-i\right)=$$
$$=\sum_{k = 1}^{r(y)} \left({f\left(x'1,i,0\right)}\prod_{j=1}^{d(y_k)}\left(g\left(y_k,j\right)-i\right)\right)+\sum_{k = 1}^{r(y)-1} \left({f\left(x',i-1,z_k-1\right)}\prod_{j=1}^{d(y_k)}\left(g\left(y_k,j\right)-i\right)\right).$$

Заметим, что по Следствию \ref{main}
$$ \prod_{j=1}^{d(y)}\left(g\left(y,j\right)-i\right)=\sum_{k = 1}^{r(y)} \prod_{j=1}^{d(y_k)}\left(g\left(y_k,j\right)-i\right)\Longrightarrow$$
$$\Longrightarrow{f\left(x'1,i,0\right)}\prod_{j=1}^{d(y)}\left(g\left(y,j\right)-i\right)=\sum_{k = 1}^{r(y)} \left({f\left(x'1,i,0\right)}\prod_{j=1}^{d(y_k)}\left(g\left(y_k,j\right)-i\right)\right).$$

Таким образом, наше утверждение равносильно следующему:
$${f\left(x',i-1,z-1\right)}\prod_{j=1}^{d(y)}\left(g\left(y,j\right)-i\right)=\sum_{k = 1}^{r(y)-1
} \left({f\left(x',i-1,z_k-1\right)}\prod_{j=1}^{d(y_k)}\left(g\left(y_k,j\right)-i\right)\right).$$

Рассмотрим $y'$ и $x'$. Пусть $z'_k=h(x',y'_k)$, а $z'=h(x',y')$. Сразу понятно, что из того, что $x$ и $y$ оканчиваются на $1$, следует, что $z'=z-1$.

\begin{picture}(400,200)
\put(280,160){$y=2221$}
\put(300,150){\vector(2,-1){200}}
\put(300,150){\vector(1,-2){50}}
\put(300,150){\vector(-1,-2){50}}
\put(300,150){\vector(-2,-1){200}}
\put(50,30){$y_1=1221$}
\put(200,30){$y_2=2121$}
\put(310,30){$y_3=2211$}
\put(470,30){$y_4=222$}
\end{picture}

\begin{picture}(400,200)
\put(280,160){$y'=222\xcancel{1}$}
\put(300,150){\vector(2,-1){200}}
\put(300,150){\vector(1,-2){50}}
\put(300,150){\vector(-1,-2){50}}
\put(300,150){\vector(-2,-1){200}}
\put(50,30){$y'_1=122\xcancel{1}$}
\put(200,30){$y'_2=212\xcancel{1}$}
\put(310,30){$y'_3=221\xcancel{1}$}
\put(470,30){$\xcancel{y'_4=222}$}
\end{picture}

На рисунках выше изображено, как могут выглядеть предки вершин $y$ и $y'$.

Заметим, что $r(y')=r(y)-1$ и $\forall k \in \{ 1,...,r(y)-1 \}$ $y_k=y'_k1$, а значит, что $z'_k=z_k-1$. Также заметим, что номер $y_{r(y)}$ оканчивается на $2$, а это значит, $z_{r(y)}=0$.

Теперь мы готовы воспользоваться предположением индукции при $x'$, $y'$ и $i'=i-1$. Воспользуемся:
\renewcommand{\labelenumiii}{\arabic{enumi}.\arabic{enumii}.\arabic{enumiii}$^\circ$}
$${f\left(x',i',z'\right)}\prod_{j=1}^{d(y')}\left(g\left(y',j\right)-i'\right)=\sum_{k = 1}^{r(y')} \left({f\left(x',i',z'_k\right)}\prod_{j=1}^{d(y'_k)}\left(g\left(y'_k,j\right)-i'\right)\right) \Longleftrightarrow$$
$$ \Longleftrightarrow {f\left(x',i-1,z-1\right)}\prod_{j=1}^{d(y)}\left(g\left(y',j\right)-i+1\right)=$$
$$=\sum_{k = 1}^{r(y)-1} \left({f\left(x',i-1,z_k-1\right)}\prod_{j=1}^{d(y_k)}\left(g\left(y'_k,j\right)-i+1\right)\right) \Longleftrightarrow$$
$$ \Longleftrightarrow {f\left(x',i-1,z-1\right)}\prod_{j=1}^{d(y)}\left(g\left(y,j\right)-i\right)=\sum_{k = 1}^{r(y)-1} \left({f\left(x',i-1,z_k-1\right)}\prod_{j=1}^{d(y_k)}\left(g\left(y_k,j\right)-i\right)\right).$$
Что и требовалось.

\item Пусть $i=0$.

Утверждение Леммы в данном случае мгновенно следует из Следствия \ref{y000}.
\end{enumerate}

\item Пусть $i \ne 1$ и $\exists x',y': x=x'2, \; y=y'2$, причём номер $y$ содержит хотя бы одну единицу.

Пусть $y''=y'11$.

\setlength{\unitlength}{0.20mm}
\begin{picture}(400,200)
\put(250,160){$y=2221...2$}
\put(300,150){\vector(2,-1){200}}
\put(300,150){\vector(1,-2){50}}
\put(300,150){\vector(-1,-2){50}}
\put(300,150){\vector(-2,-1){200}}
\put(50,30){$y_1=1221...2$}
\put(180,30){$y_2=2121...2$}
\put(310,30){$y_3=2211...2$}
\put(440,30){$y_4=222...2$}
\end{picture}

\setlength{\unitlength}{0.20mm}
\begin{picture}(400,200)
\put(250,160){$y''=2221...11$}
\put(300,150){\vector(2,-1){200}}
\put(300,150){\vector(1,-2){50}}
\put(300,150){\vector(-1,-2){50}}
\put(300,150){\vector(-2,-1){200}}
\put(50,30){$y''_1=1221...11$}
\put(180,30){$y''_2=2121...11$}
\put(310,30){$y''_3=2211...11$}
\put(440,30){$y''_4=222...11$}
\end{picture}

На рисунке выше изображено, как могут выглядеть предки вершин $y$ и $y''$. Замечаем, что номер $y$ оканчивается на $2$, а значит  $z\ge 1$. Также замечаем, что $\forall k \in \{1,...,r(y)  \}$ номер $y_k$ оканчивается на $2$, а значит  $z_k\ge 1$. И ещё замечаем, что $r(y)=r(y'')$, кроме того, $\forall k \in \{1,...,r(y)\} \quad \exists y'''_k: y'''_k2=y_k \; \& \; y'''_k11=y''_k$. Посчитаем, пользуясь рекурсивной формулой для $f$:
$${f\left(x'2,i,z\right)}\prod_{j=1}^{d(y)}\left(g\left(y,j\right)-i\right)=\sum_{k = 1}^{r(y)} \left({f\left(x'2,i,z_k\right)}\prod_{j=1}^{d(y_k)}\left(g\left(y_k,j\right)-i\right)\right) \Longleftrightarrow$$
$$\Longleftrightarrow 
\left(1-i\right){f\left(x'2,i,z\right)}\frac{\prod_{j=1}^{d(y)}\left(g\left(y,j\right)-i\right)}{1-i}=\sum_{k = 1}^{r(y)} \left(\left(1-i\right){f\left(x'2,i,z_k\right)}\frac{\prod_{j=1}^{d(y_k)}\left(g\left(y_k,j\right)-i\right)}{1-i}\right) \Longleftrightarrow$$
$$\Longleftrightarrow 
{f\left(x'11,i,z+1\right)}\prod_{j=1}^{d(y'')}\left(g\left(y'',j\right)-i\right)=\sum_{k = 1}^{r(y'')} \left({f\left(x'11,i,z_k+1\right)}\prod_{j=1}^{d(y''_k)}\left(g\left(y''_k,j\right)-i\right)\right).$$
А это условие леммы для $x'11$ и $y'11$. Он уже был разобран в случае \ref{3}$^\circ$.

\item Пусть $i \ne 1$ и $\exists x',y': x=x'2,$ $y=y'1$, причём номер $y$ содержит ровно одну единицу.

\begin{picture}(400,200)
\put(280,160){$y=2221$}
\put(300,150){\vector(2,-1){200}}
\put(300,150){\vector(1,-2){50}}
\put(300,150){\vector(-1,-2){50}}
\put(300,150){\vector(-2,-1){200}}
\put(50,30){$y_1=1221$}
\put(200,30){$y_2=2121$}
\put(310,30){$y_3=2211$}
\put(470,30){$y_4=222$}
\end{picture}

На рисунке выше изображено, как могут выглядеть предки вершины $y$. 
Замечаем, что номер $y$ оканчивается на $1$, а значит  $z=0$. Также замечаем, что $\forall k \in \{1,...,r(y)-1\}$ номер $y_k$ оканчивается на $1$, а значит  $z_k=0$. И ещё замечаем, что номер $y_{r(y)}$ оканчивается на $2$, то есть $z_{r(y)}\ge 1$. Посчитаем, пользуясь рекурсивной формулой для $f$:
$${f\left(x'2,i,0\right)}\prod_{j=1}^{d(y)}\left(g\left(y,j\right)-i\right)=\sum_{k = 1}^{r(y)-2} \left({f\left(x'2,i,0\right)}\prod_{j=1}^{d(y_k)}\left(g\left(y_k,j\right)-i\right)\right)+$$
$$+{f\left(x'2,i,0\right)}\prod_{j=1}^{d(y_{r(y)-1})}\left(g\left(y_{r(y)-1},j\right)-i\right)+{f\left(x'2,i,z_{r(y)}\right)}\prod_{j=1}^{d(y_{r(y)})}\left(g\left(y_{r(y)},j\right)-i\right) \Longleftrightarrow$$
$$\Longleftrightarrow{f\left(x'2,i,0\right)}\prod_{j=1}^{d(y)}\left(g\left(y,j\right)-i\right)=\sum_{k = 1}^{r(y)-2} \left({f\left(x'2,i,0\right)}\prod_{j=1}^{d(y_k)}\left(g\left(y_k,j\right)-i\right)\right)+$$
$$+\frac{{f\left(x'2,i,0\right)}}{1-i}\left(\prod_{j=1}^{d(y_{r(y)-1})}\left(g\left(y_{r(y)-1},j\right)-i\right)\right)\left(1-i\right)+$$
$$+{f\left(x'2,i,z_{r(y)}\right)}\left(1-i\right)\frac{\prod_{j=1}^{d(y_{r(y)})}\left(g\left(y_{r(y)},j\right)-i\right)}{1-i} \Longleftrightarrow$$
$$\Longleftrightarrow \text{(По Утверждению \ref{211})} \Longleftrightarrow$$
$$\Longleftrightarrow{f\left(x'11,i,1\right)}\prod_{j=1}^{d(y)}\left(g\left(y,j\right)-i\right)=\sum_{k = 1}^{r(y)-2} \left({f\left(x'11,i,1\right)}\prod_{j=1}^{d(y_k)}\left(g\left(y_k,j\right)-i\right)\right)+$$
$$+{{f\left(x'11,i,0\right)}}\left(\prod_{j=1}^{d(y_{r(y)})}\left(g\left(y_{r(y)},j\right)-i\right)\right)+{f\left(x'11,i,z_{r(y)}+1\right)}{\prod_{j=1}^{d(y_{r(y)-1})}\left(g\left(y_{r(y)-1},j\right)-i\right)}.$$
Я утверждаю, что это условие Леммы для $x'11$ и $y$. Почему это верно? Пусть $z'=h(x'11,y)$ и $z'_k=h(x'11,y_k)$. Несложно заметить, что $|y|\ge|x|\ge 2$, а это значит, что номер $y$ оканчивается на $21$, поэтому $z'=1$. Также понятно, что при $k \in \{1,...,r(y)-2\}$ номер $y_k$ оканчивается на $21$, поэтому $z'_k=1$. Кроме того, ясно, что $y_{r(y)-1}$ имеет вид
$$y_{r(y)-1}=\underbrace{2...2}_{d(y)-1}11,$$
в то время как $y_{r(y)}$ имеет вид $$y_{r(y)}\underbrace{2...2}_{d(y)-1}2,$$
поэтому $z'_{r(y)-1}=h(x'11,y_{r(y)-1})=h(x'2,y_{r(y)})+1=z'_{r(y)}+1$, а номер $y_{r(y)}$ оканчивается на $2$, поэтому $z'_{r(y)}=0$.

Таким образом, это и есть условие Леммы для $x'11$ и $y$. А этот случай уже был разобран в случае $\ref{3}^\circ$.

\item Пусть $i \ne 1$ и $\exists x': x=x'2$, причём номер $y$ состоит только из двоек.

Пусть $y=y'2$; $y''=y'11$.

\begin{picture}(400,200)
\put(280,160){$y=2222$}
\put(300,150){\vector(2,-1){200}}
\put(300,150){\vector(1,-2){50}}
\put(300,150){\vector(-1,-2){50}}
\put(300,150){\vector(-2,-1){200}}
\put(50,30){$y_1=1222$}
\put(200,30){$y_2=2122$}
\put(310,30){$y_3=2212$}
\put(470,30){$y_4=2221$}
\end{picture}

\begin{picture}(400,200)
\put(280,160){$y''=22211$}
\put(300,150){\vector(2,-1){200}}
\put(300,150){\vector(1,-2){50}}
\put(300,150){\vector(-1,-2){50}}
\put(300,150){\vector(-2,-1){200}}
\put(50,30){$y''_1=12211$}
\put(200,30){$y''_2=21211$}
\put(310,30){$y''_3=22111$}
\put(470,30){$y''_4=2221$}
\end{picture}

На рисунке выше изображено, как могут выглядеть предки вершин $y$ и $y''$. Замечаем, что номер $y$ оканчивается на $2$, а значит  $z\ge 1$. Также замечаем, что $\forall k \in \{1,...,r(y)-1  \}\quad$ номер $y_k$ оканчивается на $2$, а значит  $z_k\ge 1$, кроме того номер $y_{r(y)}$ оканчивается на $1$, а значит $z_{r(y)}=0$. И ещё замечаем, что $r(y)=r(y'')$ и $\forall k \in \{1,...,r(y)-1\}\quad\exists y'''_k: y'''_k2=y_k \; \& \; y'''_k11=y''_k$, а $y_{r(y)}=y''_{r(y)}$. Посчитаем, пользуясь рекурсивной формулой для $f:$
$${f\left(x'2,i,z\right)}\prod_{j=1}^{d(y)}\left(g\left(y,j\right)-i\right)=$$
$$=\sum_{k = 1}^{r(y)-1} \left({f\left(x'2,i,z_k\right)}\prod_{j=1}^{d(y_k)}\left(g\left(y_k,j\right)-i\right)\right)+{f\left(x'2,i,0\right)}\prod_{j=1}^{d(y_{r(y)})}\left(g\left(y_{r(y)},j\right)-i\right) \Longleftrightarrow$$
$$\Longleftrightarrow 
\left(1-i\right){f\left(x'2,i,z\right)}\frac{\prod_{j=1}^{d(y)}\left(g\left(y,j\right)-i\right)}{1-i}=$$
$$=\sum_{k = 1}^{r(y)-1} \left(\left(1-i\right){f\left(x'2,i,z_k\right)}\frac{\prod_{j=1}^{d(y_k)}\left(g\left(y_k,j\right)-i\right)}{1-i}\right)+{f\left(x'2,i,0\right)}\prod_{j=1}^{d(y_{r(y)})}\left(g\left(y_{r(y)},j\right)-i\right)\Longleftrightarrow$$
$$\Longleftrightarrow \text{(по Утверждению \ref{211})} \Longleftrightarrow$$
$$\Longleftrightarrow 
{f\left(x'11,i,z+1\right)}\prod_{j=1}^{d(y'')}\left(g\left(y'',j\right)-i\right)=$$
$$=\sum_{k = 1}^{r(y'')-1} \left({f\left(x'11,i,z_k+1\right)}\prod_{j=1}^{d(y''_k)}\left(g\left(y''_k,j\right)-i\right)\right)+{f\left(x'11,i,1\right)}\prod_{j=1}^{d(y''_{r(y)})}\left(g\left(y''_{r(y)},j\right)-i\right).$$
Я утверждаю, что это условие Леммы для $x'11$ и $y'11=y''$. Почему это верно? Пусть $z'=h(x'11,y'')$ и $z'_k=h(x'11,y''_k)$. Ясно, что номер $y$ имеет вид
$$\underbrace{2...2}_{d(y)-1}2,$$
в то время как номер $y''$ имеет вид $$\underbrace{2...2}_{d(y)-1}11,$$
поэтому $z'=h(x'11,y'')=h(x'2,y)+1=z+1$. Также помним, что $\forall k \in \{1,...,r(y)-1\}$ $\exists y'''_k: y'''_k2=y_k \; \& \; y'''_k11=y''_k$, то есть опять же $z'_k=h(x'11,y''_k)=h(x'11,y'''_k11)=h(x',y'''_k)+2=h(x'2,y'''_k2)+1=h(x'2,y_k)+1=z_k+1$. И наконец, номер $y''_{r(y)}$ заканчивается на $1$, поэтому $z'_{r(y)}=0$.

Таким образом, это и есть условие Леммы для $x'11$ и $y$. А этот вариант уже был разобран в случае $\ref{3}^\circ$.

\end{enumerate}

Несложно понять, что все случаи разобраны. Осталось проверить \underline{\textbf{Базу}}. Приступим к её проверке.

\underline{\textbf{База}}: $|x|=0 \Longleftrightarrow x=\varepsilon$.

Хотим доказать, что
$${f\left(\varepsilon,i,z\right)}\prod_{j=1}^{d(y)}\left(g\left(y,j\right)-i\right)=\sum_{k = 1}^{r(y)} \left({f\left(\varepsilon,i,z_k\right)}\prod_{j=1}^{d(y_k)}\left(g\left(y_k,j\right)-i\right)\right).$$ 

Мы знаем, что при $x=\varepsilon$ единственным $f(x,y,z)$ является $f(\varepsilon,0,0)=1$, поэтому нам надо доказать, что 
$$\prod_{j=1}^{d(y)}g\left(y,j\right)=\sum_{k = 1}^{r(y)} \prod_{j=1}^{d(y_k)}g\left(y_k,j\right).$$ 

Если номер $y$ содержит хотя бы одну единицу, то это Следствие \ref{main}. Поэтому нам надо только рассмотреть случай, при котором номер $y$ не содержит единиц.

В данном случае мы хотим доказать, что 
$$\prod_{j=1}^{d(y)}\left( 2j-1\right)=\sum_{k = 1}^{d(y)}\left(\prod_{j=1}^{d(y)-k}\left(2j-1\right)\prod_{j=d(y)-k+1}^{d(y)-1}\left(2j\right)\right)\Longleftrightarrow$$
$$\Longleftrightarrow (d:=j-1,e:=k-1)\Longleftrightarrow$$
$$\Longleftrightarrow \prod_{d=0}^{d(y)-1}\left( 2d+1\right)=\sum_{e = 0}^{d(y)-1}\left(\prod_{d=0}^{d(y)-e-2}\left(2d+1\right)\prod_{d=d(y)-e-1}^{d(y)-2}\left(2d+2\right)\right)$$

Замечаем, что по Лемме \ref{dim} при $a=1$ и $b=c=d(y)-1$
$$\prod_{d=0}^{d(y)-1}\left(2d+1\right)-\sum_{e = 0}^{d(y)-1}\left(\prod_{d=0}^{d(y)-e-2}\left(2d+1\right)\prod_{d=d(y)-e-1}^{d(y)-2}\left(2d+2\right)\right)=\prod_{d=0}^{-1}\left(2d+1\right)\prod_{d=-1}^{d(y)-2}\left(2d+2\right)=0.$$

Что и требовалось.

\underline{\textbf{База}} доказана.

Лемма доказана.
\end{proof}

\subsection{Вторая ключевая лемма и завершение доказательства теоремы}

\begin{Lemma} \label{2key}
Пусть $x,y \in \mathbb{YF}:$ $|y|=|x|$. Тогда 
$$d(x,y)=o(x,y).$$
\end{Lemma}
\begin{proof}
Ясно, что мы хотим доказать, что $\forall x\in\mathbb{YF} \; o(x,x)=1$, а также, что если $x,y \in\mathbb{YF}$: $|y|=|x|$, причём $x \ne y$, то $o(x,y)=0$. Пусть $z=h(x,y)$. 

Пусть $x=\varepsilon$. 

Тогда $|x|=0$. Единственный $y: |y|=0$ -- это $y=\varepsilon$. $d(\varepsilon,\varepsilon)=1$, а
$$o(\varepsilon,\varepsilon)=\sum_{i=0}^{|\varepsilon|}\left( {f\left(\varepsilon,0,0\right)}\prod_{j=1}^{0}\left(g\left(\varepsilon,j\right)-i\right)\right)=1.$$

То есть при $x=\varepsilon$ лемма доказана.

Теперь пусть $|x| \ge 1$.

Давайте доказывать Лемму по индукции по $|x|$.

\underline{\textbf{База}}: $|x|=1 \Longleftrightarrow x=1$.

Знаем, что при $x=1$ единственными $f(x,y,z)$ являются $f(1,0,0)=1$, $f(1,1,0)=-1$, $f(1,0,1)=1$, $f(1,1,1)=0$.

Единственным $y: |y|=1$ является $y=1$. $d(1,1)=1$, а
$$o(1,1)=\sum_{i=0}^{|1|}\left( {f\left(1,i,1\right)}\prod_{j=1}^{0}\left(g\left(1,j\right)-i\right)\right)=f(1,0,1)+f(1,1,1)=1.$$ 
\underline{\textbf{Переход}} к $|x|\ge2$.

Рассмотрим пять случаев:
\renewcommand{\labelenumi}{\arabic{enumi}$^\circ$}
\begin{enumerate}
    \item \label{sl1} Пусть $\exists x',y': x=x'1,y=y'1$.

Очевидно, в данном случае $z \ge 1$ и $h(x',y')=h(x'1,y'1)-1=h(x,y)$. Посчитаем, пользуясь рекурсивной формулой для $f$:
$$o(x,y)=\sum_{i=0}^{|x'1|}\left( {f\left(x'1,i,z\right)}\prod_{j=1}^{d(y'1)}\left(g\left(y'1,j\right)-i\right)\right)=$$
$$=\sum_{i=0}^{|x'1|}\left(\left( {f\left(x'1,i,0\right)+f\left(x',i-1,z-1\right)}\right)\prod_{j=1}^{d(y'1)}\left(g\left(y'1,j\right)-i\right)\right)=$$
$$=\sum_{i=0}^{|x|}\left( {f\left(x,i,0\right)}\prod_{j=1}^{d(y)}\left(g\left(y,j\right)-i\right)\right)+\sum_{i=1}^{|x'1|}\left( {f\left(x',i-1,z-1\right)}\prod_{j=1}^{d(y')}\left(g\left(y',j\right)-(i-1)\right)\right)=$$
$$=\sum_{i=0}^{|x|}\left( {f\left(x,i,0\right)}\prod_{j=1}^{d(y)}\left(g\left(y,j\right)-i\right)\right)+\sum_{i=0}^{|x'|}\left( {f\left(x',i,z-1\right)}\prod_{j=1}^{d(y')}\left(g\left(y',j\right)-i\right)\right)=$$
$$=\sum_{i=0}^{|x|}\left( {f\left(x,i,0\right)}\prod_{j=1}^{d(y)}\left(g\left(y,j\right)-i\right)\right)+o(x',y')=$$
$$=\text{(По предположению индукции)}=$$
$$=\sum_{i=0}^{|x|}\left( {f\left(x,i,0\right)}\prod_{j=1}^{d(y)}\left(g\left(y,j\right)-i\right)\right)+d(x',y').$$

Вспомним, что если $ x,y \in\mathbb{YF}$ и $|y|=|x|$, то $d(x,y)=1$, если $x=y$ и $d(x,y)=0$, если $x\ne y$. Ясно, что в данном случае $|x|=|y|\Longleftrightarrow|x'1|=|y'1|\Longleftrightarrow|x'|=|y'|$, а поэтому $d(x,y)=d(x'1,y'1)=d(x',y')$.

Таким образом, чтобы доказать, что в данном случае $d(x,y)=o(x,y)$, нужно доказать, что $\forall x,y\in\mathbb{YF}:$ $|y|=|x|$ 
$$o'(x,y):=\sum_{i=0}^{|x|}\left( {f\left(x,i,0\right)}\prod_{j=1}^{d(y)}\left(g\left(y,j\right)-i\right)\right)=0.$$

\item \label{sl2} Пусть $\exists x',y': x=x'1, \;y=y'2$.

Ясно, что в данном случае $z=0$, то есть снова нужно доказать, что
$$o'(x,y):=\sum_{i=0}^{|x|}\left( {f\left(x,i,0\right)}\prod_{j=1}^{d(y)}\left(g\left(y,j\right)-i\right)\right)=0.$$

Давайте усилим утверждение, необходимое для доказательства случаев \ref{sl1}$^\circ$ и \ref{sl2}$^\circ$, после чего докажем его по индукции.

\begin{Lemma} \label{zhopa}
Если $x,y \in \mathbb{YF}$, $|x| \ge |y|$ и при этом $\exists x': x=x'1$, то 
$$o'(x,y):=\sum_{i=0}^{|x|}\left( {f\left(x,i,0\right)}\prod_{j=1}^{d(y)}\left(g\left(y,j\right)-i\right)\right)=0.$$
\end{Lemma}
\begin{proof}

Давайте докажем Лемму по индукции по $|y|$.

\underline{\textbf{База}}: $|y|=0 \Longleftrightarrow y=\varepsilon$.

Заметим, что
$$o'(x,\varepsilon)=\sum_{i=0}^{|x|}\left( {f\left(x,i,0\right)}\prod_{j=1}^{d(\varepsilon)}\left(g\left(\varepsilon,j\right)-i\right)\right)=\sum_{i=0}^{|x|}\left( {f\left(x,i,0\right)}\prod_{j=1}^{0}\left(g\left(\varepsilon,j\right)-i\right)\right)=\sum_{i=0}^{|x|} {f\left(x,i,0\right)}.$$

Таким образом, мы хотим доказать, что если $x \in \mathbb{YF}$, и $\exists x': x=x'1$, то  $$\sum_{i=0}^{|x|} {f\left(x,i,0\right)}=0.$$
Давайте усилим это утверждение и докажем его по индукции.

\begin{Prop} \label{baza}
Если $x \in \mathbb{YF}$, и $x\ne\varepsilon$, то  $$\sum_{i=0}^{|x|} {f\left(x,i,0\right)}=0.$$
\end{Prop}
\begin{proof}

Докажем Утверждение по индукции по $|x|$.

\underline{\textbf{База}}: проверим Утверждение для $x=1$ и $x=2$.

Посчитаем, пользуясь определением функции $f$:
$$\sum_{i=0}^{|1|} {f\left(1,i,0\right)}=1+(-1)=0;$$
$$\sum_{i=0}^{|2|} {f\left(2,i,0\right)}=\frac{1}{2}+0+\left(-\frac{1}{2}\right)=0.$$
\underline{\textbf{База}} доказана.

\underline{\textbf{Переход}} к $x: \; |x|\ge2$ и $x\ne 2$.

Рассмотрим два случая:
\renewcommand{\labelenumii}{\roman{enumii}$^\circ$}
\renewcommand{\labelenumiii}{\roman{enumii}.\roman{enumiii}$^\circ$}
\begin{enumerate}
    \item Пусть $\exists x':x=\alpha_0x'1$, при $\alpha_0\in\{1,2\}$.

Воспользуемся предположением при $x'1$ и при $\alpha_0x'$, после чего посчитаем, пользуясь формулой для $f$:
$$o'(\alpha_0x',y)=0 \; \& \; o'(x'1,y)=0 \Longleftrightarrow$$
$$\Longleftrightarrow \sum_{i=0}^{|\alpha_0x'|}{f\left(\alpha_0x',i,0\right)}=0 \; \& \; \sum_{i=0}^{|x'1|}{f\left(x'1,i,0\right)}=0 \Longrightarrow$$
$$\Longrightarrow\text{(По Утверждению \ref{z0})}\Longrightarrow$$
$$\Longrightarrow \sum_{i=0}^{|x|-1}{\left(\left(-i-1\right)f\left(x,i+1,0\right)\right)}=0 \; \& \;\sum_{i=0}^{|x|-\alpha_0}{\left(\left(|x|-i\right)f\left(x,i,0\right)\right)}=0 \Longleftrightarrow$$
$$\Longleftrightarrow \sum_{i=1}^{|x|}{\left(\left(-i\right)f\left(x,i,0\right)\right)}=0 \; \& \;\sum_{i=0}^{|x|-\alpha_0}{\left(\left(|x|-i\right)f\left(x,i,0\right)\right)}=0\Longrightarrow$$
$$\Longrightarrow \sum_{i=0}^{|x|-\alpha_0}{\left(|x|f\left(x,i,0\right)\right)}+   \sum_{i=|x|-\alpha_0+1}^{|x|}{\left(if\left(x,i,0\right)\right)}=0.$$

Рассмотрим два подслучая:
\begin{enumerate}
    \item Пусть $\alpha_0=1$.
    
    В данном случае равенство принимает вид $$\sum_{i=0}^{|x|-1}{\left(|x|f\left(x,i,0\right)\right)}+   \sum_{i=|x|}^{|x|}{\left(i f\left(x,i,0\right)\right)}=0 \Longleftrightarrow \sum_{i=0}^{|x|}{\left(|x|f\left(x,i,0\right)\right)}=0 \Longleftrightarrow\sum_{i=0}^{|x|}{f\left(x,i,0\right)}=0.$$
    Что и требовалось.
    
    \item Пусть $\alpha_0=2$.
    
    В данном случае равенство принимает вид $$\sum_{i=0}^{|x|-2}{\left(|x|f\left(x,i,0\right)\right)}+   \sum_{i=|x|-1}^{|x|}{\left(if\left(x,i,0\right)\right)}=0.$$
    
    Заметим, что если $x$ начинается на двойку, то из определения функции $f$ ясно, что $f(x,|x|-1,0)=0$. Поэтому необходимое равенство равносильно следующему равенству:
    $$\sum_{i=0}^{|x|-2}{\left(|x|f\left(x,i,0\right)\right)}+   \sum_{i=|x|}^{|x|}{\left(if\left(x,i,0\right)\right)}=0 \Longleftrightarrow \sum_{i=0}^{|x|}{\left(|x|f\left(x,i,0\right)\right)}=0 \Longleftrightarrow \sum_{i=0}^{|x|}{f\left(x,i,0\right)}=0.$$
    Что и требовалось.
\end{enumerate}

\item $Пусть \exists x':x=x'2$.

Воспользуемся предположением при $x'1$ и прошлым случаем при $x'11$, после чего посчитаем, пользуясь формулой для $f$:
$$o'(x'11,y)=0 \; \& \; o'(x'1,y)=0 \Longleftrightarrow$$
$$\Longleftrightarrow \sum_{i=0}^{|x'11|}{f\left(x'11,i,0\right)}=0 \; \& \; \sum_{i=0}^{|x'1|}{f\left(x'1,i,0\right)}=0 \Longleftrightarrow$$
$$\Longleftrightarrow \sum_{i=0}^{|x'11|}{f\left(x'11,i,0\right)}=0 \; \& \; \sum_{i=0}^{|x'1|}{((-i-1)f\left(x'11,i+1,0\right))}=0 \Longleftrightarrow$$
$$\Longleftrightarrow \sum_{i=0}^{|x|}{f\left(x'11,i,0\right)}=0 \; \& \; \sum_{i=1}^{|x|}{((-i)f\left(x'11,i,0\right))}=0 \Longrightarrow$$
$$\Longrightarrow f\left(x'11,0,0\right)+ \sum_{i=1}^{|x|}{((1-i)f\left(x'11,i,0\right))}=0 \Longleftrightarrow$$
$$\Longleftrightarrow f\left(x'11,0,0\right)+ \sum_{i=2}^{|x|}{((1-i)f\left(x'11,i,0\right))}=0 \Longleftrightarrow $$
$$\Longleftrightarrow\text{(По Утверждению \ref{z0})}\Longleftrightarrow$$
$$\Longleftrightarrow f\left(x'2,0,0\right)+ \sum_{i=2}^{|x|}{f\left(x'2,i,0\right)}=0.$$
Заметим, что если $x$ заканчивается на двойку, то из определения функции $f$ ясно, что $f(x,1,0)=0$. Поэтому необходимое равенство равносильно следующему равенству:
$$ \sum_{i=0}^{|x|}{f\left(x'2,i,0\right)}=0\Longleftrightarrow\sum_{i=0}^{|x|}{f\left(x,i,0\right)}=0.$$
Что и требовалось.

\end{enumerate}
\end{proof}

Итак, ясно, что Утверждение \ref{baza} доказывает \underline{\textbf{Базу}}.

Теперь будем доказывать \underline{\textbf{База}}: к $|y|\ge 1$.

Рассмотрим два случая:
\renewcommand{\labelenumii}{\roman{enumii}$^\circ$}
\renewcommand{\labelenumiii}{\roman{enumii}.\roman{enumiii}$^\circ$}
\begin{enumerate}
    \item Пусть номер $y$ содержит хотя бы одну единицу.

Давайте докажем, что в данном случае $o'(x,y)=o'(x,y_1)+...+o'(x,y_{r(y)})$. Посчитаем:
$$o'(x,y)=o'(x,y_1)+...+o'(x,y_{r(y)}) \Longleftrightarrow$$
$$\Longleftrightarrow \sum_{i=0}^{|x|}\left( {f\left(x,i,0\right)}\prod_{j=1}^{d(y)}\left(g\left(y,j\right)-i\right)\right)=\sum_{k=1}^{r(y)}\sum_{i=0}^{|x|}\left( {f\left(x,i,0\right)}\prod_{j=1}^{d(y_k)}\left(g\left(y_k,j\right)-i\right)\right) \Longleftarrow$$
$$\Longleftarrow \sum_{i=0}^{|x|}\prod_{j=1}^{d(y)}\left(g\left(y,j\right)-i\right)=\sum_{k=1}^{r(y)}\sum_{i=0}^{|x|} \prod_{j=1}^{d(y_k)}\left(g\left(y_k,j\right)-i\right) \Longleftrightarrow $$
$$\Longleftrightarrow \sum_{i=0}^{|x|}\prod_{j=1}^{d(y)}\left(g\left(y,j\right)-i\right)=\sum_{i=0}^{|x|}\sum_{k=1}^{r(y)} \prod_{j=1}^{d(y_k)}\left(g\left(y_k,j\right)-i\right) \Longleftarrow $$
$$ \Longleftarrow \forall i \in \{0,...,|x|\} \; \prod_{j=1}^{d(y)}\left(g\left(y,j\right)-i\right)=\sum_{k=1}^{r(y)} \prod_{j=1}^{d(y_k)}\left(g\left(y_k,j\right)-i\right).$$
А это Следствие \ref{main}.

Таким образом, в данном случае $$o'(x,y)=o'(x,y_1)+...+o'(x,y_{r(y)})= \text{(По предположению индукции)}=0.$$

В данном случае переход доказан.

\item Пусть номер $y$ не содержит ни одной единицы.

Помним, что мы доказываем Лемму для $|x|\ge |y|$.

Пусть $x=\alpha_0x''$.

Воспользуемся предположением индукции при $x$ и $y'$ и при $x''$ и $y'$ ($y'=|y|-2\le |x|-2 \le |x|-\alpha_0 \le |x''| \le |x|$, поэтому мы можем воспользоваться предположением) и посчитаем:
$$o'(x,y')=0 \; \& \; o'(x'',y')=0 \Longleftrightarrow $$
$$\Longleftrightarrow \sum_{i=0}^{|x|}\left( {f\left(x,i,0\right)}\prod_{j=1}^{d(y')}\left(2j-1-i\right)\right)=0\; \& \; \sum_{i=0}^{|x''|}\left( {f\left(x'',i,0\right)}\prod_{j=1}^{d(y')}\left(2j-1-i\right)\right)=0 \Longleftrightarrow$$
$$\Longleftrightarrow \sum_{i=0}^{|x|}\left( {f\left(x,i,0\right)}\prod_{j=1}^{d(y')}\left(2j-1-i\right)\right)=0\; \& \; \sum_{i=0}^{|x|-\alpha_0}\left( {\left(|x|-i\right)f\left(x,i,0\right)}\prod_{j=1}^{d(y')}\left(2j-1-i\right)\right)=0 \Longrightarrow $$
$$\Longrightarrow \sum_{i=0}^{|x|-\alpha_0}\left(\left( {f\left(x,i,0\right)}\prod_{j=1}^{d(y')}\left(2j-1-i\right)\right)\left( (|x|-i)+(|y|-|x|-1)  \right)\right)+$$
$$+(|y|-|x|-1)\sum_{i=|x|-\alpha_0+1}^{|x|}\left( {f\left(x,i,0\right)}\prod_{j=1}^{d(y')}\left(2j-1-i\right)\right)=0 \Longleftrightarrow$$
$$\Longleftrightarrow \sum_{i=0}^{|x|-\alpha_0}\left(\left( {f\left(x,i,0\right)}\prod_{j=1}^{d(y')}\left(2j-1-i\right)\right)\left( |y|-1-i  \right)\right)+$$
$$+(|y|-|x|-1)\sum_{i=|x|-\alpha_0+1}^{|x|}\left( {f\left(x,i,0\right)}\prod_{j=1}^{d(y')}\left(2j-1-i\right)\right)=0 \Longleftrightarrow$$
$$\Longleftrightarrow \sum_{i=0}^{|x|-\alpha_0}\left( {f\left(x,i,0\right)}\prod_{j=1}^{d(y)}\left(2j-1-i\right)\right)+(|y|-|x|-1)\sum_{i=|x|-\alpha_0+1}^{|x|}\left( {f\left(x,i,0\right)}\prod_{j=1}^{d(y')}\left(2j-1-i\right)\right)=0.$$

Рассмотрим два подслучая:
\begin{enumerate}
    \item Пусть $\alpha_0=1$.
    
    В данном случае равенство принимает вид
    $$\sum_{i=0}^{|x|-1}\left( {f\left(x,i,0\right)}\prod_{j=1}^{d(y)}\left(2j-1-i\right)\right)+(|y|-|x|-1)\sum_{i=|x|}^{|x|}\left( {f\left(x,i,0\right)}\prod_{j=1}^{d(y')}\left(2j-1-i\right)\right)=0\Longleftrightarrow$$
    $$\Longleftrightarrow\sum_{i=0}^{|x|}\left( {f\left(x,i,0\right)}\prod_{j=1}^{d(y)}\left(2j-1-i\right)\right)=0\Longleftrightarrow o'(x,y)=0.$$
    Что и требовалось.
    
    \item Пусть $\alpha_0=2$.
    
    В данном случае равенство принимает вид
    $$\sum_{i=0}^{|x|-2}\left( {f\left(x,i,0\right)}\prod_{j=1}^{d(y)}\left(2j-1-i\right)\right)+(|y|-|x|-1)\sum_{i=|x|-1}^{|x|}\left( {f\left(x,i,0\right)}\prod_{j=1}^{d(y')}\left(2j-1-i\right)\right)=0.$$
    
    Заметим, что в данном случае $x$ начинается на двойку, поэтому из определения функции $f$ ясно, что $f(x,|x|-1,0)=0$. Таким образом, равенство равносильно следующему равенству:
    $$\sum_{i=0}^{|x|-1}\left( {f\left(x,i,0\right)}\prod_{j=1}^{d(y)}\left(2j-1-i\right)\right)+(|y|-|x|-1)\sum_{i=|x|}^{|x|}\left( {f\left(x,i,0\right)}\prod_{j=1}^{d(y')}\left(2j-1-i\right)\right)=0\Longleftrightarrow$$
    $$\Longleftrightarrow\sum_{i=0}^{|x|}\left( {f\left(x,i,0\right)}\prod_{j=1}^{d(y)}\left(2j-1-i\right)\right)=o'(x,y)$$
    Что и требовалось
\end{enumerate}

\end{enumerate}

Все случаи разобраны.

Лемма доказана.

\end{proof}

Ясно, что из Леммы \ref{zhopa} мгновенно следуют случаи \ref{sl1}$^\circ$ и \ref{sl2}$^\circ$.

\item Пусть $\exists x',y': x=x'2, \; y=y'2$.

Ясно, что из того, что и номер $x$, и номер $y$ заканчивается на двойку, следует, что $z \ge 1$. Также вспомним, что если номер $x$ заканчивается на двойку, то $f(x,1,0)=0$. Посчитаем, пользуясь рекурсивной формулой для $f$:
$$o(x,y)=\sum_{i=0}^{|x'2|}\left( {f\left(x'2,i,z\right)}\prod_{j=1}^{d(y'2)}\left(g\left(y'2,j\right)-i\right)\right)=$$
$$=\sum_{i \in \{0,...,|x2|\} \textbackslash 1}\left( {f\left(x'2,i,z\right)}\prod_{j=1}^{d(y'2)}\left(g\left(y'2,j\right)-i\right)\right)=$$
$$=\sum_{i \in \{0,...,|x11|\} \textbackslash 1}^{|x'11|}\left( {\frac{f\left(x'11,i,z+1\right)}{1-i}}\left(\prod_{j=1}^{d(y'11)}\left(g\left(y'11,j\right)-i\right)\right)\left(1-i\right)\right)=$$
$$=\sum_{i \in \{0,...,|x11|\} \textbackslash 1}\left( {{f\left(x'11,i,z+1\right)}}\prod_{j=1}^{d(y'11)}\left(g\left(y'11,j\right)-i\right)\right).$$

По Утверждению \ref{y01} $f(x11,1,z+1)=0$, а значит выражение можно переписать в следующем виде:
$$=\sum_{i=0}^{|x'11|}\left( {{f\left(x'11,i,z+1\right)}}\prod_{j=1}^{d(y'11)}\left(g\left(y'11,j\right)-i\right)\right).$$

Ясно, что $h(x'11,y'11)=h(x',y')+2=h(x'2,y'2)+1=h(x,y)+1=z+1$, а значит данное выражение равняется 
$$o(x'11,y'11)= (\text{По случаю \ref{sl1}}^\circ)=d(x'11,y'11).$$

Ясно, что $\forall x,y\in \mathbb{YF}$: $|x|=|y|$ $d(x11,y11)=d(x2,y2)$, поэтому выражение принимает вид
$$d(x'2,y'2)=d(x,y).$$
Что и требовалось.

\item Пусть $\exists x',y': x=x'2,y=y'21$.

Заметим, что в данном случае номера $x$ и $y$ заканчиваются на разные цифры, а поэтому $z = 0$. Посчитаем:
$$o(x,y)=\sum_{i=0}^{|x'2|}\left( {f\left(x'2,i,0\right)}\prod_{j=1}^{d(y)}\left(g\left(y,j\right)-i\right)\right)=$$
$$=\text{(По Утверждению \ref{211})}=$$
$$=\sum_{i=0}^{|x'11|}\left( {f\left(x'11,i,1\right)}\prod_{j=1}^{d(y)}\left(g\left(y,j\right)-i\right)\right).$$

Ясно, что $h(x'11,y'21)=1$, поэтому данное выражение равняется
$$o(x'11,y)=(\text{По случаю \ref{sl1}}^\circ)=d(x'11,y)=0,$$
так как $|x'11|=|y|$ и при этом $x'11\ne y$.

Мы доказали, что $o(x,y)=0$. $d(x,y)=0$, так как $|x|=|y|$, и при этом $x\ne y$.

Что и требовалось.

\item Пусть $\exists x',y': x=x'2, \; y=y'11$.

Заметим, что в данном случае номера $x$ и $y$ заканчиваются на разные цифры, а поэтому $z = 0$. Также вспомним, что если $x$ заканчивается на двойку, то $f(x,1,0)=0$. Посчитаем:
$$o(x,y)=\sum_{i=0}^{|x'2|}\left( {f\left(x'2,i,0\right)}\prod_{j=1}^{d(y'11)}\left(g\left(y'11,j\right)-i\right)\right)=$$
$$=\sum_{i\in \{0,...,|x'2|\} \textbackslash 1}\left( {f\left(x'2,i,0\right)}\prod_{j=1}^{d(y'11)}\left(g\left(y'11,j\right)-i\right)\right)=$$
$$=\text{(По Утверждению \ref{z0})}=$$
$$=\sum_{i\in \{0,...,|x'11|\} \textbackslash 1}\left( {\left(1-i\right)f\left(x'11,i,0\right)}\frac{\prod_{j=1}^{d(y'2)}\left(g\left(y'2,j\right)-i\right)}{1-i}\right)=$$
$$=\sum_{i\in \{0,...,|x'11|\} \textbackslash 1}\left( {f\left(x'11,i,0\right)}{\prod_{j=1}^{d(y'2)}\left(g\left(y'2,j\right)-i\right)}\right).$$

Заметим, что  $g(y'2,1)=1$, а значит выражение можно переписать в следующем виде:
$$=\sum_{i=0}^{|x'11|}\left( {{f\left(x'11,i,0\right)}}\prod_{j=1}^{d(y'2)}\left(g\left(y'2,j\right)-i\right)\right).$$

Ясно, что $h(x'11,y'2)=0$, а значит данное выражение равняется 
$$o(x'11,y'2)= (\text{По случаю \ref{sl2}}^\circ)=d(x'11,y'2),$$
так как, $|x'11|=|y'2|$, и при этом $x'11\ne y'2$.

Мы доказали, что $o(x,y)=0$. $d(x,y)=0$, так как $|x|=|y|$, и при этом $x\ne y$.

Что и требовалось.

\end{enumerate}

Все случаи разобраны.

Лемма доказана.
\end{proof}

Давайте доказывать теорему.

Вспомним, что $|y|\ge |x|$, зафиксируем $x$ и будем доказывать теорему по индукции по $|y|$.

\underline{\textbf{База}}: $|y|=|x|$.

$d(x,y)=o(x,y)$ по Лемме \ref{2key}.

\underline{\textbf{Переход}}: к $y:|y|>|x|$.
$$d(x,y)=\sum_{k=0}^{r(y)}d(x,y_k)=\text{(По предположению)}=\sum_{k=0}^{r(y)}o(x,y_k)=$$
$$=\sum_{k=0}^{r(y)}\left(\sum_{i=0}^{|x|}\left({f\left(x,i,z_k\right)}\prod_{j=1}^{d(y_k)}\left(g\left(y_k,j\right)-i\right)\right)\right)=$$
$$\sum_{i=0}^{|x|}\left(\sum_{k=0}^{r(y)}\left({f\left(x,i,z_k\right)}\prod_{j=1}^{d(y_k)}\left(g\left(y_k,j\right)-i\right)\right)\right)=$$
$$=\text{(По лемме \ref{1key})}=$$
$$=\sum_{i=0}^{|x|}\left({f\left(x,i,z\right)}\prod_{j=1}^{d(y)}\left(g\left(y,j\right)-i\right)\right)=o\left(x,y\right).$$

Что и требовалось доказать.

Теорема полностью доказана.

\end{proof}

\begin{Col}
Пусть $x\in\mathbb{YF}$. Тогда
$$d(\varepsilon,x)=\prod_{j=1}^{d(y)} g(y,j).$$
\end{Col}
\begin{proof}
Воспользуемся Теоремой \ref{TH1} и посчитаем:
$$d(\varepsilon,x)=\sum_{i=0}^{|\varepsilon|}\left( {f\left(\varepsilon,i,h(\varepsilon,x)\right)}\prod_{j=1}^{d(x)}\left(g\left(x,j\right)-i\right)\right)= {f\left(\varepsilon,0,0\right)}\prod_{j=1}^{d(x)}\left(g\left(x,j\right)-0\right)=\prod_{j=1}^{d(y)} g(y,j).$$ 

\end{proof}

\subsection{Формула в явном виде}

Теперь мы хотим написать формулу для $f(x,y,z)$ не в рекурсивном, а в явном виде. В явном виде определены $f(x,y,z)$ при $z=0$. Давайте выразим $f(x,y,z)$ через различные $f(x',y',0)$.

Если $y=0$, то по Утверждению \ref{y01} $f(x,0,z)=f(x,0,0)$.

Пусть $y>0$.

\begin{Def}
При $x\in\mathbb{YF}$ первой функцией начала для $x$ назовём функцию
$$s(a)=s(x,a) \quad (a \in \{0,...,|x|\} ),$$
определённую следующим образом:

\begin{itemize}
    \item Если $\exists x',x''\in \mathbb{YF}: x=x'1x''$ и $|x''|=a$, то $s(a)=x'1$;
    \item Если $\exists x',x''\in \mathbb{YF}: x=x'2x''$ и $|x''|=a$, то $s(a)=x'11$;
    \item Если $\exists x',x''\in \mathbb{YF}: x=x'2x''$ и $|x''|=a-1$, то $s(a)=x'1$;
    \item $s(|x|)=\varepsilon$.
\end{itemize}

\end{Def}

\begin{Def}
При $x\in\mathbb{YF}$ второй функцией начала для $x$ назовём функцию
$$s'(a)=s'(x,a) \quad (dom \; a \subset \{0,...,|x|\}),$$
определённую следующим образом:

\begin{itemize}
    \item Если $\exists x',x''\in \mathbb{YF}: x'x''=x$ и $|x''|=a$, то $s'(a)=x'$;
    \item Если $\nexists x',x''\in \mathbb{YF}: x'x''=x$ и $|x''|=a$, то $s'(a)$ не определено.
\end{itemize}

\end{Def}


\begin{Ex}
$s$ и $s'$ при $x=122112$:

\begin{center}
\begin{tabular}{ | m{1cm} || m{2.5cm} | m{2.5cm} | } 
  \hline
 & $s(a,122112)$ & $s'(a,122112)$\\
  \hline \hline
$a=0$ & $1221111$ & $122112$\\
  \hline
$a=1$ & $122111$ & не определено\\ 
  \hline
$a=2$ & $12211$ & $12211$\\ 
  \hline
$a=3$ & $1221$ & $1221$\\ 
  \hline
$a=4$ & $1211$ & $122$\\ 
  \hline
$a=5$ & $121$ & не определено\\ 
  \hline
$a=6$ & $111$ & $12$\\ 
  \hline
$a=7$ & $11$ & не определено\\ 
  \hline
$a=8$ & $1$ & $1$\\ 
  \hline
$a=9$ & $\varepsilon$ & $\varepsilon$\\ 
  \hline
\end{tabular}
\end{center}
\end{Ex}

\begin{Alg} \label{algo}
Давайте применять рекурсивную формулу для $f$ много раз следующим образом: 

\renewcommand{\labelenumi}{\arabic{enumi}$^\circ$}

\begin{itemize}
    \item На каждом шаге мы имеем линейную комбинацию чисел вида $f(x'',y'',0)$ и числа вида $f(x',y',z')$, равную $f(x,y,z)$; при этом в данном случае мы говорим, что мы находимся в $x'\in \mathbb{YF}$ или в $f(x',y',z')$;
    \item Перед первым шагом мы имеем линейную комбинацию $f(x,y,z)=f(s(0),y,z)$ и находимся в $f(x,y,z)=f(s(0),y,z)$;
    \item На каждом шаге мы рассматриваем два варианта:
    \begin{enumerate}
        \item \label{type1} если находимся в $f(x'1,y',z')$ и линейная комбинация имеет вид 
        $$\sum{\frac{f(x'',y'',0)}{t''}}+\frac{f(x'1,y',z')}{t'},$$
        то мы переходим к $f(x',y'-1,z'-1)$ и линейной комбинации
        $$\sum{\frac{f(x'',y'',0)}{t''}}+\frac{f(x'1,y',0)}{t'''}+\frac{f(x',y'-1,z'-1)}{t'};$$
        \item \label{type2} если находимся в $f(x'2,y',z')$ при $y'>1$ и линейная комбинация имеет вид 
        $$\sum{\frac{f(x'',y'',0)}{t''}}+\frac{f(x'2,y',z')}{t'},$$
        то мы переходим к $f(x'11,y',z'+1)$ и линейной комбинации
        $$\sum{\frac{f(x'',y'',0)}{t''}}+\frac{f(x'11,y',z'+1)}{t'\cdot(1-y')}.$$
    \end{enumerate}
    \item Заканчиваем, когда приходим к числу вида $f(x',1,z')$ или к числу вида $f(x',y',0)$.
\end{itemize}
\end{Alg}

\begin{Zam}
При 
$$x=\underbrace{1...1}_{\beta_{d(x)}}2\underbrace{1...1}_{\beta_{d(x)-1}}2...2\underbrace{1...1}_{\beta_1}2\underbrace{1...1}_{\beta_0}$$
процесс, описанный в Алгоритме \ref{algo}, проходит следующим образом:
$$f(x,y,z)=f(s(0),y,z)=$$
$$=f(s(0),y,0)+f(s(1),y-1,z-1)=$$
$$=f(s(0),y,0)+f(s(1),y-1,0)+f(s(2),y-2,z-2)=$$
$$...$$
$$=f(s(0),y,0)+f(s(1),y-1,0)+...+f(s(\beta_0-1),y-(\beta_0-1),0)+f(s'(\beta_0),y-\beta_0,z-\beta_0)=$$
$$\left(\text{Тут мы приходим к } s'(\beta_0)=\underbrace{1...1}_{\beta_{d(x)}}2\underbrace{1...1}_{\beta_{d(x)-1}}2...2\underbrace{1...1}_{\beta_1}2\right)$$
$$=f(s(0),y,0)+f(s(1),y-1,0)+...+f(s(\beta_0-1),y-(\beta_0-1),0)+\frac{f(s(\beta_0),y-\beta_0,z-\beta_0+1)}{1-(y-\beta_0)}=$$
$$\left(\text{Тут мы переходим к } s(\beta_0)=\underbrace{1...1}_{\beta_{d(x)}}2\underbrace{1...1}_{\beta_{d(x)-1}}2...2\underbrace{1...1}_{\beta_1}11\right)$$
$$=f(s(0),y,0)+f(s(1),y-1,0)+...+f(s(\beta_0-1),y-(\beta_0-1),0)+$$$$+\frac{f(s(\beta_0),y-\beta_0,0)+f(s(\beta_0+1),y-\beta_0-1,z-\beta_0)}{1-(y-\beta_0)}=$$
$$\left(\text{Тут мы переходим к } s(\beta_0+1)=\underbrace{1...1}_{\beta_{d(x)}}2\underbrace{1...1}_{\beta_{d(x)-1}}2...2\underbrace{1...1}_{\beta_1}1\right)$$
$$=f(s(0),y,0)+f(s(1),y-1,0)+...+f(s(\beta_0-1),y-(\beta_0-1),0)+$$$$+\frac{f(s(\beta_0),y-\beta_0,0)+f(s(\beta_0+1),y-\beta_0-1,0)+f(s(\beta_0+2),y-\beta_0-2,z-\beta_0-1)}{1-(y-\beta_0)}$$
$$\left(\text{Тут мы переходим к } s(\beta_0+2)=\underbrace{1...1}_{\beta_{d(x)}}2\underbrace{1...1}_{\beta_{d(x)-1}}2...2\underbrace{1...1}_{\beta_1}\right).$$
После чего процесс продолжается в том же духе, пока мы не приходим либо к числу вида $f(x',1,z')$, либо к числу вида $f(x',y',0)$.
\end{Zam}

\renewcommand{\labelenumi}{\arabic{enumi}$)$}

Посмотрим, как устроен наш процесс.

\begin{Zam} \label{algozam}

$\;$

\begin{enumerate}
    
    \item Все числа, к которым мы приходим в процессе, описанным Алгоритмом \ref{algo}, принимают вид $f(s(a),y-a,z')$ или $f(s'(a),y-a,z')$.

    \item Переходы типа \ref{type2}$^\circ$ -- это переходы от $f(s(a),y-a,z')$ к $f(s'(a),y-a,z'+1)$ с делением на $1-(y-a)$, происходящие при $a=\beta_0$, $a=\beta_0+\beta_1+2$, $a=\beta_0+\beta_1+\beta_2+4$,..., то есть при $a=g(x,1)-1,g(x,2)-1,g(x,3)-1,...$. То есть при $j$-ом переходе типа \ref{type2}$^\circ$ деление происходит на $1-(y-(g(x,j)-1))=g(x,j)-y$.

    \item Переходы типа \ref{type1}$^\circ$ -- это переходы от $f(s(a),y-a,z')$ к $f(s(a+1),y-a-1,z'-1)$ (если $s(a)$ заканчивается на $11$) или от $f(s(a),y-a,z')$ к $f(s'(a+1),y-a-1,z'-1)$ (если $s(a)$ заканчивается на $21$).
\end{enumerate}

\end{Zam}

\begin{Nab} \label{nab1}
Пусть процесс, описанный Алгоритмом \ref{algo}, заканчивается переходом к числу вида $f(x',y',0)$.

Мы знаем (Замечание \ref{algozam}), что третья координата уменьшается на $1$ при переходе типа \ref{type1}$^\circ$ и увеличивается на $1$ при переходе типа \ref{type2}$^\circ$. То есть при последовательности из перехода типа \ref{type1}$^\circ$ и двух переходов типа \ref{type2}$^\circ$ (то есть при последовательности переходов $x'2\mapsto x'11\mapsto x'1\mapsto x'$) третья координата также уменьшается на 1. То есть процесс заканчивается после перехода к числу, равному $f(s'(a),y',0)$ при некотором $a$, причём заметим, что все числа, из которых состоит итоговая линейная комбинация, кроме, может быть, последнего, имеет вид $f(s(a),y',0)$. Из всего этого и Замечания \ref{algozam} следует, что в итоге линейная комбинация принимает следующий вид:
$$f(x,y,z)=\sum_{i=0}^{c-1}\frac{f(s(i),y-i,0)}{\prod_{j=1}^{d(i)}(g(x,j)-y)}+\frac{f(s'(c),y-c,0)}{\prod_{j=1}^{d(c)}(g(x,j)-y)}$$
при некоторых $c$, $d(i)$ и $d(c)$. Давайте найдём эти $c$, $d(i)$ и $d(c)$.

Заметим, что $d(i)$ -- это количество чисел вида $g(x,j)-1$, не превышающих $i$, то есть 
$$d(i)=\max_k\{g(x,k)-1\le i\},$$
а $d(c)$ - это количество чисел вида $g(x,j)-1$, строго меньших $c$, то есть
$$d(c)=\max_k\{g(x,k)-1 < c\}.$$

\begin{Def}
При $x\in \mathbb{YF}$ функцией суммы конца для $x$ назовём функцию
$$p(x,a)\quad (a\in \{0,...,\#x\}),$$
значение которой равно сумме $a$ последних цифр числа $x$.

То есть если мы рассмотрим представление $x=\alpha_1...\alpha_{\#x}$, где $\alpha_i \in \{1,2 \}$, то
$$p(x,a)=|\alpha_{\#x-a+1}\alpha_{\#x-a+2}...\alpha_{\#x}|.$$
\end{Def}

Ясно, что в данном случае $c$ -- это сумма последних $z$ цифр числа $x$, то есть 
$$c=p(x,z).$$
\end{Nab}

\begin{Nab} \label{nab2}
Пусть процесс, описанный Алгоритмом \ref{algo}, заканчивается переходом к числу вида $f(x',1,z')$ при $z' \ge 1$. 

Мы знаем, что вторая координата уменьшается на $1$ при переходе типа \ref{type1}$^\circ$ и не меняется при переходе типа \ref{type2}$^\circ$. То есть при последовательности из перехода типа \ref{type1}$^\circ$ и двух переходов типа \ref{type2}$^\circ$ (то есть при последовательности переходов $x'2\mapsto x'11\mapsto x'1\mapsto x'$) вторая координата уменьшается на 2. Таким образом, получается, что процесс заканчивается после перехода к числу вида $f(s(a),1,z')=0$ при некотором $a$ или к числу вида $f(s'(a),1,z')=0$ при некотором $a$ (число $f$, к которому мы в итоге приходим, равно нулю, так как $z'\ge 1$, а поэтому мы можем воспользоваться Утверждением \ref{y01}). Из всего этого и Замечания \ref{algozam} следует, что в итоге линейная комбинация принимает следующий вид:
$$f(x,y,z)=\sum_{i=0}^{c}\frac{f(s(i),y-i,0)}{\prod_{j=1}^{d(i)}(g(x,j)-y)}$$
при некоторых $c$ и $d(i)$. Давайте найдём эти $c$ и $d(i)$.

Заметим, что $d(i)$ -- это опять же количество чисел вида $g(x,j)-1$, не превосходящих $i$, то есть
$$d(i)=\max_k\{g(x,k)-1\le i\}.$$

Также заметим, что вторая координата уменьшается на $1$ в том и только в том случае, когда образовывается новое слагаемое $f$, уменьшается вторая координата до значения $1$, но при значении $1$ $f=0$, поэтому
$$c=y-2.$$

Деления на $0$ тут возникнуть не может, потому что оно может произойти только тогда, когда мы переходим от $f(s'(a),y-a,z')$ при $y-a=1$, то есть тогда, когда мы переходим от $f(s'(y-1),1,z')$. А если мы придём к $f(s'(y-1),1,z')$, то сразу же закончим процесс и не перейдём дальше, поэтому деления на $0$ не возникнет.
\end{Nab}

\begin{Nab} \label{nab3}
Давайте поймём, когда у нас процесс, описанный Алгоритмом \ref{algo}, заканчивается переходом к числу вида $f(x',y',0)$, и ответ из Наблюдения \ref{nab1} является верным, а когда он заканчивается переходом к числу вида $f(x',1,z')$ при $z'\ge 1$, и ответ из Наблюдения \ref{nab2} является верным.

Ясно, что мы действовали по одному алгоритму, поэтому ответ, который соответствует более раннему шагу алгоритма и является подходящим. То есть:
\begin{itemize}
    \item Если $p(x,z)<y-2$, то, очевидно, что это случай из Наблюдения \ref{nab1};

    \item Если $p(x,z)>y-2$, то, очевидно, что это случай из Наблюдения \ref{nab2};

    \item Если $p(x,z)=y-2$, то это случай из Наблюдения \ref{nab1}, так как в данном случае у нас одинаковое число слагаемых, но в первом случае у нас последнее слагаемое зависит от $s'$, а во втором случае последнее слагаемое зависит от $s$; в первом случае это более ранний шаг алгоритма.
\end{itemize}

\end{Nab}

\begin{Zam} \label{yavf}
$\;$
\begin{enumerate}
    \item У нас есть явная формула для $f(x,y,z)$;
    \item В ней $O(|x|)$ слагаемых.
\end{enumerate}
\end{Zam}
\begin{proof}
$\;$
\begin{enumerate}
    \item Очевидно из Наблюдений \ref{nab1}, \ref{nab2} и \ref{nab3}.
    \item Из Наблюдений \ref{nab1}, \ref{nab2} и \ref{nab3} ясно, что количество слагаемых в нашей формуле равняется 
    $$1+\min\{p(x,z),y-2\}.$$
    Несложно заметить, что $1+\min\{p(x,z),y-2\} \le y-1 \le |x|-1=O(|x|)$.
\end{enumerate}

\end{proof}

\begin{Zam}
$\;$
\begin{enumerate}
    \item У нас есть явная формула для $d(x,y)$;
    \item В ней $O(|x|^2)$ слагаемых.
\end{enumerate}
\end{Zam}
\begin{proof}
Из Замечания \ref{yavf} и того, что 
$$d(x,y)=\sum_{i=0}^{|x|}\left( {f\left(x,i,z\right)}\prod_{j=1}^{d(y)}\left(g\left(y,j\right)-i\right)\right),$$
мгновенно следует первая часть, а также то, что количество слагаемых в $d(x,y)$ равно $(|x|+1)O(|x|)=O(|x|^2)$.
\end{proof}

\newpage

\section{Количество $S$-путей между двумя вершинами в графе Юнга-Фибоначчи}

\subsection{Подготовка к формулировке двух теорем и их доказательству}

Теперь мы хотим обобщить Теорему \ref{TH1}. Для этого введём несколько новых определений и обозначений.

\begin{Def}
Пусть $X,Y\in \mathbb{N}_0$, $Y\ge X$. Тогда последовательность $\Delta=\{\Delta_i: i \in \{0,1,...,k\}\}$ при некотором $k \in \mathbb{N}_0$, назовём $(Y;X)$-траекторией, если выполняются следующие условия:
\begin{enumerate}
    \item $ \forall i \in \{1,...,k\} \quad \Delta_i-\Delta_{i-1} = \pm 1$;
    \item$ \forall i \in \{0,1,...,k\} \quad \Delta_i \ge  0 $;
    \item$ \Delta_0=Y $;
    \item$ \Delta_k=X$.
\end{enumerate}
\end{Def}

\begin{Oboz}
$\;$
\begin{itemize}
    \item Пусть $X,Y\in \mathbb{N}_0$. Тогда множество $(Y; X)$-траекторий обозначим за $\Omega(X,Y)$;
    \item Множество всех $(Y;X)$-траекторий при каких-то $X,Y\in \mathbb{N}_0$ обозначим за $\Omega$. 
\end{itemize}
\end{Oboz}

\begin{Def}

Пусть $\Delta=\{\Delta_i: i \in \{0,1,...,k\}\}$ $\in \Omega$. Тогда 
\begin{itemize}
    \item Длиной траектории $\Delta$ назовём величину $|\Delta|$, равную $k$;
    \item При $i\in \{1,...|\Delta|\}$ $i$-ым шагом траектории $\Delta$ будем называть величину $\delta_i$, равную $\Delta_i-\Delta_{i-1}$;
    \item При $i\in \{1,...|\Delta|\}$ $i$-ый шаг траектории $\Delta$ будем называть шагом ``вверх'', если $\delta_i=1$ и шагом ``вниз'', если $\delta_i=-1$.
\end{itemize}
\end{Def}

\begin{Oboz}
Пусть $\Delta\in\Omega$. Тогда 
\begin{itemize}
    \item Такое $i$, что $i$-ый шаг траектории $\Delta$ является $j$-ым шагом ``вверх'' траектории $\Delta$ будем обозначать за $\gamma(\Delta,j)$;
    \item Количество шагов ``вверх'' траектории $\Delta$, то есть величину, равную $|\{i:\delta_i=1\}|$ обозначим за $s(\Delta)$. 
\end{itemize}
\end{Oboz}

\begin{Zam}
Пусть $X,Y\in\mathbb{N}_0$ и $\Delta\in \Omega(X,Y)$. Тогда
\begin{enumerate}
    \item $$\sum_{i=1}^{|\Delta|} \delta_i = -Y+X;$$
    \item $\Delta$ содержит $s(\Delta)$ шагов ``вверх'' и $Y-X+s(\Delta)$ шагов ``вниз'';
    \item $|\Delta|=Y-X+2s(\Delta)$.
\end{enumerate}
\end{Zam}

\begin{Def}
Пусть $x,y \in \mathbb{YF}: |y|\ge |x|$, $\Delta\in\Omega(|x|,|y|)$. Тогда путь 
$$y=y_0y_1y_2...y_n=x$$
в графе Юнга-Фибоначчи назовём $\Delta$-путём, если $n=|\Delta|$ и $\forall i \in \{0,...,n\} \; |y_i|=\Delta_i$. Количество $yx$-$\Delta$-путей в $\mathbb{YF}$ будем обозначать как $d(x,y,\Delta)$.
\end{Def}

\begin{Zam}
Путь по $\mathbb{YF}$ является $\Delta$-путём при $s(\Delta)=0\Longleftrightarrow$ этот путь является путём ``вниз''.
\end{Zam}

\begin{Def} \label{sopr}
Пусть $x,y \in \mathbb{YF}: |y|\ge |x|$; $S\in \mathbb{N}_0$; $\Delta\in\Omega$, при этом путь
$$y=y_0y_1y_2...y_n=x$$
является $\Delta$-путём. Тогда скажем, что он является $S$-путём, если $s(\Delta)=S$. Количество $yx$-$S$-путей в $\mathbb{YF}$ будем обозначать как $d(x,y,S)$.
\end{Def}

\begin{Zam}
Множество путей ``вниз'' равно множеству $0$-путей.
\end{Zam}

Наша цель -- найти $d(x,y,S)$. Для начала будем искать $d(x,y,\Delta)$.

\begin{Def}
При $\Delta\in\Omega$ средней функцией для $\Delta$ назовём функцию
$$e(\Delta,x) \quad (x\in\{1,...,s(\Delta)\}),$$
определённую следующим образом:

\begin{itemize}
    \item $e(\Delta,1)=\Delta_{\gamma(\Delta,1)}$;
    \item $e(\Delta,2)=\Delta_{\gamma(\Delta,2)}$;
    \item $...$
    \item $e(\Delta,m)=\Delta_{\gamma(\Delta,m)}$;
    \item $...$
    \item $e(\Delta,s(\Delta))=\Delta_{\gamma(\Delta,s(\Delta))}$.
\end{itemize}
\end{Def}

\begin{Oboz}
Пусть $\Delta \in \Omega$. Тогда мультимножество 
$$\left\{e(\Delta,1),e(\Delta,2),...,e(\Delta,m),...,e(\Delta,s(\Delta)) \right\}$$
обозначим за $E(\Delta)$.
\end{Oboz}

\begin{Zam}
Пусть $\Delta \in \Omega$. Тогда 
$$|E(\Delta)|=s(\Delta).$$
\end{Zam}

\begin{Ex}
Пример траектории $\Delta\in\Omega$:

\begin{tikzpicture}
\draw[ultra thick, ->] (0.5,0) -- (0.5,7);
\draw[thin, dashed, ->] (1.5,0) -- (1.5,7);
\draw[thin, dashed, ->] (2.5,0) -- (2.5,7);
\draw[thin, dashed, ->] (3.5,0) -- (3.5,7);
\draw[thin, dashed, ->] (4.5,0) -- (4.5,7);
\draw[thin, dashed, ->] (5.5,0) -- (5.5,7);
\draw[thin, dashed, ->] (6.5,0) -- (6.5,7);
\draw[thin, dashed, ->] (7.5,0) -- (7.5,7);
\draw[thin, dashed, ->] (8.5,0) -- (8.5,7);
\draw[thin, dashed, ->] (9.5,0) -- (9.5,7);
\draw[thin, dashed, ->] (10.5,0) -- (10.5,7);
\draw[thin, dashed, ->] (11.5,0) -- (11.5,7);
\draw[thin, dashed, ->] (12.5,0) -- (12.5,7);
\draw[thin, dashed, ->] (13.5,0) -- (13.5,7);
\draw[thin, dashed, ->] (14.5,0) -- (14.5,7);

\draw[ultra thick, ->] (0,0.5) -- (15,0.5);
\draw[thin, dashed, ->] (0,1.5) -- (15,1.5);
\draw[thin, dashed, ->] (0,2.5) -- (15,2.5);
\draw[thin, dashed, ->] (0,3.5) -- (15,3.5);
\draw[thin, dashed, ->] (0,4.5) -- (15,4.5);
\draw[thin, dashed, ->] (0,5.5) -- (15,5.5);
\draw[thin, dashed, ->] (0,6.5) -- (15,6.5);

\draw[ultra thick, red, dotted, ->] (0.5,3.5) -- (1.5,2.5);
\draw[ultra thick, red, dotted, ->] (1.5,2.5) -- (2.5,1.5);
\draw[ultra thick, red, ->] (2.5,1.5) -- (3.5,2.5);
\draw[ultra thick, red, dotted, ->] (3.5,2.5) -- (4.5,1.5);
\draw[ultra thick, red, dotted, ->] (4.5,1.5) -- (5.5,0.5);
\draw[ultra thick, red, ->] (5.5,0.5) -- (6.5,1.5);
\draw[ultra thick, red, ->] (6.5,1.5) -- (7.5,2.5);
\draw[ultra thick, red, dotted, ->] (7.5,2.5) -- (8.5,1.5);
\draw[ultra thick, red, ->] (8.5,1.5) -- (9.5,2.5);
\draw[ultra thick, red, ->] (9.5,2.5) -- (10.5,3.5);
\draw[ultra thick, red, dotted, ->] (10.5,3.5) -- (11.5,2.5);
\draw[ultra thick, red, dotted, ->] (11.5,2.5) -- (12.5,1.5);
\draw[ultra thick, red, dotted, ->] (12.5,1.5) -- (13.5,0.5);
\draw[ultra thick, red, ->] (13.5,0.5) -- (14.5,1.5);

\path (0.25,0.25) node  {$0$}
-- (1.25,0.25)  node  {$1$}
-- (2.25,0.25)  node  {$2$}
-- (3.25,0.25)  node  {$3$}
-- (4.25,0.25)  node  {$4$}
-- (5.25,0.25)  node  {$5$}
-- (6.25,0.25)  node  {$6$}
-- (7.25,0.25)  node  {$7$}
-- (8.25,0.25)  node  {$8$}
-- (9.25,0.25)  node  {$9$}
-- (10.25,0.25)  node  {$10$}
-- (11.25,0.25)  node  {$11$}
-- (12.25,0.25)  node  {$12$}
-- (13.25,0.25)  node  {$13$}
-- (14.25,0.25)  node  {$14$};

\path (0.25,1.25)  node  {$1$}
-- (0.25,2.25)  node  {$2$}
-- (0.25,3.25)  node  {$3$}
-- (0.25,4.25)  node  {$4$}
-- (0.25,5.25)  node  {$5$}
-- (0.25,6.25)  node  {$6$};

\path (3.25,2.75) node [font=\Large,red]  {$2$}
-- (6.25,1.75) node [font=\Large,red]  {$1$}
-- (7.25,2.75) node [font=\Large,red]  {$2$}
-- (9.25,2.75) node [font=\Large,red]  {$2$}
-- (10.25,3.75) node [font=\Large,red]  {$3$}
-- (14.25,1.75) node [font=\Large,red]  {$1$};

\end{tikzpicture}

На рисунке изображена траектория $\Delta$:
$$\Delta_0=3,\;\Delta_1=2,\;\Delta_2=1,\;\Delta_3=2,\;\Delta_4=1,\;\Delta_5=0,\;\Delta_6=1,\;\Delta_7=2,$$
$$\Delta_8=1,\;\Delta_9=2,\;\Delta_{10}=3,\;\Delta_{11}=2,\;\Delta_{12}=1,\;\Delta_{13}=0,\;\Delta_{14}=1.$$

Свойства данной траектории:

\begin{itemize}
    \item Это $(3,1)$-траектория, иначе говоря, $\Delta\in\Omega(1,3)$;
    \item Длина траектории $|\Delta|=14$;
    \item $$\delta_1=-1,\;\delta_2=-1,\;\delta_3=1,\;\delta_4=-1,\;\delta_5=-1,\;\delta_6=1,\;\delta_7=1,$$
    $$\delta_8=-1,\;\delta_9=1,\;\delta_{10}=1,\;\delta_{11}=-1,\;\delta_{12}=-1,\;\delta_{13}=-1,\;\delta_{14}=1;$$
    \item $3$-ий, $6$-ой, $7$-ой, $9$-ый, $10$-ый и $14$-ый шаг траектории $\Delta$ -- шаги ``вверх'';
    \item $1$-ый, $2$-ой, $4$-ый, $5$-ый, $8$-ой, $11$-ый, $12$-ый и $13$-ый шаг траектории $\Delta$ -- шаги ``вниз'';
    \item $$\gamma(\Delta,1)=3,  \;\gamma(\Delta,2)=6, \gamma(\Delta,3)=7, \;\gamma(\Delta,4)=9, \;\gamma(\Delta,5)=10, \; \gamma(\Delta,6)=14;$$
    \item $s(\Delta)=6$;
    \item $$e(\Delta,1)=2,\;e(\Delta,1)=1,\;e(\Delta,1)=2,\;e(\Delta,1)=2,\;e(\Delta,1)=3,\;e(\Delta,1)=1;$$
    \item $E(\Delta)=\{1,1,2,2,2,3\}$.
\end{itemize}
\end{Ex}

Мы готовы к формулировке теоремы о количестве $yx$-$\Delta$ путей и её доказательству.

\subsection{Теорема о количестве $\Delta$-путей между двумя вершинами в графе Юнга-Фибоначчи}

\begin{theorem} \label{delta}
Пусть $x,y\in\mathbb{YF}:$ $|y|\ge |x|$, $z=h(x,y)$ и $\Delta\in\Omega(|x|,|y|)$. Тогда
$$d(x,y,\Delta)=\sum_{i=0}^{|x|}\left( {f\left(x,i,z\right)}\prod_{j=1}^{d(y)}\left(g\left(y,j\right)-i\right)\prod_{j=1}^{s(\Delta)}\left(e\left(\Delta,j\right)-i\right)\right).$$
\end{theorem}
\begin{proof}
Обозначим предполагаемый ответ за $o(x,y,\Delta)$.

Давайте докажем вспомогательное утверждение.

\begin{Prop} \label{2vnachale}
$$\forall y \in \mathbb{YF} \; L(y)=R(2y).$$  
\end{Prop}
\begin{proof}
\renewcommand{\labelenumi}{\arabic{enumi}$^\circ$}
$y':=2y$.

Рассмотрим два случая: 
\begin{enumerate}
    \item Пусть номер $y$ содержит хотя бы одну единицу.

\setlength{\unitlength}{0.20mm}
\begin{picture}(400,300)
\put(250,270){$y'=2221y''$}
\put(300,260){\vector(2,-1){200}}
\put(300,260){\vector(1,-2){50}}
\put(300,260){\vector(-1,-2){50}}
\put(300,260){\vector(-2,-1){200}}
\put(50,140){$y'_1=1221y''$}
\put(180,140){$y'_2=2121y''$}
\put(310,140){$y'_3=2211y''$}
\put(440,140){$y'_4=222y''$}
\put(90,130){\vector(2,-1){200}}
\put(245,130){\vector(1,-2){50}}
\put(355,130){\vector(-1,-2){50}}
\put(510,130){\vector(-2,-1){200}}
\put(270,10){$y=221y''$}
\end{picture}

На рисунке выше изображено, как могут выглядеть предки вершины $y'$ и потомки вершины $y$.

Представим $y$ в виде 
$$y=\underbrace{2...2}_{d'(y)}1y''.$$

Тогда 
$$y'=\underbrace{2...2}_{d'(y)+1}1y''.$$

Несложно убедиться в том, что $$L(y)=R(y')=\left\{1\underbrace{2...2}_{d'(y)}1y'',21\underbrace{2...2}_{d'(y)-1}1y'',...,\underbrace{2...2}_{k}1\underbrace{2...2}_{d'(y)-k}1y'',...,\underbrace{2...2}_{d'(y)}11y'',\underbrace{2...2}_{d'(y)+1}y''\right\}.$$ 

    \item Пусть номер $y$ не содержит ни одной единицы.

\setlength{\unitlength}{0.20mm}
\begin{picture}(400,300)
\put(250,270){$y'=2222$}
\put(300,260){\vector(2,-1){200}}
\put(300,260){\vector(1,-2){50}}
\put(300,260){\vector(-1,-2){50}}
\put(300,260){\vector(-2,-1){200}}
\put(50,140){$y'_1=1222$}
\put(180,140){$y'_2=2122$}
\put(310,140){$y'_3=2212$}
\put(440,140){$y'_4=2221$}
\put(90,130){\vector(2,-1){200}}
\put(245,130){\vector(1,-2){50}}
\put(355,130){\vector(-1,-2){50}}
\put(510,130){\vector(-2,-1){200}}
\put(270,10){$y=222$}
\end{picture}

На рисунке выше изображено, как могут выглядеть предки вершины $y'$ и потомки вершины $y$. 

В данном случае
$$y=\underbrace{2...2}_{d'(y)}$$
и
$$y'=\underbrace{2...2}_{d'(y)+1}.$$

Несложно убедиться в том, что $$L(y)=R(y')=\left\{1\underbrace{2...2}_{d'(y)},21\underbrace{2...2}_{d'(y)-1},...,\underbrace{2...2}_{k}1\underbrace{2...2}_{d'(y)-k},...,\underbrace{2...2}_{d'(y)-1}12,\underbrace{2...2}_{d'(y)}1\right\}.$$ 
\end{enumerate}
\end{proof}

Давайте доказывать теорему по индукции по $|\Delta|$, а при данном $|\Delta|$ по $s\left(\Delta\right)$.

\underline{\textbf{База}}: $\left|\Delta\right|=0$.

В данном случае мы рассматриваем только пути длины 0.

Очевидно, что $|y| \ge |x| \; \& \; s(\Delta) \ge 0 \; \& \; |\Delta|=|y|-|x|+2s(\Delta)=0$ $\Longrightarrow |x|=|y| \; \& \; s(\Delta)=0$. То есть мы рассматриваем пути ``вниз'' длины $0$ от $y$ до $x$, такие, что $|y|=|x|$.

Таким образом, понятно, что в данном случае 
$$d(x,y,\Delta)=d(x,y)=\sum_{i=0}^{|x|}\left( {f\left(x,i,z\right)}\prod_{j=1}^{d(y)}\left(g\left(y,j\right)-i\right)\right)=$$
$$=\sum_{i=0}^{|x|}\left( {f\left(x,i,z\right)}\prod_{j=1}^{d(y)}\left(g\left(y,j\right)-i\right)\prod_{j=1}^{0}\left(e\left(\Delta,j\right)-i\right)\right)=o(x,y,\Delta).$$

Что и требовалось.

\underline{\textbf{Переход}} к $|\Delta| \ge 1$.

Снова будем разбирать случаи.
\renewcommand{\labelenumi}{\arabic{enumi}$^\circ$}
\begin{enumerate}
    \item Пусть $|y|>|x|$ и $\delta_1=-1$.

Пусть $\forall k\in\{1,...,r(y)\} \quad z_k=h(x,y_k)$. 

Давайте введём $\Delta'$ следующим образом: 

\begin{tikzpicture}

\draw[ultra thick, ->] (0.5,0) -- (0.5,7);
\draw[thin, dashed, ->] (1.5,0) -- (1.5,7);
\draw[thin, dashed, ->] (2.5,0) -- (2.5,7);
\draw[thin, dashed, ->] (3.5,0) -- (3.5,7);
\draw[thin, dashed, ->] (4.5,0) -- (4.5,7);
\draw[thin, dashed, ->] (5.5,0) -- (5.5,7);
\draw[thin, dashed, ->] (6.5,0) -- (6.5,7);
\draw[thin, dashed, ->] (7.5,0) -- (7.5,7);
\draw[thin, dashed, ->] (8.5,0) -- (8.5,7);
\draw[thin, dashed, ->] (9.5,0) -- (9.5,7);
\draw[thin, dashed, ->] (10.5,0) -- (10.5,7);
\draw[thin, dashed, ->] (11.5,0) -- (11.5,7);
\draw[thin, dashed, ->] (12.5,0) -- (12.5,7);
\draw[thin, dashed, ->] (13.5,0) -- (13.5,7);
\draw[thin, dashed, ->] (14.5,0) -- (14.5,7);

\draw[ultra thick, ->] (0,0.5) -- (15,0.5);
\draw[thin, dashed, ->] (0,1.5) -- (15,1.5);
\draw[thin, dashed, ->] (0,2.5) -- (15,2.5);
\draw[thin, dashed, ->] (0,3.5) -- (15,3.5);
\draw[thin, dashed, ->] (0,4.5) -- (15,4.5);
\draw[thin, dashed, ->] (0,5.5) -- (15,5.5);
\draw[thin, dashed, ->] (0,6.5) -- (15,6.5);

\draw[ultra thick, red, ->] (0.5,4.5) -- (1.5,3.5);
\draw[ultra thick, red, ->] (1.5,3.5) -- (2.5,2.5);
\draw[ultra thick, red, ->] (2.5,2.5) -- (3.5,3.5);
\draw[ultra thick, red, ->] (3.5,3.5) -- (4.5,2.5);
\draw[ultra thick, red, ->] (4.5,2.5) -- (5.5,1.5);
\draw[ultra thick, red, ->] (5.5,1.5) -- (6.5,0.5);
\draw[ultra thick, red, ->] (6.5,0.5) -- (7.5,1.5);
\draw[ultra thick, red, ->] (7.5,1.5) -- (8.5,2.5);
\draw[ultra thick, red, ->] (8.5,2.5) -- (9.5,3.5);
\draw[ultra thick, red, ->] (9.5,3.5) -- (10.5,4.5);
\draw[ultra thick, red, ->] (10.5,4.5) -- (11.5,5.5);
\draw[ultra thick, red, ->] (11.5,5.5) -- (12.5,4.5);
\draw[ultra thick, red, ->] (12.5,4.5) -- (13.5,3.5);
\draw[ultra thick, red, ->] (13.5,3.5) -- (14.5,2.5);

\path (0.25,0.25) node  {$0$}
-- (1.25,0.25)  node  {$1$}
-- (2.25,0.25)  node  {$2$}
-- (3.25,0.25)  node  {$3$}
-- (4.25,0.25)  node  {$4$}
-- (5.25,0.25)  node  {$5$}
-- (6.25,0.25)  node  {$6$}
-- (7.25,0.25)  node  {$7$}
-- (8.25,0.25)  node  {$8$}
-- (9.25,0.25)  node  {$9$}
-- (10.25,0.25)  node  {$10$}
-- (11.25,0.25)  node  {$11$}
-- (12.25,0.25)  node  {$12$}
-- (13.25,0.25)  node  {$13$}
-- (14.25,0.25)  node  {$14$};

\path (0.25,1.25)  node  {$1$}
-- (0.25,2.25)  node  {$2$}
-- (0.25,3.25)  node  {$3$}
-- (0.25,4.25)  node  {$4$}
-- (0.25,5.25)  node  {$5$}
-- (0.25,6.25)  node  {$6$};
\end{tikzpicture}

$$$$

\begin{tikzpicture}
\draw[ultra thick, ->] (1.5,0) -- (1.5,7);
\draw[thin, dashed, ->] (0.5,0) -- (0.5,7);
\draw[thin, dashed, ->] (2.5,0) -- (2.5,7);
\draw[thin, dashed, ->] (3.5,0) -- (3.5,7);
\draw[thin, dashed, ->] (4.5,0) -- (4.5,7);
\draw[thin, dashed, ->] (5.5,0) -- (5.5,7);
\draw[thin, dashed, ->] (6.5,0) -- (6.5,7);
\draw[thin, dashed, ->] (7.5,0) -- (7.5,7);
\draw[thin, dashed, ->] (8.5,0) -- (8.5,7);
\draw[thin, dashed, ->] (9.5,0) -- (9.5,7);
\draw[thin, dashed, ->] (10.5,0) -- (10.5,7);
\draw[thin, dashed, ->] (11.5,0) -- (11.5,7);
\draw[thin, dashed, ->] (12.5,0) -- (12.5,7);
\draw[thin, dashed, ->] (13.5,0) -- (13.5,7);
\draw[thin, dashed, ->] (14.5,0) -- (14.5,7);

\draw[ultra thick, ->] (0,0.5) -- (15,0.5);
\draw[thin, dashed, ->] (0,1.5) -- (15,1.5);
\draw[thin, dashed, ->] (0,2.5) -- (15,2.5);
\draw[thin, dashed, ->] (0,3.5) -- (15,3.5);
\draw[thin, dashed, ->] (0,4.5) -- (15,4.5);
\draw[thin, dashed, ->] (0,5.5) -- (15,5.5);
\draw[thin, dashed, ->] (0,6.5) -- (15,6.5);

\draw[ultra thick, blue, ->] (1.5,3.5) -- (2.5,2.5);
\draw[ultra thick, blue, ->] (2.5,2.5) -- (3.5,3.5);
\draw[ultra thick, blue, ->] (3.5,3.5) -- (4.5,2.5);
\draw[ultra thick, blue, ->] (4.5,2.5) -- (5.5,1.5);
\draw[ultra thick, blue, ->] (5.5,1.5) -- (6.5,0.5);
\draw[ultra thick, blue, ->] (6.5,0.5) -- (7.5,1.5);
\draw[ultra thick, blue, ->] (7.5,1.5) -- (8.5,2.5);
\draw[ultra thick, blue, ->] (8.5,2.5) -- (9.5,3.5);
\draw[ultra thick, blue, ->] (9.5,3.5) -- (10.5,4.5);
\draw[ultra thick, blue, ->] (10.5,4.5) -- (11.5,5.5);
\draw[ultra thick, blue, ->] (11.5,5.5) -- (12.5,4.5);
\draw[ultra thick, blue, ->] (12.5,4.5) -- (13.5,3.5);
\draw[ultra thick, blue, ->] (13.5,3.5) -- (14.5,2.5);

\path (0.25,0.25)
-- (1.25,0.25)  node  {$0$}
-- (2.25,0.25)  node  {$1$}
-- (3.25,0.25)  node  {$2$}
-- (4.25,0.25)  node  {$3$}
-- (5.25,0.25)  node  {$4$}
-- (6.25,0.25)  node  {$5$}
-- (7.25,0.25)  node  {$6$}
-- (8.25,0.25)  node  {$7$}
-- (9.25,0.25)  node  {$8$}
-- (10.25,0.25)  node  {$9$}
-- (11.25,0.25)  node  {$10$}
-- (12.25,0.25)  node  {$11$}
-- (13.25,0.25)  node  {$12$}
-- (14.25,0.25)  node  {$13$};

\path (1.25,1.25)  node  {$1$}
-- (1.25,2.25)  node  {$2$}
-- (1.25,3.25)  node  {$3$}
-- (1.25,4.25)  node  {$4$}
-- (1.25,5.25)  node  {$5$}
-- (1.25,6.25)  node  {$6$};
\end{tikzpicture}

На рисунках изображено, как по траектории $\Delta$ (верхний рисунок) строится траектория $\Delta'$ (нижний рисунок). В данном случае первый шаг траектории $\Delta$ -- это шаг ``вниз'', и мы просто его убираем. Формально:
$$\left|\Delta'\right|=|\Delta|-1; \;  \forall i \in \{ 1,...,|\Delta'| \}\;\Delta'_i=\Delta_{i+1}.$$ 

Несложно заметить, что 
\begin{itemize}
    \item $\forall k \in \{1,...,r(y) \} \quad \Delta' \in \Omega(|x|,|y_k|)$;
    \item $s(\Delta)=s(\Delta')$;
    \item $\forall j \in \{1,...,s(\Delta)\} \quad e(\Delta,j)=e(\Delta',j)$;
    \item $E(\Delta)=E(\Delta')$;
    \item $|\Delta'|<|\Delta|$, то есть мы можем воспользоваться предположением индукции.
\end{itemize}

Но главное наблюдение заключается в том, что $\forall k\in\{1,...,r(y)\}$ существует биекция между $\Delta$-путями $yy_k...x$ и $\Delta'$-путями $y_k...x$. Давайте считать:
$$d(x,y,\Delta)=\sum_{k=1}^{r(y)}d(x,y_k,\Delta')=(\text{По предположению})=\sum_{k=1}^{r(y)}o\left(x,y_k,\Delta'\right)=$$
$$=\sum_{k=1}^{r(y)} \left(\sum_{i=0}^{|x|}\left( {f\left(x,i,z\right)}\prod_{j=1}^{d\left(y_k\right)}\left(g\left(y_k,j\right)-i\right)\prod_{j=1}^{s\left(\Delta'\right)}\left(e\left(\Delta',j\right)-i\right)\right)\right)= $$
$$=\sum_{i=0}^{|x|}\left(\sum_{k=1}^{r(y)}\left( {f\left(x,i,z\right)}\prod_{j=1}^{d(y_k)}\left(g\left(y_k,j\right)-i\right)\prod_{j=1}^{s\left(\Delta'\right)}\left(e\left(\Delta',j\right)-i\right)\right)\right)=$$
$$=\sum_{i=0}^{|x|}\left(\sum_{k=1}^{r(y)}\left( {f\left(x,i,z_k\right)}\prod_{j=1}^{d(y_k)}\left(g\left(y_k,j\right)-i\right)\prod_{j=1}^{s\left(\Delta\right)}\left(e\left(\Delta,j\right)-i\right)\right)\right)=$$
$$=\text{(По лемме \ref{main})}=$$
$$=\sum_{i=0}^{|x|}\left({f\left(x,i,z\right)}\prod_{j=1}^{d(y)}\left(g\left(y,j\right)-i\right)\prod_{j=1}^{s\left(\Delta\right)}\left(e\left(\Delta,j\right)-i\right)\right)=o\left(x,y,\Delta \right).$$

Что и требовалось.
    
\item

Пусть $|y|\ge|x|$ и $\delta_1=1$.

Пусть $y'=2y$ и $\forall k\in\{1,...,r(y)\} \quad z_k=h(x,y_k)$. 

Давайте введём $\Delta'$ следующим образом: 

\begin{tikzpicture}

\draw[ultra thick, ->] (0.5,0) -- (0.5,7);
\draw[thin, dashed, ->] (1.5,0) -- (1.5,7);
\draw[thin, dashed, ->] (2.5,0) -- (2.5,7);
\draw[thin, dashed, ->] (3.5,0) -- (3.5,7);
\draw[thin, dashed, ->] (4.5,0) -- (4.5,7);
\draw[thin, dashed, ->] (5.5,0) -- (5.5,7);
\draw[thin, dashed, ->] (6.5,0) -- (6.5,7);
\draw[thin, dashed, ->] (7.5,0) -- (7.5,7);
\draw[thin, dashed, ->] (8.5,0) -- (8.5,7);
\draw[thin, dashed, ->] (9.5,0) -- (9.5,7);
\draw[thin, dashed, ->] (10.5,0) -- (10.5,7);
\draw[thin, dashed, ->] (11.5,0) -- (11.5,7);
\draw[thin, dashed, ->] (12.5,0) -- (12.5,7);
\draw[thin, dashed, ->] (13.5,0) -- (13.5,7);
\draw[thin, dashed, ->] (14.5,0) -- (14.5,7);

\draw[ultra thick, ->] (0,0.5) -- (15,0.5);
\draw[thin, dashed, ->] (0,1.5) -- (15,1.5);
\draw[thin, dashed, ->] (0,2.5) -- (15,2.5);
\draw[thin, dashed, ->] (0,3.5) -- (15,3.5);
\draw[thin, dashed, ->] (0,4.5) -- (15,4.5);
\draw[thin, dashed, ->] (0,5.5) -- (15,5.5);
\draw[thin, dashed, ->] (0,6.5) -- (15,6.5);

\draw[ultra thick, red, ->] (0.5,2.5) -- (1.5,3.5);
\draw[ultra thick, red, ->] (1.5,3.5) -- (2.5,2.5);
\draw[ultra thick, red, ->] (2.5,2.5) -- (3.5,3.5);
\draw[ultra thick, red, ->] (3.5,3.5) -- (4.5,2.5);
\draw[ultra thick, red, ->] (4.5,2.5) -- (5.5,1.5);
\draw[ultra thick, red, ->] (5.5,1.5) -- (6.5,0.5);
\draw[ultra thick, red, ->] (6.5,0.5) -- (7.5,1.5);
\draw[ultra thick, red, ->] (7.5,1.5) -- (8.5,2.5);
\draw[ultra thick, red, ->] (8.5,2.5) -- (9.5,3.5);
\draw[ultra thick, red, ->] (9.5,3.5) -- (10.5,4.5);
\draw[ultra thick, red, ->] (10.5,4.5) -- (11.5,5.5);
\draw[ultra thick, red, ->] (11.5,5.5) -- (12.5,4.5);
\draw[ultra thick, red, ->] (12.5,4.5) -- (13.5,3.5);
\draw[ultra thick, red, ->] (13.5,3.5) -- (14.5,2.5);

\path (0.25,0.25) node  {$0$}
-- (1.25,0.25)  node  {$1$}
-- (2.25,0.25)  node  {$2$}
-- (3.25,0.25)  node  {$3$}
-- (4.25,0.25)  node  {$4$}
-- (5.25,0.25)  node  {$5$}
-- (6.25,0.25)  node  {$6$}
-- (7.25,0.25)  node  {$7$}
-- (8.25,0.25)  node  {$8$}
-- (9.25,0.25)  node  {$9$}
-- (10.25,0.25)  node  {$10$}
-- (11.25,0.25)  node  {$11$}
-- (12.25,0.25)  node  {$12$}
-- (13.25,0.25)  node  {$13$}
-- (14.25,0.25)  node  {$14$};

\path (0.25,1.25)  node  {$1$}
-- (0.25,2.25)  node  {$2$}
-- (0.25,3.25)  node  {$3$}
-- (0.25,4.25)  node  {$4$}
-- (0.25,5.25)  node  {$5$}
-- (0.25,6.25)  node  {$6$};
\end{tikzpicture}

$$$$

\begin{tikzpicture}
\draw[ultra thick, ->] (0.5,0) -- (0.5,7);
\draw[thin, dashed, ->] (1.5,0) -- (1.5,7);
\draw[thin, dashed, ->] (2.5,0) -- (2.5,7);
\draw[thin, dashed, ->] (3.5,0) -- (3.5,7);
\draw[thin, dashed, ->] (4.5,0) -- (4.5,7);
\draw[thin, dashed, ->] (5.5,0) -- (5.5,7);
\draw[thin, dashed, ->] (6.5,0) -- (6.5,7);
\draw[thin, dashed, ->] (7.5,0) -- (7.5,7);
\draw[thin, dashed, ->] (8.5,0) -- (8.5,7);
\draw[thin, dashed, ->] (9.5,0) -- (9.5,7);
\draw[thin, dashed, ->] (10.5,0) -- (10.5,7);
\draw[thin, dashed, ->] (11.5,0) -- (11.5,7);
\draw[thin, dashed, ->] (12.5,0) -- (12.5,7);
\draw[thin, dashed, ->] (13.5,0) -- (13.5,7);
\draw[thin, dashed, ->] (14.5,0) -- (14.5,7);

\draw[ultra thick, ->] (0,0.5) -- (15,0.5);
\draw[thin, dashed, ->] (0,1.5) -- (15,1.5);
\draw[thin, dashed, ->] (0,2.5) -- (15,2.5);
\draw[thin, dashed, ->] (0,3.5) -- (15,3.5);
\draw[thin, dashed, ->] (0,4.5) -- (15,4.5);
\draw[thin, dashed, ->] (0,5.5) -- (15,5.5);
\draw[thin, dashed, ->] (0,6.5) -- (15,6.5);

\draw[ultra thick, blue, ->] (0.5,4.5) -- (1.5,3.5);
\draw[ultra thick, blue, ->] (1.5,3.5) -- (2.5,2.5);
\draw[ultra thick, blue, ->] (2.5,2.5) -- (3.5,3.5);
\draw[ultra thick, blue, ->] (3.5,3.5) -- (4.5,2.5);
\draw[ultra thick, blue, ->] (4.5,2.5) -- (5.5,1.5);
\draw[ultra thick, blue, ->] (5.5,1.5) -- (6.5,0.5);
\draw[ultra thick, blue, ->] (6.5,0.5) -- (7.5,1.5);
\draw[ultra thick, blue, ->] (7.5,1.5) -- (8.5,2.5);
\draw[ultra thick, blue, ->] (8.5,2.5) -- (9.5,3.5);
\draw[ultra thick, blue, ->] (9.5,3.5) -- (10.5,4.5);
\draw[ultra thick, blue, ->] (10.5,4.5) -- (11.5,5.5);
\draw[ultra thick, blue, ->] (11.5,5.5) -- (12.5,4.5);
\draw[ultra thick, blue, ->] (12.5,4.5) -- (13.5,3.5);
\draw[ultra thick, blue, ->] (13.5,3.5) -- (14.5,2.5);

\path (0.25,0.25) node  {$0$}
-- (1.25,0.25)  node  {$1$}
-- (2.25,0.25)  node  {$2$}
-- (3.25,0.25)  node  {$3$}
-- (4.25,0.25)  node  {$4$}
-- (5.25,0.25)  node  {$5$}
-- (6.25,0.25)  node  {$6$}
-- (7.25,0.25)  node  {$7$}
-- (8.25,0.25)  node  {$8$}
-- (9.25,0.25)  node  {$9$}
-- (10.25,0.25)  node  {$10$}
-- (11.25,0.25)  node  {$11$}
-- (12.25,0.25)  node  {$12$}
-- (13.25,0.25)  node  {$13$}
-- (14.25,0.25)  node  {$14$};

\path (0.25,1.25)  node  {$1$}
-- (0.25,2.25)  node  {$2$}
-- (0.25,3.25)  node  {$3$}
-- (0.25,4.25)  node  {$4$}
-- (0.25,5.25)  node  {$5$}
-- (0.25,6.25)  node  {$6$};
\end{tikzpicture}

На рисунках выше изображено, как по траектории $\Delta$ (верхний рисунок) строится траектория $\Delta'$ (нижний рисунок). В данном случае первый шаг траектории $\Delta$ -- это шаг ``вверх'', и мы просто меняем его на шаг ``вниз'', при это оставляя все значения $\Delta_i$, кроме $\Delta_1$ прежними. Формально:
$$|\Delta'|=|\Delta|; \; \Delta'_1=\Delta_1+2; \;  \forall i \in \left\{ 1,...,|\Delta'| \right\} \; \Delta'_i=\Delta_{i}.$$ 

Несложно заметить, что 
\begin{itemize}
    \item $ \Delta'\in \Omega(|x|,|y'|)$;
    \item $s(\Delta')=s(\Delta)-1$;
    \item $\forall j \in \{1,...,s(\Delta')\}\quad e(\Delta',j)=e(\Delta,j+1)$;
    \item ${|y|+1}\cup E(\Delta')=E(\Delta)$;
    \item $|\Delta'|=|\Delta|$ и $s(\Delta')<s(\Delta)$, то есть мы можем воспользоваться предположением индукции.
\end{itemize}

Но главное наблюдение заключается в том, что по Утверждению \ref{2vnachale} $\forall k\in\{1,...,r(y)\}$  существует биекция между $\Delta$-путями $yy_k...x$ и $\Delta'$-путями $y'y_k...x$. Давайте считать:
$$d(x,y,\Delta)=d\left(x,y',\Delta'\right)=$$
$$=\sum_{i=0}^{|x|}\left( {f\left(x,i,z\right)}\prod_{j=1}^{d\left(y'\right)}\left(g\left(y',j\right)-i\right)\prod_{j=1}^{s\left(\Delta'\right)}\left(e\left(\Delta',j\right)-i\right)\right)=$$
$$=\sum_{i=0}^{|x|}\left( {f\left(x,i,z\right)}\prod_{j=1}^{d\left(y'\right)}\left(g\left(y',j\right)-i\right)\prod_{j=2}^{s(\Delta)}\left(e\left(\Delta,j\right)-i\right)\right)=$$
$$=\text{(Знаем, что $|y| \ge |x| \ge |i| \Longrightarrow |y|+1>i \Longrightarrow |y|+1-i \ne 0$)}=$$
$$=\sum_{i=0}^{|x|}\left( {f\left(x,i,z\right)}\left(\prod_{j=1}^{d(y)}\left(g\left(y,j\right)-i\right)\right)(|y|+1-i) \frac{\prod_{j=1}^{s(\Delta)}\left(e\left(\Delta,j\right)-i\right)}{|y|+1-i}\right)=$$
$$=\sum_{i=0}^{|x|}\left( {f\left(x,i,z\right)}\prod_{j=1}^{d(y)}\left(g\left(y,j\right)-i\right)\prod_{j=1}^{s(\Delta)}\left(e\left(\Delta,j\right)-i\right)\right)=o(x,y,\Delta).$$
Второй случай разобран.

Теперь перед переходом к следующему случаю давайте докажем ещё вспомогательные утверждения:
\renewcommand{\labelenumii}{\arabic{enumii}$)$}

\begin{Prop} \label{drob}
$\;$

\begin{enumerate}
    \item Пусть $x\in\mathbb{YF}$ представляется в виде $x=\alpha_1...\alpha_A$, где $\alpha_j \in \{1,2 \}$. Тогда
$$\prod_{j=1}^{d(x)}g\left(x,j\right)=\frac{|x|!}{(\alpha_{A})(\alpha_{A}+\alpha_{A-1})...(\alpha_{A}+...+\alpha_1)};$$

    \item
Пусть $x\in\mathbb{YF}$ представляется в виде $x=\alpha_1...\alpha_a\alpha_{a+1}...\alpha_A$, где $|\alpha_{a+1}...\alpha_A|=i, \alpha_j \in \{1,2 \}$. Тогда
$$\prod_{j=1}^{d(x)}\left(g\left(x,j\right)-i\right)=\prod_{j\in \{1,...,|x| \} \textbackslash i}\left( j-i\right)\frac{1}{(\alpha_{A}-i)(\alpha_{A}+\alpha_{A-1}-i)...(\alpha_{A}+...+\alpha_{a+2}-i)}\cdot$$
$$\cdot\frac{1}{(\alpha_{A}+...+\alpha_{a}-i)(\alpha_{A}+...+\alpha_{a-1}-i)...(\alpha_{A}+...+\alpha_{1}-i)}.$$
\end{enumerate}

\end{Prop}
\begin{proof}
$\;$

\begin{enumerate}
    \item

Введём временное обозначение: при данном представлении $x$
$$P(x):=(\alpha_{A})(\alpha_{A}+\alpha_{A-1})...(\alpha_{A}+...+\alpha_1).$$

Ясно, что каждый множитель в правой части положителен, поэтому при данных $P(x)\ne 0$.

Поэтому наше Утверждение равносильно тому, что
$$P(x)\prod_{j=1}^{d(x)}g\left(x,j\right)=|x|!.$$

Рассмотрим представление $x$ в виде $$x=\underbrace{1...1}_{\beta_{d(x)}}2\underbrace{1...1}_{\beta_{d(x)-1}}2...2\underbrace{1...1}_{\beta_1}2\underbrace{1...1}_{\beta_0},$$
и вспомним определение функции $g$:
\begin{itemize}
    \item $g(x,1)=\beta_0+1$;
    \item $g(x,2)=\beta_0+\beta_1+3$;
    \item ...
    \item $g(x,m)=\beta_0+...+\beta_{m-1}+2m-1$;
    \item ...
    \item $g(x,d(x))=\beta_0+...+\beta_{d(x)-1}+2d(x)-1$.
\end{itemize}

Множители $P(x)$ имеют вид:
\begin{itemize}
    \item 1;
    \item 2;
    \item ...
    \item $\beta_0$;
    \item $\beta_0+2$;
    \item $\beta_0+2+1$;
    \item $\beta_0+2+2$;
    \item ...
    \item $\beta_0+2+\beta_1$;
    \item $\beta_0+2+\beta_1+2$;
    \item ... ...
    \item $\beta_0+2+\beta_1+2+...+\beta_{d(x)-1}+2+\beta_{d(x)}$.
\end{itemize}

Ясно, что эти числа вместе -- это в точности числа $1$,...,$|x|$. Отсюда и следует Утверждение. 
    
    \item

Введём новое временное обозначение: при данном представлении $x$
$$P(x):=(\alpha_{A}-i)(\alpha_{A}+\alpha_{A-1}-i)...(\alpha_{A}+...+\alpha_{a+2}-i)\cdot$$
$$\cdot(\alpha_{A}+...+\alpha_{a}-i)(\alpha_{A}+...+\alpha_{a-1}-i)...(\alpha_{A}+...+\alpha_{1}-i).$$

Ясно, что каждый множитель в верхней строке правой части отрицателен, а в нижней -- положителен, поэтому при данных $x\in\mathbb{YF}$ $P(x)\ne 0$.

Поэтому наше Утверждение равносильно тому, что
$$P(x)\prod_{j=1}^{d(x)}\left(g\left(x,j\right)-i\right)=\prod_{j\in \{1,...,|x| \} \textbackslash i}\left( j-i\right).$$

Рассмотрим представление $x$ в виде $$x=\underbrace{1...1}_{\beta_{d(x)}}2\underbrace{1...1}_{\beta_{d(x)-1}}2...2\underbrace{1...1}_{\beta_1}2\underbrace{1...1}_{\beta_0},$$
и вспомним определение функции $g$:
\begin{itemize}
    \item $g(x,1)-i=\beta_0+1-i;$
    \item $g(x,2)-i=\beta_0+\beta_1+3-i;$
    \item ...
    \item $g(x,m)-i=\beta_0+...+\beta_{m-1}+2m-1-i;$
    \item ...
    \item $g(x,d(x))-i=\beta_0+...+\beta_{d(x)-1}+2d(x)-1-i.$
\end{itemize}

Числа вида $\alpha_A+...+\alpha_k$ имеют вид
\begin{itemize}
    \item 1-$i$;
    \item 2-$i$;
    \item ...
    \item $\beta_0-i$;
    \item $\beta_0+2-i$;
    \item $\beta_0+2+1-i$;
    \item $\beta_0+2+2-i$;
    \item ...
    \item $\beta_0+2+\beta_1-i$;
    \item $\beta_0+2+\beta_1+2-i$;
    \item ... ...
    \item $\beta_0+2+\beta_1+2+...+\beta_{d(x)-1}+2+\beta_{d(x)}-i$.
\end{itemize}

Ясно, что эти числа вместе -- это в точности числа $(1-i)$,...,$(|x|-i)$. Кроме того, несложно убедиться в том, что множители левой части равенства 
$$P(x)\prod_{j=1}^{d(x)}\left(g\left(x,j\right)-i\right)=\prod_{j\in \{1,...,|x| \} \textbackslash i}\left( j-i\right)$$
-- это те же самые числа, кроме числа $\alpha_{a+1}+...+\alpha_{A+1}-i=0$ (находящегося во второй группе чисел). То есть эти множители -- это числа $1-i$,...,$|x|-i$, кроме $i-i$. Отсюда Утверждение мгновенно следует. 
\end{enumerate}
\end{proof}

\begin{Lemma} \label{xyz0}
Пусть $x,y \in \mathbb{YF}:$ $|x|=|y|$. Тогда $\forall i \in\{1,...,|x|\}$
$$ {f\left(x,i,0\right)}\prod_{j=1}^{d(y)}\left(g\left(y,j\right)-i\right)={f\left(y,i,0\right)}\prod_{j=1}^{d(x)}\left(g\left(x,j\right)-i\right).$$ 
\end{Lemma}
\begin{proof}

Рассмотрим пять случаев:
\renewcommand{\labelenumi}{\roman{enumi}$^\circ$}
\renewcommand{\labelenumii}{\roman{enumii}$^\circ$}

\begin{enumerate}
    \item Пусть $i=0$.
    
    Рассмотрим представление $x$ в виде $x=\alpha_1...\alpha_a$, где $\alpha_j \in \{1,2 \}$ и представление $y$ в виде $y=\beta_1...\beta_b$, где $\beta_j \in \{1,2 \}$. Тогда 
$${f\left(x,0,0\right)}\prod_{j=1}^{d(y)}g\left(y,j\right)=\frac{1}{(\alpha_a)(\alpha_a+\alpha_{a-1})...(\alpha_a+...+\alpha_1)}\prod_{j=1}^{d(y)}g\left(y,j\right)=$$
$$= (\text{По Утверждению \ref{drob}})=$$
$$=\frac{1}{(\alpha_a)(\alpha_a+\alpha_{a-1})...(\alpha_a+...+\alpha_1)}\cdot \frac{|y|!}{(\beta_{b})(\beta_{b}+\beta_{b-1})...(\beta_{b}+...+\beta_{1})}=$$
$$=\frac{|x|!}{(\alpha_a)(\alpha_a+\alpha_{a-1})...(\alpha_a+...+\alpha_1)}\cdot \frac{1}{(\beta_{b})(\beta_{b}+\beta_{b-1})...(\beta_{b}+...+\beta_{1})}=$$
$$=(\text{По Утверждению \ref{drob}})=$$
$$= \frac{1}{(\beta_{b})(\beta_{b}+\beta_{b-1})...(\beta_{b}+...+\beta_{1})}\prod_{j=1}^{d(x)}g\left(x,j\right)= f(y,0,0)\prod_{j=1}^{d(x)}g\left(x,j\right).$$
Что и требовалось.

\item Пусть $i\ne 0$, и $x$ представляется в виде $x=\alpha_1...\alpha_a\alpha_{a+1}...\alpha_A$, где $|\alpha_{a+1}...\alpha_A|=i, \; \alpha_j \in \{1,2 \}$, а $y$ представляется в виде $y=\beta_1...\beta_b\beta_{b+1}...\beta_B$, где $|\beta_{b+1}...\beta_B|=i, \; \beta_j \in \{1,2 \}$. 

Тогда 
$${f\left(x,i,0\right)}\prod_{j=1}^{d(y)}\left(g\left(y,j\right)-i\right)=\frac{1}{(-\alpha_{a+1})(-\alpha_{a+1}-\alpha_{a+2})...(-\alpha_{a+1}-...-\alpha_{A})}\cdot$$
$$\cdot\frac{1}{(\alpha_a)(\alpha_a+\alpha_{a-1})...(\alpha_a+...+\alpha_1)}\prod_{j=1}^{d(y)}\left(g\left(y,j\right)-i\right)=(\text{По Утверждению \ref{drob}})=$$
$$=\frac{1}{(-\alpha_{a+1})(-\alpha_{a+1}-\alpha_{a+2})...(-\alpha_{a+1}-...-\alpha_{A})}\cdot\frac{1}{(\alpha_a)(\alpha_a+\alpha_{a-1})...(\alpha_a+...+\alpha_1)}\cdot$$
$$\cdot\prod_{j\in \{1,...,|y| \} \textbackslash i}\left( j-i\right)\frac{1}{(\beta_{B}-i)(\beta_{B}+\beta_{B-1}-i)...(\beta_{B}+...+\beta_{b+2}-i)}\cdot$$
$$\cdot\frac{1}{(\beta_{B}+...+\beta_{b}-i)(\beta_{B}+...+\beta_{b-1}-i)...(\beta_{B}+...+\beta_{1}-i)}=$$
$$=\frac{1}{(-\alpha_{a+1})(-\alpha_{a+1}-\alpha_{a+2})...(-\alpha_{a+1}-...-\alpha_{A})}\cdot\frac{1}{(\alpha_a)(\alpha_a+\alpha_{a-1})...(\alpha_a+...+\alpha_1)}\cdot$$
$$\cdot\prod_{j\in \{1,...,|y| \} \textbackslash i}\left( j-i\right)\frac{1}{(-\beta_{B-1}-...-\beta_{b+1})(-\beta_{B-2}-...-\beta_{b+1})...(-\beta_{b+1})}\cdot$$
$$\cdot\frac{1}{(\beta_{b})(\beta_{b}+\beta_{b-1})...(\beta_{b}+...+\beta_{1})}=$$
$$=\frac{1}{(-\alpha_{a+1})(-\alpha_{a+1}-\alpha_{a+2})...(-\alpha_{a+1}-...-\alpha_{A})}\cdot\frac{1}{(\alpha_a)(\alpha_a+\alpha_{a-1})...(\alpha_a+...+\alpha_1)}\cdot$$
$$\cdot\prod_{j\in \{0,1,...,|y| \} \textbackslash i}\left( j-i\right)\frac{1}{(-\beta_{B}-...-\beta_{b+1})(-\beta_{B-1}-...-\beta_{b+1})(-\beta_{B-2}-...-\beta_{b+1})...(-\beta_{b+1})}\cdot$$
$$\cdot\frac{1}{(\beta_{b})(\beta_{b}+\beta_{b-1})...(\beta_{b}+...+\beta_{1})}=$$
$$=\frac{1}{(-\beta_{b+1})(-\beta_{b+1}-\beta_{b+2})...(-\beta_{b+1}-...-\beta_{B})}\cdot\frac{1}{(\beta_b)(\beta_b+\beta_{b-1})...(\beta_b+...+\beta_1)}\cdot$$
$$\cdot\prod_{j\in \{1,...,|x| \} \textbackslash i}\left( j-i\right)\frac{1}{(-\alpha_{A-1}-...-\alpha_{a+1})(-\alpha_{A-2}-...-\alpha_{a+1})...(-\alpha_{a+1})}\cdot$$
$$\cdot\frac{1}{(\alpha_{a})(\alpha_{a}+\alpha_{a-1})...(\alpha_{a}+...+\alpha_{1})}=$$
$$=\frac{1}{(-\beta_{b+1})(-\beta_{b+1}-\beta_{b+2})...(-\beta_{b+1}-...-\beta_{B})}\cdot\frac{1}{(\beta_b)(\beta_b+\beta_{b-1})...(\beta_b+...+\beta_1)}\cdot$$
$$\cdot\prod_{j\in \{1,...,|x| \} \textbackslash i}\left( j-i\right)\frac{1}{(\alpha_{A}-i)(\alpha_{A}+\alpha_{A-1}-i)...(\alpha_{A}+...+\alpha_{a+2}-i)}\cdot$$
$$\cdot\frac{1}{(\alpha_{A}+...+\alpha_{a}-i)(\alpha_{A}+...+\alpha_{a-1}-i)...(\alpha_{A}+...+\alpha_{1}-i)}=$$
$$=(\text{По Утверждению \ref{drob}})=$$
$$=\frac{1}{(-\beta_{b+1})(-\beta_{b+1}-\beta_{b+2})...(-\beta_{b+1}-...-\beta_{B})}\cdot\frac{1}{(\beta_b)(\beta_b+\beta_{b-1})...(\beta_b+...+\beta_1)}\prod_{j=1}^{d(x)}\left(g\left(x,j\right)-i\right)=$$
$$=f(y,i,0)\prod_{j=1}^{d(x)}\left(g\left(x,j\right)-i\right).$$
Что и требовалось.

\item Пусть $i\ne 0$, и $x$ представляется в виде $x=\alpha_1...\alpha_a\alpha_{a+1}...\alpha_A$, где $|\alpha_{a+1}...\alpha_A|=i, \; \alpha_j \in \{1,2 \}$, а $y$ не представляется в виде $y=\beta_1...\beta_b\beta_{b+1}...\beta_B$, где $|\beta_{b+1}...\beta_B|=i, \; \beta_j \in \{1,2 \}$. 

Тогда
$$\prod_{j=1}^{d(y)}\left(g\left(y,j\right)-i\right)=0,$$
так как в данном случае $y$ представляется в виде $y=\beta_1...\beta_b2\beta_{b+1}...\beta_B$, где $|\beta_{b+1}...\beta_B|=i-1, \; \beta_j \in \{1,2 \}$. А кроме того
$${f\left(y,i,0\right)}=0$$
по определению функции $f$.
Значит,
$${f\left(x,i,0\right)}\prod_{j=1}^{d(y)}\left(g\left(y,j\right)-i\right)=0={f\left(y,i,0\right)}\prod_{j=1}^{d(x)}\left(g\left(x,j\right)-i\right).$$

\item Пусть $i\ne 0$, и $x$ не представляется в виде $x=\alpha_1...\alpha_a\alpha_{a+1}...\alpha_A$, где $|\alpha_{a+1}...\alpha_A|=i, \; \alpha_j \in \{1,2 \}$, а $y$ представляется в виде $y=\beta_1...\beta_b\beta_{b+1}...\beta_B$, где $|\beta_{b+1}...\beta_B|=i, \; \beta_j \in \{1,2 \}$.

Этот случай полностью аналогичен предыдущему случаю, так как мы можем поменять $x$ и $y$ местами.

\item Пусть $i \ne 0$, и $x$ не представляется в виде $x=\alpha_1...\alpha_a\alpha_{a+1}...\alpha_A$, где $|\alpha_{a+1}...\alpha_A|=i,\; \alpha_j \in \{1,2 \}$, а $y$ тоже не представляется в виде $y=\beta_1...\beta_b\beta_{b+1}...\beta_B$, где $|\beta_{b+1}...\beta_B|=i, \; \beta_j \in \{1,2 \}$. 

В данном случае
$${f\left(x,i,0\right)}=0={f\left(y,i,0\right)}$$
по определению функции $f$.
Значит,
$${f\left(x,i,0\right)}\prod_{j=1}^{d(y)}\left(g\left(y,j\right)-i\right)=0={f\left(y,i,0\right)}\prod_{j=1}^{d(x)}\left(g\left(x,j\right)-i\right).$$
\end{enumerate}

Все случаи разобраны.

Лемма доказана.
\end{proof}

\begin{Lemma} \label{xyz}
Пусть $x,y \in \mathbb{YF}:$ $|x|=|y|$, $z=h(x,y)=h(y,x)$. Тогда $\forall i \in\{1,...,|x|\}$
$$ {f\left(x,i,z\right)}\prod_{j=1}^{d(y)}\left(g\left(y,j\right)-i\right)={f\left(y,i,z\right)}\prod_{j=1}^{d(x)}\left(g\left(x,j\right)-i\right).$$ 
\end{Lemma}
\begin{proof}

Докажем Лемму по индукции по $z$. 

\underline{\textbf{База}}: $z=0$ -- верно по Лемме \ref{xyz0}.

\underline{\textbf{Переход}} к $z+1\ge 1$.

Рассмотрим два случая:
\renewcommand{\labelenumi}{\roman{enumi}$^\circ$}
\renewcommand{\labelenumii}{\roman{enumii}$^\circ$}
\begin{enumerate}
    \item Пусть $\exists x',y':x=x'1,y=y'1$. 

В данном случае по рекурсивному определению функции $f$ 
$${f\left(x,i,z+1\right)}\prod_{j=1}^{d(y)}\left(g\left(y,j\right)-i\right)={f\left(x'1,i,z+1\right)}\prod_{j=1}^{d(y)}\left(g\left(y,j\right)-i\right)=$$
$$=\left({f\left(x'1,i,0\right)}+{f\left(x',i-1,z\right)}\right)\prod_{j=1}^{d(y)}\left(g\left(y,j\right)-i\right);$$
$${f\left(y,i,z+1\right)}\prod_{j=1}^{d(x)}\left(g\left(x,j\right)-i\right)={f\left(y'1,i,z+1\right)}\prod_{j=1}^{d(x)}\left(g\left(x,j\right)-i\right)=$$
$$=\left({f\left(y'1,i,0\right)}+{f\left(y',i-1,z\right)}\right)\prod_{j=1}^{d(x)}\left(g\left(x,j\right)-i\right).$$
Ясно, что $d(x)=d(x')$ и $d(y)=d(y')$, а значит по Лемме \ref{xyz0}
$${f\left(x'1,i,0\right)}\prod_{j=1}^{d(y)}\left(g\left(y,j\right)-i\right)={f\left(y'1,i,0\right)}\prod_{j=1}^{d(x)}\left(g\left(x,j\right)-i\right);$$
Кроме того, ясно, что если $h(x,y)=z+1$, то $h(x',y')=z$, а поэтому можно воспользоваться предположением индукции. Пользуемся:
$${f\left(x',i-1,z\right)}\prod_{j=1}^{d(y)}\left(g\left(y,j\right)-i\right)={f\left(y',i-1,z\right)}\prod_{j=1}^{d(x)}\left(g\left(x,j\right)-i\right).$$

Сложив два последних на данный момент равенства, получаем следующее:
$$\left({f\left(x'1,i,0\right)}+{f\left(x',i-1,z\right)}\right)\prod_{j=1}^{d(y)}\left(g\left(y,j\right)-i\right)=\left({f\left(y'1,i,0\right)}+{f\left(y',i-1,z\right)}\right)\prod_{j=1}^{d(x)}\left(g\left(x,j\right)-i\right) \Longleftrightarrow$$
$$\Longleftrightarrow {f\left(x,i,z+1\right)}\prod_{j=1}^{d(y)}\left(g\left(y,j\right)-i\right)={f\left(y,i,z+1\right)}\prod_{j=1}^{d(x)}\left(g\left(x,j\right)-i\right).$$
Что и требовалось.

    \item Пусть $\exists x',y':x=x'2,y=y'2$. 

Рассмотрим два подслучая:
\renewcommand{\labelenumiii}{\roman{enumii}.\roman{enumiii}$^\circ$}
\begin{enumerate}
    \item Пусть $i\ne 1$.

В данном случае по рекурсивному определению функции $f$
$${f\left(x,i,z+1\right)}\prod_{j=1}^{d(y)}\left(g\left(y,j\right)-i\right)={f\left(x'2,i,z+1\right)}\prod_{j=1}^{d(y)}\left(g\left(y,j\right)-i\right)=$$
$$=\frac{f\left(x'11,i,z+2\right)}{1-i}\prod_{j=1}^{d(y)}\left(g\left(y,j\right)-i\right)=\frac{f\left(x'11,i,z+2\right)}{1-i}\left(\prod_{j=1}^{d(y'11)}\left(g\left(y'11,j\right)-i\right)\right)(1-i)=$$
$$={f\left(x'11,i,z+2\right)}\prod_{j=1}^{d(y'11)}\left(g\left(y'11,j\right)-i\right);$$
$${f\left(y,i,z+1\right)}\prod_{j=1}^{d(x)}\left(g\left(x,j\right)-i\right)={f\left(y'2,i,z+1\right)}\prod_{j=1}^{d(x)}\left(g\left(x,j\right)-i\right)=$$
$$=\frac{f\left(y'11,i,z+2\right)}{1-i}\prod_{j=1}^{d(x)}\left(g\left(x,j\right)-i\right)=\frac{f\left(y'11,i,z+2\right)}{1-i}\left(\prod_{j=1}^{d(x'11)}\left(g\left(x'11,j\right)-i\right)\right)(1-i)=$$
$$={f\left(y'11,i,z+2\right)}\prod_{j=1}^{d(x'11)}\left(g\left(x'11,j\right)-i\right).$$

Таким образом,
$${f\left(x,i,z+1\right)}\prod_{j=1}^{d(y)}\left(g\left(y,j\right)-i\right)={f\left(y,i,z+1\right)}\prod_{j=1}^{d(x)}\left(g\left(x,j\right)-i\right) \Longleftrightarrow$$
$$\Longleftrightarrow {f\left(x'11,i,z+2\right)}\prod_{j=1}^{d(y'11)}\left(g\left(y'11,j\right)-i\right)={f\left(y'11,i,z+2\right)}\prod_{j=1}^{d(x'11)}\left(g\left(x'11,j\right)-i\right).$$
Кроме того ясно, что если $h(x,y)=z+1$, то $h(x'11,y'11)=z+2$, а поэтому можно воспользоваться предыдущим случаем. Воспользовавшись, понимаем, что для доказательства этого равенства достаточно доказать, что
$${f\left(x'1,i,z+1\right)}\prod_{j=1}^{d(y'1)}\left(g\left(y'1,j\right)-i\right)={f\left(y'1,i,z+1\right)}\prod_{j=1}^{d(x'1)}\left(g\left(x'1,j\right)-i\right).$$
Также ясно, что если $h(x'11,y'11)=z+2$, то $h(x'1,y'1)=z+1$, а поэтому можно снова воспользоваться предыдущим случаем. Воспользовавшись, понимаем, что для доказательства этого равенства достаточно доказать, что
$${f\left(x',i,z\right)}\prod_{j=1}^{d(y')}\left(g\left(y',j\right)-i\right)={f\left(y',i,z\right)}\prod_{j=1}^{d(x')}\left(g\left(x',j\right)-i\right).$$
И наконец, ясно что если $h(x'1,y'1)=z+1$, то $h(x',y')=z$, а поэтому можно воспользоваться предположением индукции. Воспользовавшись, понимаем, что в данном случае переход доказан.

\item Пусть $i = 1$. 

В данном случае по Утверждению \ref{y01}
$${f\left(x,1,z+1\right)}=0={f\left(y,1,z+1\right)}.$$

А значит,
$${f\left(x,1,z+1\right)}\prod_{j=1}^{d(y)}\left(g\left(y,j\right)-1\right)=0={f\left(y,1,z+1\right)}\prod_{j=1}^{d(x)}\left(g\left(x,j\right)-1\right).$$
Что и требовалось.

Разобраны все случаи, так как если $x$ и $y$ заканчиваются на разные цифры, то $z=0$, а этот случай был разобран в базе.
 
\end{enumerate}
Лемма доказана.
\end{enumerate}

\end{proof}

Итак, продолжим доказывать Теорему, разбирая случаи:

\item Пусть $|y|=|x|$ и $\delta_1=-1$.

Пусть $z=h(x,y)$.

В данном случае мы делаем $s(\Delta)$ шагов ``вверх'' и $s(\Delta)$ шагов ``вниз''. Тогда давайте смотреть на траекторию $\Delta$ ``наоборот'':

\begin{tikzpicture}

\draw[ultra thick, ->] (0.5,0) -- (0.5,7);
\draw[thin, dashed, ->] (1.5,0) -- (1.5,7);
\draw[thin, dashed, ->] (2.5,0) -- (2.5,7);
\draw[thin, dashed, ->] (3.5,0) -- (3.5,7);
\draw[thin, dashed, ->] (4.5,0) -- (4.5,7);
\draw[thin, dashed, ->] (5.5,0) -- (5.5,7);
\draw[thin, dashed, ->] (6.5,0) -- (6.5,7);
\draw[thin, dashed, ->] (7.5,0) -- (7.5,7);
\draw[thin, dashed, ->] (8.5,0) -- (8.5,7);
\draw[thin, dashed, ->] (9.5,0) -- (9.5,7);
\draw[thin, dashed, ->] (10.5,0) -- (10.5,7);
\draw[thin, dashed, ->] (11.5,0) -- (11.5,7);
\draw[thin, dashed, ->] (12.5,0) -- (12.5,7);
\draw[thin, dashed, ->] (13.5,0) -- (13.5,7);
\draw[thin, dashed, ->] (14.5,0) -- (14.5,7);

\draw[ultra thick, ->] (0,0.5) -- (15,0.5);
\draw[thin, dashed, ->] (0,1.5) -- (15,1.5);
\draw[thin, dashed, ->] (0,2.5) -- (15,2.5);
\draw[thin, dashed, ->] (0,3.5) -- (15,3.5);
\draw[thin, dashed, ->] (0,4.5) -- (15,4.5);
\draw[thin, dashed, ->] (0,5.5) -- (15,5.5);
\draw[thin, dashed, ->] (0,6.5) -- (15,6.5);

\draw[ultra thick, red, ->] (0.5,4.5) -- (1.5,3.5);
\draw[ultra thick, red, ->] (1.5,3.5) -- (2.5,2.5);
\draw[ultra thick, red, ->] (2.5,2.5) -- (3.5,3.5);
\draw[ultra thick, red, ->] (3.5,3.5) -- (4.5,2.5);
\draw[ultra thick, red, ->] (4.5,2.5) -- (5.5,1.5);
\draw[ultra thick, red, ->] (5.5,1.5) -- (6.5,0.5);
\draw[ultra thick, red, ->] (6.5,0.5) -- (7.5,1.5);
\draw[ultra thick, red, ->] (7.5,1.5) -- (8.5,2.5);
\draw[ultra thick, red, ->] (8.5,2.5) -- (9.5,3.5);
\draw[ultra thick, red, ->] (9.5,3.5) -- (10.5,4.5);
\draw[ultra thick, red, ->] (10.5,4.5) -- (11.5,5.5);
\draw[ultra thick, red, ->] (11.5,5.5) -- (12.5,6.5);
\draw[ultra thick, red, ->] (12.5,6.5) -- (13.5,5.5);
\draw[ultra thick, red, ->] (13.5,5.5) -- (14.5,4.5);

\path (0.25,0.25) node  {$0$}
-- (1.25,0.25)  node  {$1$}
-- (2.25,0.25)  node  {$2$}
-- (3.25,0.25)  node  {$3$}
-- (4.25,0.25)  node  {$4$}
-- (5.25,0.25)  node  {$5$}
-- (6.25,0.25)  node  {$6$}
-- (7.25,0.25)  node  {$7$}
-- (8.25,0.25)  node  {$8$}
-- (9.25,0.25)  node  {$9$}
-- (10.25,0.25)  node  {$10$}
-- (11.25,0.25)  node  {$11$}
-- (12.25,0.25)  node  {$12$}
-- (13.25,0.25)  node  {$13$}
-- (14.25,0.25)  node  {$14$};

\path (0.25,1.25)  node  {$1$}
-- (0.25,2.25)  node  {$2$}
-- (0.25,3.25)  node  {$3$}
-- (0.25,4.25)  node  {$4$}
-- (0.25,5.25)  node  {$5$}
-- (0.25,6.25)  node  {$6$};
\end{tikzpicture}

$$$$

\begin{tikzpicture}

\draw[thin, dashed, ->] (0.5,0) -- (0.5,7);
\draw[thin, dashed, ->] (1.5,0) -- (1.5,7);
\draw[thin, dashed, ->] (2.5,0) -- (2.5,7);
\draw[thin, dashed, ->] (3.5,0) -- (3.5,7);
\draw[thin, dashed, ->] (4.5,0) -- (4.5,7);
\draw[thin, dashed, ->] (5.5,0) -- (5.5,7);
\draw[thin, dashed, ->] (6.5,0) -- (6.5,7);
\draw[thin, dashed, ->] (7.5,0) -- (7.5,7);
\draw[thin, dashed, ->] (8.5,0) -- (8.5,7);
\draw[thin, dashed, ->] (9.5,0) -- (9.5,7);
\draw[thin, dashed, ->] (10.5,0) -- (10.5,7);
\draw[thin, dashed, ->] (11.5,0) -- (11.5,7);
\draw[thin, dashed, ->] (12.5,0) -- (12.5,7);
\draw[thin, dashed, ->] (13.5,0) -- (13.5,7);
\draw[ultra thick, ->] (14.5,0) -- (14.5,7);

\draw[ultra thick, <-] (0,0.5) -- (15,0.5);
\draw[thin, dashed, <-] (0,1.5) -- (15,1.5);
\draw[thin, dashed, <-] (0,2.5) -- (15,2.5);
\draw[thin, dashed, <-] (0,3.5) -- (15,3.5);
\draw[thin, dashed, <-] (0,4.5) -- (15,4.5);
\draw[thin, dashed, <-] (0,5.5) -- (15,5.5);
\draw[thin, dashed, <-] (0,6.5) -- (15,6.5);

\draw[ultra thick, blue, <-] (0.5,4.5) -- (1.5,3.5);
\draw[ultra thick, blue, <-] (1.5,3.5) -- (2.5,2.5);
\draw[ultra thick, blue, <-] (2.5,2.5) -- (3.5,3.5);
\draw[ultra thick, blue, <-] (3.5,3.5) -- (4.5,2.5);
\draw[ultra thick, blue, <-] (4.5,2.5) -- (5.5,1.5);
\draw[ultra thick, blue, <-] (5.5,1.5) -- (6.5,0.5);
\draw[ultra thick, blue, <-] (6.5,0.5) -- (7.5,1.5);
\draw[ultra thick, blue, <-] (7.5,1.5) -- (8.5,2.5);
\draw[ultra thick, blue, <-] (8.5,2.5) -- (9.5,3.5);
\draw[ultra thick, blue, <-] (9.5,3.5) -- (10.5,4.5);
\draw[ultra thick, blue, <-] (10.5,4.5) -- (11.5,5.5);
\draw[ultra thick, blue, <-] (11.5,5.5) -- (12.5,6.5);
\draw[ultra thick, blue, <-] (12.5,6.5) -- (13.5,5.5);
\draw[ultra thick, blue, <-] (13.5,5.5) -- (14.5,4.5);

\path (0.75,0.25) node  {$14$}
-- (1.75,0.25)  node  {$13$}
-- (2.75,0.25)  node  {$12$}
-- (3.75,0.25)  node  {$11$}
-- (4.75,0.25)  node  {$10$}
-- (5.75,0.25)  node  {$9$}
-- (6.75,0.25)  node  {$8$}
-- (7.75,0.25)  node  {$7$}
-- (8.75,0.25)  node  {$6$}
-- (9.75,0.25)  node  {$5$}
-- (10.75,0.25)  node  {$4$}
-- (11.75,0.25)  node  {$3$}
-- (12.75,0.25)  node  {$2$}
-- (13.75,0.25)  node  {$1$}
-- (14.75,0.25)  node  {$0$};

\path (14.75,1.25)  node  {$1$}
-- (14.75,2.25)  node  {$2$}
-- (14.75,3.25)  node  {$3$}
-- (14.75,4.25)  node  {$4$}
-- (14.75,5.25)  node  {$5$}
-- (14.75,6.25)  node  {$6$};
\end{tikzpicture}

На рисунках изображено, как по траектории $\Delta$ (верхний рисунок) строится ``обратная'' траектория $\overline{\Delta}$ (нижний рисунок). Теперь давайте введём формальное определение ``обратной'' траектории к траектории $\Delta$:

\begin{Def}

Пусть $X\in \mathbb{N}_0$ и $\Delta\in \Omega(X,X)$. Тогда ``обратной'' к $\Delta$ траекторией назовём траекторию $\overline{\Delta}$, определённую следующим образом:
$$|\overline{\Delta}|=|\Delta| \text{ и } \forall i \in \{0,1,...,|\Delta|\} \; \overline{\Delta}_i=\Delta_{|\Delta|-i}.$$

\end{Def}
Сразу из определения следует:
\begin{Zam} 
Пусть $X\in \mathbb{N}_0$ и $\Delta\in \Omega(X,X)$. Тогда:
\begin{itemize}
    \item $ \overline{\Delta} \in \Omega(X,X)$;
    \item $s(\overline{\Delta})=s(\Delta)$;
    \item $\overline{\Delta}$ как и $\Delta$ состоит из $s(\Delta)$ шагов ``вверх'' и такого же числа шагов ``вниз'';
    \item $$\forall i \in \{1,...,|\Delta|\} \;\overline{\delta}_i=\overline{\Delta}_{i}-\overline{\Delta}_{i-1}=\Delta_{|\Delta|-i}-\Delta_{|\Delta|-i+1}=-\delta_{|\Delta|-i+1}.$$
\end{itemize}
\end{Zam}

Кроме того, докажем следующее утверждение:

\begin{Prop} \label{obrat}
Пусть $X\in \mathbb{N}_0$ и $\Delta\in \Omega(X,X)$. Тогда
$$E(\Delta)=E\left(\overline{\Delta}\right).$$
\end{Prop}

\begin{proof}
Давайте докажем Утверждение для всех $\Delta\in\Omega(X,X)$ по индукции по $s(\Delta)$:

\underline{\textbf{База}}: $s(\Delta)=0$ -- очевидно, так как в данном случае $E(\Delta)=E(\overline{\Delta})= \o$.

\underline{\textbf{Переход}} к $\Delta:$ $s(\Delta)\ge 1$.

Пусть траектория ${\Delta}\in\Omega(X,X)$ не содержит двух соседних шагов, среди которых один шаг -- это шаг ``вверх'', а другой шаг -- это шаг ``вниз''. Ясно, что тогда либо все шаги траектории $\Delta$ -- это шаги ``вниз'', либо все шаги траектории $\Delta$ -- это шаги ``вверх''. Но $s(\Delta)\ge 1$, а значит, $|\Delta|\ge 1$, а поэтому $\Delta \notin \Omega(X,X)$, и мы приходим к противоречию.  

Значит, каждая из траекторий $\Delta$ и $\overline{\Delta}$ содержит пару соседних шагов, среди которых один шаг -- это шаг ``вверх'', а другой шаг -- это шаг ``вниз''.

Тогда давайте введём $\Delta'$ и ${\Delta''}$ следующим образом:

\begin{tikzpicture}

\draw[ultra thick, ->] (0.5,0) -- (0.5,7);
\draw[ultra thick, red, ->] (1.5,0) -- (1.5,7);
\draw[thin, dashed, ->] (2.5,0) -- (2.5,7);
\draw[ultra thick, red, ->] (3.5,0) -- (3.5,7);
\draw[thin, dashed, ->] (4.5,0) -- (4.5,7);
\draw[thin, dashed, ->] (5.5,0) -- (5.5,7);
\draw[thin, dashed, ->] (6.5,0) -- (6.5,7);
\draw[thin, dashed, ->] (7.5,0) -- (7.5,7);
\draw[thin, dashed, ->] (8.5,0) -- (8.5,7);
\draw[thin, dashed, ->] (9.5,0) -- (9.5,7);
\draw[thin, dashed, ->] (10.5,0) -- (10.5,7);
\draw[thin, dashed, ->] (11.5,0) -- (11.5,7);
\draw[thin, dashed, ->] (12.5,0) -- (12.5,7);
\draw[thin, dashed, ->] (13.5,0) -- (13.5,7);
\draw[thin, dashed, ->] (14.5,0) -- (14.5,7); 

\draw[ultra thick, ->] (0,0.5) -- (15,0.5);
\draw[thin, dashed, ->] (0,1.5) -- (15,1.5);
\draw[thin, dashed, ->] (0,2.5) -- (15,2.5);
\draw[thin, dashed, ->] (0,3.5) -- (15,3.5);
\draw[thin, dashed, ->] (0,4.5) -- (15,4.5);
\draw[thin, dashed, ->] (0,5.5) -- (15,5.5);
\draw[thin, dashed, ->] (0,6.5) -- (15,6.5);

\draw[ultra thick, red, ->] (0.5,4.5) -- (1.5,3.5);
\draw[ultra thick, red, ->] (1.5,3.5) -- (2.5,2.5);
\draw[ultra thick, red, ->] (2.5,2.5) -- (3.5,3.5);
\draw[ultra thick, red, ->] (3.5,3.5) -- (4.5,2.5);
\draw[ultra thick, red, ->] (4.5,2.5) -- (5.5,1.5);
\draw[ultra thick, red, ->] (5.5,1.5) -- (6.5,0.5);
\draw[ultra thick, red, ->] (6.5,0.5) -- (7.5,1.5);
\draw[ultra thick, red, ->] (7.5,1.5) -- (8.5,2.5);
\draw[ultra thick, red, ->] (8.5,2.5) -- (9.5,3.5);
\draw[ultra thick, red, ->] (9.5,3.5) -- (10.5,4.5);
\draw[ultra thick, red, ->] (10.5,4.5) -- (11.5,5.5);
\draw[ultra thick, red, ->] (11.5,5.5) -- (12.5,6.5);
\draw[ultra thick, red, ->] (12.5,6.5) -- (13.5,5.5);
\draw[ultra thick, red, ->] (13.5,5.5) -- (14.5,4.5);

\path (0.25,0.25) node  {$0$}
-- (1.25,0.25)  node  {$1$}
-- (2.25,0.25)  node  {$2$}
-- (3.25,0.25)  node  {$3$}
-- (4.25,0.25)  node  {$4$}
-- (5.25,0.25)  node  {$5$}
-- (6.25,0.25)  node  {$6$}
-- (7.25,0.25)  node  {$7$}
-- (8.25,0.25)  node  {$8$}
-- (9.25,0.25)  node  {$9$}
-- (10.25,0.25)  node  {$10$}
-- (11.25,0.25)  node  {$11$}
-- (12.25,0.25)  node  {$12$}
-- (13.25,0.25)  node  {$13$}
-- (14.25,0.25)  node  {$14$};

\path (0.25,1.25)  node  {$1$}
-- (0.25,2.25)  node  {$2$}
-- (0.25,3.25)  node  {$3$}
-- (0.25,4.25)  node  {$4$}
-- (0.25,5.25)  node  {$5$}
-- (0.25,6.25)  node  {$6$};
\end{tikzpicture}

$$$$

\begin{tikzpicture}

\draw[thin, dashed, ->] (0.5,0) -- (0.5,7);
\draw[ultra thick, blue,  ->] (1.5,0) -- (1.5,7);
\draw[thin, dashed, ->] (2.5,0) -- (2.5,7);
\draw[ultra thick, blue, ->] (3.5,0) -- (3.5,7);
\draw[thin, dashed, ->] (4.5,0) -- (4.5,7);
\draw[thin, dashed, ->] (5.5,0) -- (5.5,7);
\draw[thin, dashed, ->] (6.5,0) -- (6.5,7);
\draw[thin, dashed, ->] (7.5,0) -- (7.5,7);
\draw[thin, dashed, ->] (8.5,0) -- (8.5,7);
\draw[thin, dashed, ->] (9.5,0) -- (9.5,7);
\draw[thin, dashed, ->] (10.5,0) -- (10.5,7);
\draw[thin, dashed, ->] (11.5,0) -- (11.5,7);
\draw[thin, dashed, ->] (12.5,0) -- (12.5,7);
\draw[thin, dashed, ->] (13.5,0) -- (13.5,7);
\draw[ultra thick, ->] (14.5,0) -- (14.5,7);

\draw[ultra thick, <-] (0,0.5) -- (15,0.5);
\draw[thin, dashed, <-] (0,1.5) -- (15,1.5);
\draw[thin, dashed, <-] (0,2.5) -- (15,2.5);
\draw[thin, dashed, <-] (0,3.5) -- (15,3.5);
\draw[thin, dashed, <-] (0,4.5) -- (15,4.5);
\draw[thin, dashed, <-] (0,5.5) -- (15,5.5);
\draw[thin, dashed, <-] (0,6.5) -- (15,6.5);

\draw[ultra thick, blue, <-] (0.5,4.5) -- (1.5,3.5);
\draw[ultra thick, blue, <-] (1.5,3.5) -- (2.5,2.5);
\draw[ultra thick, blue, <-] (2.5,2.5) -- (3.5,3.5);
\draw[ultra thick, blue, <-] (3.5,3.5) -- (4.5,2.5);
\draw[ultra thick, blue, <-] (4.5,2.5) -- (5.5,1.5);
\draw[ultra thick, blue, <-] (5.5,1.5) -- (6.5,0.5);
\draw[ultra thick, blue, <-] (6.5,0.5) -- (7.5,1.5);
\draw[ultra thick, blue, <-] (7.5,1.5) -- (8.5,2.5);
\draw[ultra thick, blue, <-] (8.5,2.5) -- (9.5,3.5);
\draw[ultra thick, blue, <-] (9.5,3.5) -- (10.5,4.5);
\draw[ultra thick, blue, <-] (10.5,4.5) -- (11.5,5.5);
\draw[ultra thick, blue, <-] (11.5,5.5) -- (12.5,6.5);
\draw[ultra thick, blue, <-] (12.5,6.5) -- (13.5,5.5);
\draw[ultra thick, blue, <-] (13.5,5.5) -- (14.5,4.5);

\path (0.75,0.25) node  {$14$}
-- (1.75,0.25)  node  {$13$}
-- (2.75,0.25)  node  {$12$}
-- (3.75,0.25)  node  {$11$}
-- (4.75,0.25)  node  {$10$}
-- (5.75,0.25)  node  {$9$}
-- (6.75,0.25)  node  {$8$}
-- (7.75,0.25)  node  {$7$}
-- (8.75,0.25)  node  {$6$}
-- (9.75,0.25)  node  {$5$}
-- (10.75,0.25)  node  {$4$}
-- (11.75,0.25)  node  {$3$}
-- (12.75,0.25)  node  {$2$}
-- (13.75,0.25)  node  {$1$}
-- (14.75,0.25)  node  {$0$};

\path (14.75,1.25)  node  {$1$}
-- (14.75,2.25)  node  {$2$}
-- (14.75,3.25)  node  {$3$}
-- (14.75,4.25)  node  {$4$}
-- (14.75,5.25)  node  {$5$}
-- (14.75,6.25)  node  {$6$};
\end{tikzpicture}

$$$$

\begin{tikzpicture}

\draw[ultra thick, ->] (0.5,0) -- (0.5,7);
\draw[thin, dashed, ->] (1.5,0) -- (1.5,7);
\draw[thin, dashed, ->] (2.5,0) -- (2.5,7);
\draw[thin, dashed, ->] (3.5,0) -- (3.5,7);
\draw[thin, dashed, ->] (4.5,0) -- (4.5,7);
\draw[thin, dashed, ->] (5.5,0) -- (5.5,7);
\draw[thin, dashed, ->] (6.5,0) -- (6.5,7);
\draw[thin, dashed, ->] (7.5,0) -- (7.5,7);
\draw[thin, dashed, ->] (8.5,0) -- (8.5,7);
\draw[thin, dashed, ->] (9.5,0) -- (9.5,7);
\draw[thin, dashed, ->] (10.5,0) -- (10.5,7);
\draw[thin, dashed, ->] (11.5,0) -- (11.5,7);
\draw[thin, dashed, ->] (12.5,0) -- (12.5,7);
\draw[thin, dashed, ->] (13.5,0) -- (13.5,7);
\draw[thin, dashed, ->] (14.5,0) -- (14.5,7); 

\draw[ultra thick, ->] (0,0.5) -- (15,0.5);
\draw[thin, dashed, ->] (0,1.5) -- (15,1.5);
\draw[thin, dashed, ->] (0,2.5) -- (15,2.5);
\draw[thin, dashed, ->] (0,3.5) -- (15,3.5);
\draw[thin, dashed, ->] (0,4.5) -- (15,4.5);
\draw[thin, dashed, ->] (0,5.5) -- (15,5.5);
\draw[thin, dashed, ->] (0,6.5) -- (15,6.5);

\draw[ultra thick, red, ->] (0.5,4.5) -- (1.5,3.5);
\draw[ultra thick, red, ->] (1.5,3.5) -- (2.5,2.5);
\draw[ultra thick, red, ->] (2.5,2.5) -- (3.5,1.5);
\draw[ultra thick, red, ->] (3.5,1.5) -- (4.5,0.5);
\draw[ultra thick, red, ->] (4.5,0.5) -- (5.5,1.5);
\draw[ultra thick, red, ->] (5.5,1.5) -- (6.5,2.5);
\draw[ultra thick, red, ->] (6.5,2.5) -- (7.5,3.5);
\draw[ultra thick, red, ->] (7.5,3.5) -- (8.5,4.5);
\draw[ultra thick, red, ->] (8.5,4.5) -- (9.5,5.5);
\draw[ultra thick, red, ->] (9.5,5.5) -- (10.5,6.5);
\draw[ultra thick, red, ->] (10.5,6.5) -- (11.5,5.5);
\draw[ultra thick, red, ->] (11.5,5.5) -- (12.5,4.5);

\path (0.25,0.25) node  {$0$}
-- (1.25,0.25)  node  {$1$}
-- (2.25,0.25)  node  {$2$}
-- (3.25,0.25)  node  {$3$}
-- (4.25,0.25)  node  {$4$}
-- (5.25,0.25)  node  {$5$}
-- (6.25,0.25)  node  {$6$}
-- (7.25,0.25)  node  {$7$}
-- (8.25,0.25)  node  {$8$}
-- (9.25,0.25)  node  {$9$}
-- (10.25,0.25)  node  {$10$}
-- (11.25,0.25)  node  {$11$}
-- (12.25,0.25)  node  {$12$}
-- (13.25,0.25)  node  {$13$}
-- (14.25,0.25)  node  {$14$};

\path (0.25,1.25)  node  {$1$}
-- (0.25,2.25)  node  {$2$}
-- (0.25,3.25)  node  {$3$}
-- (0.25,4.25)  node  {$4$}
-- (0.25,5.25)  node  {$5$}
-- (0.25,6.25)  node  {$6$};
\end{tikzpicture}

$$$$

\begin{tikzpicture}

\draw[thin, dashed, ->] (0.5,0) -- (0.5,7);
\draw[thin, dashed, ->] (1.5,0) -- (1.5,7);
\draw[thin, dashed, ->] (2.5,0) -- (2.5,7);
\draw[thin, dashed, ->] (3.5,0) -- (3.5,7);
\draw[thin, dashed, ->] (4.5,0) -- (4.5,7);
\draw[thin, dashed, ->] (5.5,0) -- (5.5,7);
\draw[thin, dashed, ->] (6.5,0) -- (6.5,7);
\draw[thin, dashed, ->] (7.5,0) -- (7.5,7);
\draw[thin, dashed, ->] (8.5,0) -- (8.5,7);
\draw[thin, dashed, ->] (9.5,0) -- (9.5,7);
\draw[thin, dashed, ->] (10.5,0) -- (10.5,7);
\draw[thin, dashed, ->] (11.5,0) -- (11.5,7);
\draw[thin, dashed, ->] (12.5,0) -- (12.5,7);
\draw[thin, dashed, ->] (13.5,0) -- (13.5,7);
\draw[ultra thick, ->] (14.5,0) -- (14.5,7);

\draw[ultra thick, <-] (0,0.5) -- (15,0.5);
\draw[thin, dashed, <-] (0,1.5) -- (15,1.5);
\draw[thin, dashed, <-] (0,2.5) -- (15,2.5);
\draw[thin, dashed, <-] (0,3.5) -- (15,3.5);
\draw[thin, dashed, <-] (0,4.5) -- (15,4.5);
\draw[thin, dashed, <-] (0,5.5) -- (15,5.5);
\draw[thin, dashed, <-] (0,6.5) -- (15,6.5);

\draw[ultra thick, blue, <-] (2.5,4.5) -- (3.5,3.5);
\draw[ultra thick, blue, <-] (3.5,3.5) -- (4.5,2.5);
\draw[ultra thick, blue, <-] (4.5,2.5) -- (5.5,1.5);
\draw[ultra thick, blue, <-] (5.5,1.5) -- (6.5,0.5);
\draw[ultra thick, blue, <-] (6.5,0.5) -- (7.5,1.5);
\draw[ultra thick, blue, <-] (7.5,1.5) -- (8.5,2.5);
\draw[ultra thick, blue, <-] (8.5,2.5) -- (9.5,3.5);
\draw[ultra thick, blue, <-] (9.5,3.5) -- (10.5,4.5);
\draw[ultra thick, blue, <-] (10.5,4.5) -- (11.5,5.5);
\draw[ultra thick, blue, <-] (11.5,5.5) -- (12.5,6.5);
\draw[ultra thick, blue, <-] (12.5,6.5) -- (13.5,5.5);
\draw[ultra thick, blue, <-] (13.5,5.5) -- (14.5,4.5);

\path (0.75,0.25) node  {$14$}
-- (1.75,0.25)  node  {$13$}
-- (2.75,0.25)  node  {$12$}
-- (3.75,0.25)  node  {$11$}
-- (4.75,0.25)  node  {$10$}
-- (5.75,0.25)  node  {$9$}
-- (6.75,0.25)  node  {$8$}
-- (7.75,0.25)  node  {$7$}
-- (8.75,0.25)  node  {$6$}
-- (9.75,0.25)  node  {$5$}
-- (10.75,0.25)  node  {$4$}
-- (11.75,0.25)  node  {$3$}
-- (12.75,0.25)  node  {$2$}
-- (13.75,0.25)  node  {$1$}
-- (14.75,0.25)  node  {$0$};

\path (14.75,1.25)  node  {$1$}
-- (14.75,2.25)  node  {$2$}
-- (14.75,3.25)  node  {$3$}
-- (14.75,4.25)  node  {$4$}
-- (14.75,5.25)  node  {$5$}
-- (14.75,6.25)  node  {$6$};
\end{tikzpicture}

На рисунках изображено, как по траекториям $\Delta$ и $\overline{\Delta}$ (верхние два рисунка) строятся траектории $\Delta'$ и $\Delta''$ (нижние два рисунка). В данном случае мы берём два соседних шага траектории, среди которых один шаг -- это шаг ``вверх'', а другой шаг -- это шаг ``вниз'', и просто убираем их. Формально:

Мы уже убедились в том, что существует $r\in\{1,...,|\Delta|-2\}$: $\delta_{r+1}\ne \delta_{r+2}$. Ясно, если $r$ удовлетворяет данному условию, то $\overline{\delta}_{|\Delta|-r-1}\ne\overline{\delta}_{|\Delta|-r}$, причём $\delta_{r+1}=\overline{\delta}_{|\Delta|-r}$ и $\delta_{r+2}=\overline{\delta}_{|\Delta|-r-1}$. Введём $\Delta'$ следующим образом:
$$|\Delta'|=|\Delta|-2; \text{ при } i\in \{0,1,...,r\} \; \Delta'_i=\Delta_i; \text{ при } i\in \{r+1,...,|\Delta'|\} \; \Delta'_i=\Delta_{i+2},$$
а $\Delta''$ -- следующим:
$$|\Delta''|=\left|\overline{\Delta}\right|-2; \text{ при } i \in \{0,1,...,|\Delta''|-r-1\} \; \Delta''_i=\overline{\Delta}_i; \text{ при } i \in \{|\Delta''|-r,...,|\Delta''|\} \; \Delta''_i=\overline{\Delta}_{i+2}.$$

Несложно заметить, что 
\begin{itemize}
    \item $ \Delta',\Delta''\in \Omega(X,X)$;
    \item $ |\Delta'|=|\Delta''|$;
    \item $s(\Delta')=s(\Delta'')=s(\Delta)-1=s\left(\overline{\Delta}\right)-1$.
\end{itemize}

Теперь мы готовы доказать, что $\overline{\Delta'}=\Delta''$. Давайте докажем:

\begin{itemize}
    \item При $i \in \{0,1,...,|\Delta''|-r-1\}$ $$\overline{\Delta'_i}=\Delta'_{|\Delta'|-i}=\Delta_{|\Delta'|-i+2}=\Delta_{|\Delta|-i}=\overline{\Delta_{i}}=\Delta''_i;$$
    \item При $i \in \{|\Delta''|-r,...,|\Delta''|\}$
    $$\overline{\Delta'_i}=\Delta'_{|\Delta'|-i}=\Delta_{|\Delta'|-i}=\Delta_{|\Delta|-2-i}=\overline{\Delta}_{i+2}=\Delta''_i.$$
\end{itemize}

Доказали.

Заметим, что 
\begin{itemize}
    \item Если $\delta_{r+1}=1$ и $\delta_{r+2}=-1$, то $E(\Delta)=E(\Delta')\cup\Delta_{r+1}$, а $E(\overline{\Delta})=E(\Delta'')\cup\overline{\Delta}_{|\Delta|-r-1}=E(\Delta'')\cup\Delta_{r+1}$. Мы уже доказали, что ${\overline{\Delta'}}={\Delta''} \Longrightarrow$ (по предположению, которое мы можем применить, так как $s(\Delta')=s(\Delta'')<s(\Delta)=s\left(\overline{\Delta}\right)$) $\Longrightarrow E(\Delta')=E(\Delta'') \Longleftrightarrow E(\Delta)=E\left(\overline{\Delta}\right) $.
    \item Если же, наоборот, $\delta_{r+1}=-1$ и $\delta_{r+2}=1$, то $E(\Delta)=E(\Delta')\cup\Delta_{r+2}$, а $E(\overline{\Delta})=E(\Delta'')\cup\overline{\Delta}_{|\Delta|-r}=E(\Delta'')\cup\Delta_{r}=$(очевидно)=$E(\Delta'')\cup\Delta_{r+2}$. Мы уже доказали, что ${\overline{\Delta'}}={\Delta''} \Longrightarrow$ (по предположению, которое мы можем применить, так как $s(\Delta')=s(\Delta'')<s(\Delta)=s\left(\overline{\Delta}\right)$) $\Longrightarrow E(\Delta')=E(\Delta'') \Longleftrightarrow E(\Delta)=E\left(\overline{\Delta}\right) $.
\end{itemize}

Переход доказан.
\end{proof}

Вернёмся к разбору третьего случая.

Вспомним, что $|x|=|y|$, а поэтому по Утверждению \ref{obrat} $\forall i\in\{1,...,|x|\}$ $$\prod_{j=1}^{s(\Delta)}\left(e\left(\Delta,j\right)-i\right)=\prod_{j=1}^{s\left(\overline{\Delta}\right)}\left(e\left(\overline{\Delta},j\right)-i\right).$$

Значит,
$$o(x,y,\Delta)=\sum_{i=0}^{|x|}\left( {f\left(x,i,z\right)}\prod_{j=1}^{d(y)}\left(g\left(y,j\right)-i\right)\prod_{j=1}^{s(\Delta)}\left(e\left(\Delta,j\right)-i\right)\right)=$$
$$=\text{(по Лемме \ref{xyz})}=$$
$$=\sum_{i=0}^{|y|}\left( {f\left(y,i,z\right)}\prod_{j=1}^{d(x)}\left(g\left(x,j\right)-i\right)\prod_{j=1}^{s\left(\overline{\Delta}\right)}\left(e\left(\overline{\Delta},j\right)-i\right)\right)=o\left(y,x,\overline{\Delta}\right).$$

Несложно понять, что $d(x,y,\Delta)=d\left(y,x,\overline{\Delta}\right)$. Поэтому достаточно доказать, что $d\left(y,x,\overline{\Delta}\right)=o\left(y,x,\overline{\Delta}\right).$

Вспомним, что в данном случае $\delta_{1}=-1$. Это значит, что $\overline\delta_{\left|\overline{\Delta}\right|}=1$.

\renewcommand{\labelenumi}{\theenumi)}
\renewcommand{\labelenumii}{\roman{enumii}$^\circ$}
\renewcommand{\labelenumiii}{\roman{enumii}.\roman{enumiii}$^\circ$}

\begin{Lemma} \label{nol}
Пусть $x\in\mathbb{YF}, \; y \in \{0,...,|x|\}, \; z \in \{0,...,\#x\},\; \alpha_0 \in \{1,2\}$. Тогда
$$f(x,y,z)=f(\alpha_0x,y,z)(|\alpha_0x|-y)$$
\end{Lemma}
\begin{proof}
Будем доказывать Лемму по индукции по $z$.

\underline{\textbf{База}}: $z=0$ -- мгновенно следует из Утверждения \ref{z0}.

\underline{\textbf{Переход}} к $z+1\ge 1$.

Разберём три случая:
\begin{enumerate}
    \item Пусть $y=0$.
    
    В данном случае по Утверждению \ref{y01}  $$f(x,0,z+1)=f(\alpha_0x,0,z+1)|\alpha_0x|\Longleftrightarrow f(x,0,0)=f(\alpha_0x,0,0)|\alpha_0x|.$$
    А это уже было доказано в Базе.

    \item Пусть $y>0$ и $\exists x': x=x'1$.

Воспользуемся рекурсивной формулой для $f$:
$$f(x'1,y,z+1)=f(\alpha_0x'1,y,z+1)(|\alpha_0x'1|-y) \Longleftrightarrow$$
$$\Longleftrightarrow f(x'1,y,0)+f(x',y-1,z)=(f(\alpha_0x'1,y,0)+f(\alpha_0x',y-1,z))(|\alpha_0x'1|-y).$$

По Утверждению \ref{z0}
$$f(x'1,y,0)=f(\alpha_0x'1,y,0)(|\alpha_0x'1|-y).$$

По предположению индукции
$$f(x',y-1,z)=f(\alpha_0x',y-1,z)(|\alpha_0x'|-(y-1)) \Longleftrightarrow f(x',y-1,z)=f(\alpha_0x',y-1,z)(|\alpha_0x'1|-y).$$

Сложив эти два равенства, получаем требуемое.

\item Пусть $y>0$ и $\exists x': x=x'2$.

Рассмотрим два подслучая:
\begin{enumerate}
    \item Пусть $y \ge 2$.

Воспользуемся рекурсивной формулой для $f$:
$$f(x'2,y,z+1)=f(\alpha_0x'2,y,z+1)(|\alpha_0x'2|-y) \Longleftrightarrow$$
$$\Longleftrightarrow \frac{f(x'11,y,z+2)}{1-y}=\frac{f(\alpha_0x'11,y,z+2)}{1-y}(|\alpha_0x'2|-y) \Longleftrightarrow $$
$$\Longleftrightarrow {f(x'11,y,z+2)}={f(\alpha_0x'11,y,z+2)}(|\alpha_0x'11|-y) \Longleftarrow $$
$$\Longleftarrow \text{(аналогично второму случаю)} \Longleftarrow$$
$$\Longleftarrow {f(x'1,y-1,z+1)}={f(\alpha_0x'1,y-1,z+1)}(|\alpha_0x'1|-(y-1)) \Longleftarrow$$
$$\Longleftarrow \text{(аналогично второму случаю)} \Longleftarrow$$
$$\Longleftarrow {f(x',y-2,z)}={f(\alpha_0x',y-2,z)}(|\alpha_0x'|-(y-2)).$$
А это верно по предположению.

\item Пусть $y = 1$.

В данном случае по Утверждению \ref{y01} 
$$f(x'2,1,z+1)=0=f(\alpha_0x'2,1,z+1) \Longrightarrow f(x'2,1,z+1)=0=f(\alpha_0x'2,1,z+1)(|\alpha_0x'2|-1).$$
\end{enumerate}
\end{enumerate}
Все случаи разобраны.

Лемма доказана.
\end{proof}
\renewcommand{\labelenumii}{\arabic{enumii}$)$}
\renewcommand{\labelenumiii}{\roman{enumiii}$^\circ$}
\renewcommand{\labelenumiv}{\roman{enumiii}.\roman{enumiv}$^\circ$}

\begin{Prop} \label{uroven}
$\;$
\begin{enumerate}
    \item Пусть $x\in\mathbb{YF}$, $y\in\{0,1,...,|2x|\}$, $z=\#x$. Тогда
    $$f(2x,y,z)=f(2x,y,z+1);$$
    \item Пусть $x\in\mathbb{YF}$, $y\in\{0,1,...,|x|\}$, $z=\#x$. Тогда
    $$f(1x,y,z)=f(1x,y,z+1);$$
    \item Пусть $x\in\mathbb{YF}$, $z\in\{0,1...,\#(2x)\}$. Тогда
    $$f(2x,|x|+1,z)=0.$$
    
\end{enumerate}
\end{Prop}
\begin{proof}
$\;$

\begin{enumerate}
    \item Докажем Утверждение по индукции по $|x|$.

\underline{\textbf{База}}: $|x|=0\Longleftrightarrow x=\varepsilon$.

Несложно убедиться в том, что $f(2,y,z)$ выглядят следующим образом:
\begin{center}
\begin{tabular}{ | m{0.5cm} || m{0.5cm} | m{0.5cm} | m{0.5cm} | } 
  \hline
 & $y=0$ & $y=1$ & $y=2$ \\
  \hline \hline
$z=0$ & $\frac{1}{2}$ & $0$ & $-\frac{1}{2}$\\
  \hline
$z=1$ & $\frac{1}{2}$ & $0$ &$-\frac{1}{2}$ \\ 
  \hline
\end{tabular}
\end{center}
То есть $f(2,0,0)=f(2,0,1)=\frac{1}{2}$, $f(2,1,0)=f(2,1,1)=0$, $f(2,2,0)=f(2,2,1)=-\frac{1}{2}$.

\underline{\textbf{Переход}} к $x:$ $|x|\ge1$.

Разберём три случая:
\begin{enumerate}
    \item Пусть $y=0$.
    
    В данном случае по Утверждению \ref{y01}  $$f(2x,0,z)=f(2x,0,z+1)\Longleftrightarrow f(2x,0,0)=f(2x,0,0).$$
    Что и требовалось.
    
    \item Пусть $y>0$ и $\exists x': x=x'1$.

Воспользуемся рекурсивной формулой для $f$:
$$f(2x'1,y,z)=f(2x'1,y,z+1)\Longleftrightarrow$$
$$\Longleftrightarrow f(2x'1,y,0)+f(2x',y-1,z-1)=f(2x'1,y,0)+f(2x',y-1,z)$$
$$\Longleftrightarrow  f(2x',y-1,z-1)=f(2x',y-1,z).$$
Ясно, что $\#(2x')=\#(2x'1)-1$, поэтому мы можем воспользоваться предположением индукции.

\item Пусть $y>0$ и $\exists x': x=x'2$.

Рассмотрим два подслучая:
\begin{enumerate}
    \item Пусть $y \ge 2$.

Воспользуемся рекурсивной формулой для $f$:
$$f(2x'2,y,z)=f(2x'2,y,z+1)\Longleftrightarrow$$
$$\Longleftrightarrow \frac{f(2x'11,y,z+1)}{1-y}=\frac{f(2x'11,y,z+2)}{1-y}\Longleftrightarrow $$
$$\Longleftrightarrow {f(2x'11,y,z+1)}={2x'11,y,z+2)}.$$

Ясно, что $\#(2x'11)=\#(2x'1)+1$, поэтому мы можем аналогично второму случаю свести это равенство к следующему равенству:
$${f(2x'1,y-1,z)}=f{(2x'1,y-1,z+1)}.$$

Ясно, что $\#(2x'1)=\#(2x'2)$, и при этом $|2x'1|<|2x'2|$, поэтому мы можем воспользоваться предположением индукции.

\item Пусть $y = 1$.

В данном случае по Утверждению \ref{y01} 
$$f(2x'2,1,z)=0=f(2x'2,1,z+1).$$
Что и требовалось.
\end{enumerate}
\end{enumerate}
Все случаи разобраны.

Утверждение доказано.
    \item 
Докажем Утверждение по индукции по $|x|$.

\underline{\textbf{База}}: $|x|=0\Longleftrightarrow x=\varepsilon$.

Несложно убедиться в том, что $f(1,y,z)$ выглядят следующим образом:

\begin{center}
\begin{tabular}{ | m{0.5cm} || m{0.5cm} | m{0.5cm} | } 
  \hline
 & $y=0$ & $y=1$ \\
  \hline \hline
$z=0$ & $1$ & $-1$ \\
  \hline
$z=1$ & $1$ & $0$ \\ 
  \hline
\end{tabular}
\end{center}

То есть $f(1,0,0)=f(1,0,1)=1$.

\underline{\textbf{Переход}} к $x:$ $|x|\ge1$.

Разберём три случая:
\begin{enumerate}
    \item Пусть $y=0$.
    
    В данном случае по Утверждению \ref{y01}  $$f(1x,0,z)=f(1x,0,z+1)\Longleftrightarrow f(1x,0,0)=f(1x,0,0).$$
    Что и требовалось.
    
    \item Пусть $y>0$ и $\exists x': x=x'1$.

В данном случае воспользуемся рекурсивной формулой для $f$:
$$f(1x'1,y,z)=f(1x'1,y,z+1)\Longleftrightarrow$$
$$\Longleftrightarrow f(1x'1,y,0)+f(1x',y-1,z-1)=f(1x'1,y,0)+f(1x',y-1,z)$$
$$\Longleftrightarrow  f(1x',y-1,z-1)=f(1x',y-1,z).$$
Ясно, что $\#(2x')=\#(2x'1)-1$ и то, что если $y\in\{1,...,|x|\}$, то $(y-1)\in\{1,...,|x'|\}$, поэтому мы можем воспользоваться предположением индукции.

\item Пусть $y>0$ и $\exists x': x=x'2$.

Рассмотрим два подслучая:
\begin{enumerate}
    \item Пусть $y \ge 2$.

В данном случае воспользуемся рекурсивной формулой для $f$:
$$f(1x'2,y,z)=f(1x'2,y,z+1)\Longleftrightarrow$$
$$\Longleftrightarrow \frac{f(1x'11,y,z+1)}{1-y}=\frac{f(1x'11,y,z+2)}{1-y}\Longleftrightarrow $$
$$\Longleftrightarrow {f(1x'11,y,z+1)}={1x'11,y,z+2)}.$$

Ясно, что $\#(1x'11)=\#(1x'1)+1$ и то, что если $y\in\{2,...,|x'2|\}$, то $y\in\{2,...,|x'11|\}$, поэтому мы можем аналогично второму случаю свести это равенство к следующему равенству:
$${f(1x'1,y-1,z)}=f{(1x'1,y-1,z+1)}.$$

Ясно, что $\#(2x'1)=\#(2x'2)$, и при этом $|2x'1|<|2x'2|$, а также то, что если $y\in\{2,...,|x'11|\}$, то $(y-1)\in\{1,...,|x'1|\}$, поэтому мы можем воспользоваться предположением индукции.

\item Пусть $y = 1$.

В данном случае по Утверждению \ref{y01} 
$$f(1x'2,1,z)=0=f(1x'2,1,z+1).$$
Что и требовалось.

\end{enumerate}
\end{enumerate}
Все случаи разобраны.

Утверждение доказано.

    \item 
Докажем Утверждение по индукции по $|x|$.

\underline{\textbf{База}}: $|x|=0\Longleftrightarrow x=\varepsilon$.

Несложно убедиться в том, что $f(2,y,z)$ выглядят следующим образом:
\begin{center}
\begin{tabular}{ | m{0.5cm} || m{0.5cm} | m{0.5cm} | m{0.5cm} | } 
  \hline
 & $y=0$ & $y=1$ & $y=2$ \\
  \hline \hline
$z=0$ & $\frac{1}{2}$ & $0$ & $-\frac{1}{2}$\\
  \hline
$z=1$ & $\frac{1}{2}$ & $0$ &$-\frac{1}{2}$ \\ 
  \hline
\end{tabular}
\end{center}
То есть $f(2,1,0)=f(2,1,1)=0$, что и требовалось.

\underline{\textbf{Переход}} к $x$: $|x|\ge 1$.

Ясно, что $|x|+1 \ge 2 $.

Разберём три случая:
\begin{enumerate}
    \item Пусть $z=0$.
    
    В данном случае
    $$f(2x,|x|+1,0)=0$$
    по формуле для функции $f$.
    \item Пусть $z>0$ и $\exists x': x=x'1$.
    
В данном случае воспользуемся рекурсивной формулой для $f$:
$$f(2x'1,|x|+1,z)=0\Longleftrightarrow f(2x'1,|x|+1,0)+f(2x',|x|,z-1)=0.$$

Знаем, что $f(2x'1,|x|+1,0)=0$ по формуле для функции $f$.

Также несложно заметить, что $|x|=|x'|+1$, а значит $f(2x',|x|,z-1)=0$ по предположению индукции.

\item Пусть $z>0$ и $\exists x': x=x'2$.

В данном случае воспользуемся рекурсивной формулой для $f$:
$$f(2x'2,|x|+1,z)=0\Longleftrightarrow \frac{f(2x'11,|x|+1,z+1)}{1-(|x|+1)}=0\Longleftrightarrow$$
$$\Longleftrightarrow \text{(Так как $|x|\ge 1$)}\Longleftrightarrow {f(2x'11,|x|+1,z+1)}=0.$$

Ясно, что $|2x'11|=|2x'2|=|2x|$, поэтому мы можем аналогично второму случаю свести это равенство к следующему равенству:
$${f(2x'1,|x|,z)}=0 $$

Несложно заметить, что $|x|=|x'|+1$, а значит мы можем воспользоваться предположением индукции.

\end{enumerate}
Все случаи разобраны.

Утверждение доказано.
\end{enumerate}
\end{proof}

Вернёмся к доказательству третьего случая.

Вспомним, что мы пришли к тому, что достаточно доказать, что $d\left(y,x,\overline{\Delta}\right)=o\left(y,x,\overline{\Delta}\right).$

Рассмотрим два подслучая:
\renewcommand{\labelenumi}{\arabic{enumi}$^\circ$}
\renewcommand{\labelenumii}{\arabic{enumi}.\arabic{enumii}$^\circ$}
\begin{enumerate}
    \item Пусть $\exists   y': y=2y'$.
    
Сразу заметим, что в данном случае $|y|\ge 2$.

После этого введём $\Delta'$ следующим образом:

\begin{tikzpicture}
\draw[ultra thick, ->] (0.5,0) -- (0.5,7);
\draw[thin, dashed, ->] (1.5,0) -- (1.5,7);
\draw[thin, dashed, ->] (2.5,0) -- (2.5,7);
\draw[thin, dashed, ->] (3.5,0) -- (3.5,7);
\draw[thin, dashed, ->] (4.5,0) -- (4.5,7);
\draw[thin, dashed, ->] (5.5,0) -- (5.5,7);
\draw[thin, dashed, ->] (6.5,0) -- (6.5,7);
\draw[thin, dashed, ->] (7.5,0) -- (7.5,7);
\draw[thin, dashed, ->] (8.5,0) -- (8.5,7);
\draw[thin, dashed, ->] (9.5,0) -- (9.5,7);
\draw[thin, dashed, ->] (10.5,0) -- (10.5,7);
\draw[thin, dashed, ->] (11.5,0) -- (11.5,7);
\draw[thin, dashed, ->] (12.5,0) -- (12.5,7);
\draw[thin, dashed, ->] (13.5,0) -- (13.5,7);
\draw[thin, dashed, ->] (14.5,0) -- (14.5,7);

\draw[ultra thick, ->] (0,0.5) -- (15,0.5);
\draw[thin, dashed, ->] (0,1.5) -- (15,1.5);
\draw[thin, dashed, ->] (0,2.5) -- (15,2.5);
\draw[thin, dashed, ->] (0,3.5) -- (15,3.5);
\draw[thin, dashed, ->] (0,4.5) -- (15,4.5);
\draw[thin, dashed, ->] (0,5.5) -- (15,5.5);
\draw[thin, dashed, ->] (0,6.5) -- (15,6.5);

\draw[ultra thick, red, ->] (0.5,4.5) -- (1.5,3.5);
\draw[ultra thick, red, ->] (1.5,3.5) -- (2.5,2.5);
\draw[ultra thick, red, ->] (2.5,2.5) -- (3.5,3.5);
\draw[ultra thick, red, ->] (3.5,3.5) -- (4.5,2.5);
\draw[ultra thick, red, ->] (4.5,2.5) -- (5.5,1.5);
\draw[ultra thick, red, ->] (5.5,1.5) -- (6.5,0.5);
\draw[ultra thick, red, ->] (6.5,0.5) -- (7.5,1.5);
\draw[ultra thick, red, ->] (7.5,1.5) -- (8.5,2.5);
\draw[ultra thick, red, ->] (8.5,2.5) -- (9.5,3.5);
\draw[ultra thick, red, ->] (9.5,3.5) -- (10.5,4.5);
\draw[ultra thick, red, ->] (10.5,4.5) -- (11.5,5.5);
\draw[ultra thick, red, ->] (11.5,5.5) -- (12.5,6.5);
\draw[ultra thick, red, ->] (12.5,6.5) -- (13.5,5.5);
\draw[ultra thick, red, ->] (13.5,5.5) -- (14.5,4.5);

\path (0.25,0.25) node  {$0$}
-- (1.25,0.25)  node  {$1$}
-- (2.25,0.25)  node  {$2$}
-- (3.25,0.25)  node  {$3$}
-- (4.25,0.25)  node  {$4$}
-- (5.25,0.25)  node  {$5$}
-- (6.25,0.25)  node  {$6$}
-- (7.25,0.25)  node  {$7$}
-- (8.25,0.25)  node  {$8$}
-- (9.25,0.25)  node  {$9$}
-- (10.25,0.25)  node  {$10$}
-- (11.25,0.25)  node  {$11$}
-- (12.25,0.25)  node  {$12$}
-- (13.25,0.25)  node  {$13$}
-- (14.25,0.25)  node  {$14$};

\path (0.25,1.25)  node  {$1$}
-- (0.25,2.25)  node  {$2$}
-- (0.25,3.25)  node  {$3$}
-- (0.25,4.25)  node  {$4$}
-- (0.25,5.25)  node  {$5$}
-- (0.25,6.25)  node  {$6$};
\end{tikzpicture}

$$$$

\begin{tikzpicture}

\draw[thin, dashed, ->] (0.5,0) -- (0.5,7);
\draw[thin, dashed, ->] (1.5,0) -- (1.5,7);
\draw[thin, dashed, ->] (2.5,0) -- (2.5,7);
\draw[thin, dashed, ->] (3.5,0) -- (3.5,7);
\draw[thin, dashed, ->] (4.5,0) -- (4.5,7);
\draw[thin, dashed, ->] (5.5,0) -- (5.5,7);
\draw[thin, dashed, ->] (6.5,0) -- (6.5,7);
\draw[thin, dashed, ->] (7.5,0) -- (7.5,7);
\draw[thin, dashed, ->] (8.5,0) -- (8.5,7);
\draw[thin, dashed, ->] (9.5,0) -- (9.5,7);
\draw[thin, dashed, ->] (10.5,0) -- (10.5,7);
\draw[thin, dashed, ->] (11.5,0) -- (11.5,7);
\draw[thin, dashed, ->] (12.5,0) -- (12.5,7);
\draw[thin, dashed, ->] (13.5,0) -- (13.5,7);
\draw[ultra thick, ->] (14.5,0) -- (14.5,7);

\draw[ultra thick, <-] (0,0.5) -- (15,0.5);
\draw[thin, dashed, <-] (0,1.5) -- (15,1.5);
\draw[thin, dashed, <-] (0,2.5) -- (15,2.5);
\draw[thin, dashed, <-] (0,3.5) -- (15,3.5);
\draw[thin, dashed, <-] (0,4.5) -- (15,4.5);
\draw[thin, dashed, <-] (0,5.5) -- (15,5.5);
\draw[thin, dashed, <-] (0,6.5) -- (15,6.5);

\draw[ultra thick, blue, <-] (0.5,4.5) -- (1.5,3.5);
\draw[ultra thick, blue, <-] (1.5,3.5) -- (2.5,2.5);
\draw[ultra thick, blue, <-] (2.5,2.5) -- (3.5,3.5);
\draw[ultra thick, blue, <-] (3.5,3.5) -- (4.5,2.5);
\draw[ultra thick, blue, <-] (4.5,2.5) -- (5.5,1.5);
\draw[ultra thick, blue, <-] (5.5,1.5) -- (6.5,0.5);
\draw[ultra thick, blue, <-] (6.5,0.5) -- (7.5,1.5);
\draw[ultra thick, blue, <-] (7.5,1.5) -- (8.5,2.5);
\draw[ultra thick, blue, <-] (8.5,2.5) -- (9.5,3.5);
\draw[ultra thick, blue, <-] (9.5,3.5) -- (10.5,4.5);
\draw[ultra thick, blue, <-] (10.5,4.5) -- (11.5,5.5);
\draw[ultra thick, blue, <-] (11.5,5.5) -- (12.5,6.5);
\draw[ultra thick, blue, <-] (12.5,6.5) -- (13.5,5.5);
\draw[ultra thick, blue, <-] (13.5,5.5) -- (14.5,4.5);

\path (0.75,0.25) node  {$14$}
-- (1.75,0.25)  node  {$13$}
-- (2.75,0.25)  node  {$12$}
-- (3.75,0.25)  node  {$11$}
-- (4.75,0.25)  node  {$10$}
-- (5.75,0.25)  node  {$9$}
-- (6.75,0.25)  node  {$8$}
-- (7.75,0.25)  node  {$7$}
-- (8.75,0.25)  node  {$6$}
-- (9.75,0.25)  node  {$5$}
-- (10.75,0.25)  node  {$4$}
-- (11.75,0.25)  node  {$3$}
-- (12.75,0.25)  node  {$2$}
-- (13.75,0.25)  node  {$1$}
-- (14.75,0.25)  node  {$0$};

\path (14.75,1.25)  node  {$1$}
-- (14.75,2.25)  node  {$2$}
-- (14.75,3.25)  node  {$3$}
-- (14.75,4.25)  node  {$4$}
-- (14.75,5.25)  node  {$5$}
-- (14.75,6.25)  node  {$6$};
\end{tikzpicture}

$$$$

\begin{tikzpicture}

\draw[thin, dashed, ->] (0.5,0) -- (0.5,7);
\draw[thin, dashed, ->] (1.5,0) -- (1.5,7);
\draw[thin, dashed, ->] (2.5,0) -- (2.5,7);
\draw[thin, dashed, ->] (3.5,0) -- (3.5,7);
\draw[thin, dashed, ->] (4.5,0) -- (4.5,7);
\draw[thin, dashed, ->] (5.5,0) -- (5.5,7);
\draw[thin, dashed, ->] (6.5,0) -- (6.5,7);
\draw[thin, dashed, ->] (7.5,0) -- (7.5,7);
\draw[thin, dashed, ->] (8.5,0) -- (8.5,7);
\draw[thin, dashed, ->] (9.5,0) -- (9.5,7);
\draw[thin, dashed, ->] (10.5,0) -- (10.5,7);
\draw[thin, dashed, ->] (11.5,0) -- (11.5,7);
\draw[thin, dashed, ->] (12.5,0) -- (12.5,7);
\draw[thin, dashed, ->] (13.5,0) -- (13.5,7);
\draw[ultra thick, ->] (14.5,0) -- (14.5,7);

\draw[ultra thick, <-] (0,0.5) -- (15,0.5);
\draw[thin, dashed, <-] (0,1.5) -- (15,1.5);
\draw[thin, dashed, <-] (0,2.5) -- (15,2.5);
\draw[thin, dashed, <-] (0,3.5) -- (15,3.5);
\draw[thin, dashed, <-] (0,4.5) -- (15,4.5);
\draw[thin, dashed, <-] (0,5.5) -- (15,5.5);
\draw[thin, dashed, <-] (0,6.5) -- (15,6.5);

\draw[ultra thick, green, <-] (0.5,2.5) -- (1.5,3.5);
\draw[ultra thick, green, <-] (1.5,3.5) -- (2.5,2.5);
\draw[ultra thick, green, <-] (2.5,2.5) -- (3.5,3.5);
\draw[ultra thick, green, <-] (3.5,3.5) -- (4.5,2.5);
\draw[ultra thick, green, <-] (4.5,2.5) -- (5.5,1.5);
\draw[ultra thick, green, <-] (5.5,1.5) -- (6.5,0.5);
\draw[ultra thick, green, <-] (6.5,0.5) -- (7.5,1.5);
\draw[ultra thick, green, <-] (7.5,1.5) -- (8.5,2.5);
\draw[ultra thick, green, <-] (8.5,2.5) -- (9.5,3.5);
\draw[ultra thick, green, <-] (9.5,3.5) -- (10.5,4.5);
\draw[ultra thick, green, <-] (10.5,4.5) -- (11.5,5.5);
\draw[ultra thick, green, <-] (11.5,5.5) -- (12.5,6.5);
\draw[ultra thick, green, <-] (12.5,6.5) -- (13.5,5.5);
\draw[ultra thick, green, <-] (13.5,5.5) -- (14.5,4.5);

\path (0.75,0.25) node  {$14$}
-- (1.75,0.25)  node  {$13$}
-- (2.75,0.25)  node  {$12$}
-- (3.75,0.25)  node  {$11$}
-- (4.75,0.25)  node  {$10$}
-- (5.75,0.25)  node  {$9$}
-- (6.75,0.25)  node  {$8$}
-- (7.75,0.25)  node  {$7$}
-- (8.75,0.25)  node  {$6$}
-- (9.75,0.25)  node  {$5$}
-- (10.75,0.25)  node  {$4$}
-- (11.75,0.25)  node  {$3$}
-- (12.75,0.25)  node  {$2$}
-- (13.75,0.25)  node  {$1$}
-- (14.75,0.25)  node  {$0$};

\path (14.75,1.25)  node  {$1$}
-- (14.75,2.25)  node  {$2$}
-- (14.75,3.25)  node  {$3$}
-- (14.75,4.25)  node  {$4$}
-- (14.75,5.25)  node  {$5$}
-- (14.75,6.25)  node  {$6$};

\end{tikzpicture}

На рисунках изображено, как по траектории $\Delta$ (верхний рисунок) строится траектория $\overline{\Delta}$ (средний рисунок), а по ней строится траектория $\Delta'$ (нижний рисунок). В данном случае последний шаг траектории $\overline{\Delta}$ -- это шаг ``вверх'', и мы просто меняем его на шаг ``вниз'', при это оставляя все значения $\overline{\Delta_i}$, кроме последнего, прежними.  Формально:
$$|\Delta'|=|\overline{\Delta}|; \;   \forall i \in \left\{ 1,...,|\Delta'|-1 \right\} \;  \Delta'_{i}=\overline{\Delta}_{i};\; \Delta'_{|\Delta'|}=\overline{\Delta}_{|\Delta'|}-2.$$

Несложно заметить, что 
\begin{itemize}
    \item $\Delta'_{|\Delta'|}=\overline{\Delta}_{|\Delta'|}-2=|y|-2=|y'|$;
    \item $ \Delta'\in \Omega(|y'|,|x|)$;
    \item $s(\Delta')=s\left(\overline{\Delta}\right)-1$;
    \item $\forall j \in \{1,...,s(\Delta')\}\quad e(\Delta',j)=e\left(\overline{\Delta},j\right)$;
    \item $\{|y|\}\cup E(\Delta')=E\left(\overline{\Delta}\right)$;
    \item $|\Delta'|=|\overline{\Delta}|$ и $s(\Delta')<s\left(\overline{\Delta}\right)$, то есть мы можем воспользоваться предположением индукции.
\end{itemize}

Но главное наблюдение заключается в том, что по Утверждению \ref{2vnachale} $\forall k\in\{1,...,r(y)\}$  существует биекция между $\overline{\Delta}$-путями $x...y_ky$ и $\Delta'$-путями $x...y_ky'$.
\renewcommand{\labelenumiii}{\arabic{enumi}.\arabic{enumii}.\arabic{enumiii}$^\circ$}
Далее рассмотрим два подслучая:
\begin{enumerate}
    \item Пусть $x\ne y$.
    
    В данном случае $z=h(y,x)<\#y$, а значит $h(y',x)=z$. С этими знаниями мы готовы считать. Посчитаем:
$$o\left(y,x,\overline{\Delta}\right)=\sum_{i=0}^{|y|}\left( {f\left(y,i,z\right)}\prod_{j=1}^{d(x)}\left(g\left(x,j\right)-i\right)\prod_{j=1}^{s\left(\overline{\Delta}\right)}\left(e\left(\overline{\Delta},j\right)-i\right)\right)=$$
$$=\sum_{i=0}^{|y|}\left( {f\left(2y',i,z\right)}\prod_{j=1}^{d(x)}\left(g\left(x,j\right)-i\right)\prod_{j=1}^{s\left(\overline{\Delta}\right)}\left(e\left(\overline{\Delta},j\right)-i\right)\right)=$$
$$=\text{(По Лемме \ref{nol})}= $$
$$=\sum_{i=0}^{|y'|}\left( \frac{f\left(y',i,z\right)}{|2y'|-i}\prod_{j=1}^{d(x)}\left(g\left(x,j\right)-i\right)\left(\prod_{j=1}^{s\left(\Delta'\right)}\left(e\left(\Delta',j\right)-i\right)\right)(|2y'|-i)\right)+$$
$$+\sum_{i=|y'|+1}^{|y'|+2}\left( {f\left(2y',i,z\right)}\prod_{j=1}^{d(x)}\left(g\left(x,j\right)-i\right)\prod_{j=1}^{s\left(\overline{\Delta}\right)}\left(e\left(\overline{\Delta},j\right)-i\right)\right).$$

Давайте докажем, что вторая сумма равняется нулю:
\begin{itemize}
    \item По Утверждению \ref{uroven} $\forall y'\in\mathbb{YF}$ и $z\in\{0,1,...,\#(2y')\}\quad$ $f(2y',|y'|+1,z)=0$, а значит
    $$ {f\left(2y',|y'|+1,z\right)}\prod_{j=1}^{d(x)}\left(g\left(x,j\right)-(|y'|+1)\right)\prod_{j=1}^{s\left(\overline{\Delta}\right)}\left(e\left(\overline{\Delta},j\right)-(|y'+1|)\right)=0.$$
    \item
    Знаем, что $e\left(\overline{\Delta},s(\Delta)\right)=|y|=|y'|+2$, а значит
    $$ {f\left(2y',|y'|+2,z\right)}\prod_{j=1}^{d(x)}\left(g\left(x,j\right)-(|y'|+2)\right)\prod_{j=1}^{s\left(\overline{\Delta}\right)}\left(e\left(\overline{\Delta},j\right)-(|y'+2|)\right)=0.$$
\end{itemize}

Таким образом, наше выражение равняется следующему выражению:
$$\sum_{i=0}^{|y'|}\left( {f\left(y',i,z\right)}\prod_{j=1}^{d(x)}\left(g\left(x,j\right)-i\right)\prod_{j=1}^{s\left(\Delta'\right)}\left(e\left(\Delta',j\right)-i\right)\right)=$$
$$=o(y',x,\Delta')=\text{(По предположению)}=d(y',x,\Delta')=d\left(y,x,\overline{\Delta}\right).$$

Что и требовалось.

    \item Пусть $x= y$.
    
    В данном случае $z=h(y,x)=\#y$, а значит $z':=h(y',x)=z-1$. Счёт будет аналогичным предыдущему случаю, однако, чуть более длинным:
$$o\left(y,x,\overline{\Delta}\right)=\sum_{i=0}^{|y|}\left( {f\left(y,i,z\right)}\prod_{j=1}^{d(x)}\left(g\left(x,j\right)-i\right)\prod_{j=1}^{s\left(\overline{\Delta}\right)}\left(e\left(\overline{\Delta},j\right)-i\right)\right)=$$
$$=\sum_{i=0}^{|y|}\left( {f\left(2y',i,z\right)}\prod_{j=1}^{d(x)}\left(g\left(x,j\right)-i\right)\prod_{j=1}^{s\left(\overline{\Delta}\right)}\left(e\left(\overline{\Delta},j\right)-i\right)\right)=$$
$$=\text{(По Утверждению \ref{uroven})}=$$
$$=\sum_{i=0}^{|y|}\left( {f\left(2y',i,z-1\right)}\prod_{j=1}^{d(x)}\left(g\left(x,j\right)-i\right)\prod_{j=1}^{s\left(\overline{\Delta}\right)}\left(e\left(\overline{\Delta},j\right)-i\right)\right)=$$
$$=\text{(По Лемме \ref{nol})}= $$
$$=\sum_{i=0}^{|y'|}\left( \frac{f\left(y',i,z'\right)}{|2y'|-i}\prod_{j=1}^{d(x)}\left(g\left(x,j\right)-i\right)\left(\prod_{j=1}^{s\left(\Delta'\right)}\left(e\left(\Delta',j\right)-i\right)\right)(|2y'|-i)\right)+$$
$$+\sum_{i=|y'|+1}^{|y'|+2}\left( {f\left(2y',i,z'\right)}\prod_{j=1}^{d(x)}\left(g\left(x,j\right)-i\right)\prod_{j=1}^{s\left(\overline{\Delta}\right)}\left(e\left(\overline{\Delta},j\right)-i\right)\right).$$

Давайте докажем, что вторая сумма равняется нулю:
\begin{itemize}
    \item По Утверждению \ref{uroven} $\forall y'\in\mathbb{YF}$ и $z\in\{0,1,...,\#(2y')\}\quad f(2y',|y'|+1,z)=0$, а значит
    $$ {f\left(2y',|y'|+1,z'\right)}\prod_{j=1}^{d(x)}\left(g\left(x,j\right)-(|y'|+1)\right)\prod_{j=1}^{s\left(\overline{\Delta}\right)}\left(e\left(\overline{\Delta},j\right)-(|y'+1|)\right)=0.$$
    \item
    Знаем, что $e\left(\overline{\Delta},s(\Delta)\right)=|y|=|y'|+2$, а значит
    $$ {f\left(2y',|y'|+2,z'\right)}\prod_{j=1}^{d(x)}\left(g\left(x,j\right)-(|y'|+2)\right)\prod_{j=1}^{s\left(\overline{\Delta}\right)}\left(e\left(\overline{\Delta},j\right)-(|y'+2|)\right)=0.$$
\end{itemize}

Таким образом, наше выражение равняется следующему выражению:
$$\sum_{i=0}^{|y'|}\left( {f\left(y',i,z'\right)}\prod_{j=1}^{d(x)}\left(g\left(x,j\right)-i\right)\prod_{j=1}^{s\left(\Delta'\right)}\left(e\left(\Delta',j\right)-i\right)\right)=$$
$$=o(y',x,\Delta')=\text{(По предположению)}=d(y',x,\Delta')=d\left(y,x,\overline{\Delta}\right).$$
Что и требовалось.

\end{enumerate}

\item Пусть $\exists   y': y=1y'$.

В данном случае введём $\Delta'$ следующим образом:

\begin{tikzpicture}
\draw[ultra thick, ->] (0.5,0) -- (0.5,7);
\draw[thin, dashed, ->] (1.5,0) -- (1.5,7);
\draw[thin, dashed, ->] (2.5,0) -- (2.5,7);
\draw[thin, dashed, ->] (3.5,0) -- (3.5,7);
\draw[thin, dashed, ->] (4.5,0) -- (4.5,7);
\draw[thin, dashed, ->] (5.5,0) -- (5.5,7);
\draw[thin, dashed, ->] (6.5,0) -- (6.5,7);
\draw[thin, dashed, ->] (7.5,0) -- (7.5,7);
\draw[thin, dashed, ->] (8.5,0) -- (8.5,7);
\draw[thin, dashed, ->] (9.5,0) -- (9.5,7);
\draw[thin, dashed, ->] (10.5,0) -- (10.5,7);
\draw[thin, dashed, ->] (11.5,0) -- (11.5,7);
\draw[thin, dashed, ->] (12.5,0) -- (12.5,7);
\draw[thin, dashed, ->] (13.5,0) -- (13.5,7);
\draw[thin, dashed, ->] (14.5,0) -- (14.5,7);

\draw[ultra thick, ->] (0,0.5) -- (15,0.5);
\draw[thin, dashed, ->] (0,1.5) -- (15,1.5);
\draw[thin, dashed, ->] (0,2.5) -- (15,2.5);
\draw[thin, dashed, ->] (0,3.5) -- (15,3.5);
\draw[thin, dashed, ->] (0,4.5) -- (15,4.5);
\draw[thin, dashed, ->] (0,5.5) -- (15,5.5);
\draw[thin, dashed, ->] (0,6.5) -- (15,6.5);

\draw[ultra thick, red, ->] (0.5,4.5) -- (1.5,3.5);
\draw[ultra thick, red, ->] (1.5,3.5) -- (2.5,2.5);
\draw[ultra thick, red, ->] (2.5,2.5) -- (3.5,3.5);
\draw[ultra thick, red, ->] (3.5,3.5) -- (4.5,2.5);
\draw[ultra thick, red, ->] (4.5,2.5) -- (5.5,1.5);
\draw[ultra thick, red, ->] (5.5,1.5) -- (6.5,0.5);
\draw[ultra thick, red, ->] (6.5,0.5) -- (7.5,1.5);
\draw[ultra thick, red, ->] (7.5,1.5) -- (8.5,2.5);
\draw[ultra thick, red, ->] (8.5,2.5) -- (9.5,3.5);
\draw[ultra thick, red, ->] (9.5,3.5) -- (10.5,4.5);
\draw[ultra thick, red, ->] (10.5,4.5) -- (11.5,5.5);
\draw[ultra thick, red, ->] (11.5,5.5) -- (12.5,6.5);
\draw[ultra thick, red, ->] (12.5,6.5) -- (13.5,5.5);
\draw[ultra thick, red, ->] (13.5,5.5) -- (14.5,4.5);

\path (0.25,0.25) node  {$0$}
-- (1.25,0.25)  node  {$1$}
-- (2.25,0.25)  node  {$2$}
-- (3.25,0.25)  node  {$3$}
-- (4.25,0.25)  node  {$4$}
-- (5.25,0.25)  node  {$5$}
-- (6.25,0.25)  node  {$6$}
-- (7.25,0.25)  node  {$7$}
-- (8.25,0.25)  node  {$8$}
-- (9.25,0.25)  node  {$9$}
-- (10.25,0.25)  node  {$10$}
-- (11.25,0.25)  node  {$11$}
-- (12.25,0.25)  node  {$12$}
-- (13.25,0.25)  node  {$13$}
-- (14.25,0.25)  node  {$14$};

\path (0.25,1.25)  node  {$1$}
-- (0.25,2.25)  node  {$2$}
-- (0.25,3.25)  node  {$3$}
-- (0.25,4.25)  node  {$4$}
-- (0.25,5.25)  node  {$5$}
-- (0.25,6.25)  node  {$6$};
\end{tikzpicture}

$$$$

\begin{tikzpicture}

\draw[thin, dashed, ->] (0.5,0) -- (0.5,7);
\draw[thin, dashed, ->] (1.5,0) -- (1.5,7);
\draw[thin, dashed, ->] (2.5,0) -- (2.5,7);
\draw[thin, dashed, ->] (3.5,0) -- (3.5,7);
\draw[thin, dashed, ->] (4.5,0) -- (4.5,7);
\draw[thin, dashed, ->] (5.5,0) -- (5.5,7);
\draw[thin, dashed, ->] (6.5,0) -- (6.5,7);
\draw[thin, dashed, ->] (7.5,0) -- (7.5,7);
\draw[thin, dashed, ->] (8.5,0) -- (8.5,7);
\draw[thin, dashed, ->] (9.5,0) -- (9.5,7);
\draw[thin, dashed, ->] (10.5,0) -- (10.5,7);
\draw[thin, dashed, ->] (11.5,0) -- (11.5,7);
\draw[thin, dashed, ->] (12.5,0) -- (12.5,7);
\draw[thin, dashed, ->] (13.5,0) -- (13.5,7);
\draw[ultra thick, ->] (14.5,0) -- (14.5,7);

\draw[ultra thick, <-] (0,0.5) -- (15,0.5);
\draw[thin, dashed, <-] (0,1.5) -- (15,1.5);
\draw[thin, dashed, <-] (0,2.5) -- (15,2.5);
\draw[thin, dashed, <-] (0,3.5) -- (15,3.5);
\draw[thin, dashed, <-] (0,4.5) -- (15,4.5);
\draw[thin, dashed, <-] (0,5.5) -- (15,5.5);
\draw[thin, dashed, <-] (0,6.5) -- (15,6.5);

\draw[ultra thick, blue, <-] (0.5,4.5) -- (1.5,3.5);
\draw[ultra thick, blue, <-] (1.5,3.5) -- (2.5,2.5);
\draw[ultra thick, blue, <-] (2.5,2.5) -- (3.5,3.5);
\draw[ultra thick, blue, <-] (3.5,3.5) -- (4.5,2.5);
\draw[ultra thick, blue, <-] (4.5,2.5) -- (5.5,1.5);
\draw[ultra thick, blue, <-] (5.5,1.5) -- (6.5,0.5);
\draw[ultra thick, blue, <-] (6.5,0.5) -- (7.5,1.5);
\draw[ultra thick, blue, <-] (7.5,1.5) -- (8.5,2.5);
\draw[ultra thick, blue, <-] (8.5,2.5) -- (9.5,3.5);
\draw[ultra thick, blue, <-] (9.5,3.5) -- (10.5,4.5);
\draw[ultra thick, blue, <-] (10.5,4.5) -- (11.5,5.5);
\draw[ultra thick, blue, <-] (11.5,5.5) -- (12.5,6.5);
\draw[ultra thick, blue, <-] (12.5,6.5) -- (13.5,5.5);
\draw[ultra thick, blue, <-] (13.5,5.5) -- (14.5,4.5);

\path (0.75,0.25) node  {$14$}
-- (1.75,0.25)  node  {$13$}
-- (2.75,0.25)  node  {$12$}
-- (3.75,0.25)  node  {$11$}
-- (4.75,0.25)  node  {$10$}
-- (5.75,0.25)  node  {$9$}
-- (6.75,0.25)  node  {$8$}
-- (7.75,0.25)  node  {$7$}
-- (8.75,0.25)  node  {$6$}
-- (9.75,0.25)  node  {$5$}
-- (10.75,0.25)  node  {$4$}
-- (11.75,0.25)  node  {$3$}
-- (12.75,0.25)  node  {$2$}
-- (13.75,0.25)  node  {$1$}
-- (14.75,0.25)  node  {$0$};

\path (14.75,1.25)  node  {$1$}
-- (14.75,2.25)  node  {$2$}
-- (14.75,3.25)  node  {$3$}
-- (14.75,4.25)  node  {$4$}
-- (14.75,5.25)  node  {$5$}
-- (14.75,6.25)  node  {$6$};
\end{tikzpicture}

$$$$

\begin{tikzpicture}

\draw[thin, dashed, ->] (0.5,0) -- (0.5,7);
\draw[thin, dashed, ->] (1.5,0) -- (1.5,7);
\draw[thin, dashed, ->] (2.5,0) -- (2.5,7);
\draw[thin, dashed, ->] (3.5,0) -- (3.5,7);
\draw[thin, dashed, ->] (4.5,0) -- (4.5,7);
\draw[thin, dashed, ->] (5.5,0) -- (5.5,7);
\draw[thin, dashed, ->] (6.5,0) -- (6.5,7);
\draw[thin, dashed, ->] (7.5,0) -- (7.5,7);
\draw[thin, dashed, ->] (8.5,0) -- (8.5,7);
\draw[thin, dashed, ->] (9.5,0) -- (9.5,7);
\draw[thin, dashed, ->] (10.5,0) -- (10.5,7);
\draw[thin, dashed, ->] (11.5,0) -- (11.5,7);
\draw[thin, dashed, ->] (12.5,0) -- (12.5,7);
\draw[thin, dashed, ->] (13.5,0) -- (13.5,7);
\draw[ultra thick, ->] (14.5,0) -- (14.5,7);

\draw[ultra thick, <-] (0,0.5) -- (15,0.5);
\draw[thin, dashed, <-] (0,1.5) -- (15,1.5);
\draw[thin, dashed, <-] (0,2.5) -- (15,2.5);
\draw[thin, dashed, <-] (0,3.5) -- (15,3.5);
\draw[thin, dashed, <-] (0,4.5) -- (15,4.5);
\draw[thin, dashed, <-] (0,5.5) -- (15,5.5);
\draw[thin, dashed, <-] (0,6.5) -- (15,6.5);

\draw[ultra thick, green, <-] (1.5,3.5) -- (2.5,2.5);
\draw[ultra thick, green, <-] (2.5,2.5) -- (3.5,3.5);
\draw[ultra thick, green, <-] (3.5,3.5) -- (4.5,2.5);
\draw[ultra thick, green, <-] (4.5,2.5) -- (5.5,1.5);
\draw[ultra thick, green, <-] (5.5,1.5) -- (6.5,0.5);
\draw[ultra thick, green, <-] (6.5,0.5) -- (7.5,1.5);
\draw[ultra thick, green, <-] (7.5,1.5) -- (8.5,2.5);
\draw[ultra thick, green, <-] (8.5,2.5) -- (9.5,3.5);
\draw[ultra thick, green, <-] (9.5,3.5) -- (10.5,4.5);
\draw[ultra thick, green, <-] (10.5,4.5) -- (11.5,5.5);
\draw[ultra thick, green, <-] (11.5,5.5) -- (12.5,6.5);
\draw[ultra thick, green, <-] (12.5,6.5) -- (13.5,5.5);
\draw[ultra thick, green, <-] (13.5,5.5) -- (14.5,4.5);

\path (0.75,0.25) node  {$14$}
-- (1.75,0.25)  node  {$13$}
-- (2.75,0.25)  node  {$12$}
-- (3.75,0.25)  node  {$11$}
-- (4.75,0.25)  node  {$10$}
-- (5.75,0.25)  node  {$9$}
-- (6.75,0.25)  node  {$8$}
-- (7.75,0.25)  node  {$7$}
-- (8.75,0.25)  node  {$6$}
-- (9.75,0.25)  node  {$5$}
-- (10.75,0.25)  node  {$4$}
-- (11.75,0.25)  node  {$3$}
-- (12.75,0.25)  node  {$2$}
-- (13.75,0.25)  node  {$1$}
-- (14.75,0.25)  node  {$0$};

\path (14.75,1.25)  node  {$1$}
-- (14.75,2.25)  node  {$2$}
-- (14.75,3.25)  node  {$3$}
-- (14.75,4.25)  node  {$4$}
-- (14.75,5.25)  node  {$5$}
-- (14.75,6.25)  node  {$6$};

\end{tikzpicture}

На рисунках изображено, как по траектории $\Delta$ (верхний рисунок) строится траектория $\overline{\Delta}$ (средний рисунок), а по ней строиться траектория $\Delta'$ (нижний рисунок). В данном случае последний шаг траектории $\overline{\Delta}$ -- это шаг ``вверх'', и мы просто убираем его.  Формально:
$$|\Delta'|=\left|\overline{\Delta}\right|-1; \;   \forall i \in \left\{ 1,...,|\Delta'| \right\} \;  \Delta'_{i}=\overline{\Delta}_{i}.$$

Несложно заметить, что:
\begin{itemize}
    \item $\Delta'_{|\Delta'|}=\overline{\Delta}_{|\Delta'|}=\overline{\Delta}_{\left|\overline{\Delta}\right|-1}=\overline{\Delta}_{\left|\overline{\Delta}\right|}-1=|y|-1=|y'|$;
    \item $ \Delta'\in \Omega(|y'|,|x|)$;
    \item $s(\Delta')=s\left(\overline{\Delta}\right)-1$;
    \item $\forall j \in \{1,...,s(\Delta')\}\quad e(\Delta',j)=e\left(\overline{\Delta},j\right)$;
    \item $\{|y|\}\cup E(\Delta')=E\left(\overline{\Delta}\right)$;
    \item $|\Delta'|<\left|\overline{\Delta}\right|$, то есть мы можем воспользоваться предположением индукции.
\end{itemize}

Но главное наблюдение заключается в том, что существует биекция между $\overline{\Delta}$-путями $x...y'y$ и $\Delta'$-путями $x...y'$, так как несложно убедиться в том, что $R(y)=R(1y')=\{y'\}$.
\renewcommand{\labelenumiii}{\arabic{enumi}.\arabic{enumii}.\arabic{enumiii}$^\circ$}

Далее рассмотрим два случая:
\begin{enumerate}
    \item Пусть $x\ne y$.
    
    В данном случае $z=h(y,x)<\#y$, а значит $h(y',x)=z$. С этими знаниями мы готовы считать. Посчитаем:
$$o\left(y,x,\overline{\Delta}\right)=\sum_{i=0}^{|y|}\left( {f\left(y,i,z\right)}\prod_{j=1}^{d(x)}\left(g\left(x,j\right)-i\right)\prod_{j=1}^{s\left(\overline{\Delta}\right)}\left(e\left(\overline{\Delta},j\right)-i\right)\right)=$$
$$=\sum_{i=0}^{|y|}\left( {f\left(1y',i,z\right)}\prod_{j=1}^{d(x)}\left(g\left(x,j\right)-i\right)\prod_{j=1}^{s\left(\overline{\Delta}\right)}\left(e\left(\overline{\Delta},j\right)-i\right)\right)=$$
$$=\text{(По Лемме \ref{nol})}= $$
$$=\sum_{i=0}^{|y'|}\left( \frac{f\left(y',i,z\right)}{|1y'|-i}\prod_{j=1}^{d(x)}\left(g\left(x,j\right)-i\right)\left(\prod_{j=1}^{s\left(\Delta'\right)}\left(e\left(\Delta',j\right)-i\right)\right)(|1y'|-i)\right)+$$
$$+{f\left(1y',|1y'|,z\right)}\prod_{j=1}^{d(x)}\left(g\left(x,j\right)-|1y'|\right)\prod_{j=1}^{s\left(\overline{\Delta}\right)}\left(e\left(\overline{\Delta},j\right)-|1y'|\right).$$

    Мы знаем, что $e\left(\overline{\Delta},s(\Delta)\right)=|y|=|y'|+1$, а значит
    $$ {f\left(1y',|y'|+1,z\right)}\prod_{j=1}^{d(x)}\left(g\left(x,j\right)-(|y'|+1)\right)\prod_{j=1}^{s\left(\overline{\Delta}\right)}\left(e\left(\overline{\Delta},j\right)-(|y'+1|)\right)=0.$$

Таким образом, наше выражение равняется следующему выражению:
$$\sum_{i=0}^{|y'|}\left( {f\left(y',i,z\right)}\prod_{j=1}^{d(x)}\left(g\left(x,j\right)-i\right)\prod_{j=1}^{s\left(\Delta'\right)}\left(e\left(\Delta',j\right)-i\right)\right)=$$
$$=o(y',x,\Delta')=\text{(По предположению)}=d(y',x,\Delta')=d\left(y,x,\overline{\Delta}\right).$$

Что и требовалось.

    \item Пусть $x= y$.
    
    В данном случае $z=h(y,x)=\#y$, а значит $z':=h(y',x)=z-1$. Счёт будет аналогичным предыдущему случаю, однако, чуть более длинным:
$$o\left(y,x,\overline{\Delta}\right)=\sum_{i=0}^{|y|}\left( {f\left(y,i,z\right)}\prod_{j=1}^{d(x)}\left(g\left(x,j\right)-i\right)\prod_{j=1}^{s\left(\overline{\Delta}\right)}\left(e\left(\overline{\Delta},j\right)-i\right)\right)=$$
$$=\sum_{i=0}^{|y|}\left( {f\left(1y',i,z\right)}\prod_{j=1}^{d(x)}\left(g\left(x,j\right)-i\right)\prod_{j=1}^{s\left(\overline{\Delta}\right)}\left(e\left(\overline{\Delta},j\right)-i\right)\right)=$$
$$=\text{(По Утверждению \ref{uroven})}=$$
$$=\sum_{i=0}^{|y'|}\left( {f\left(1y',i,z-1\right)}\prod_{j=1}^{d(x)}\left(g\left(x,j\right)-i\right)\prod_{j=1}^{s\left(\overline{\Delta}\right)}\left(e\left(\overline{\Delta},j\right)-i\right)\right)+$$
$$+{f\left(1y',|y'|+1,z\right)}\prod_{j=1}^{d(x)}\left(g\left(x,j\right)-(|y'|+1)\right)\prod_{j=1}^{s\left(\overline{\Delta}\right)}\left(e\left(\overline{\Delta},j\right)-(|y'+1|)\right)=$$
$$=\text{(По Лемме \ref{nol})}= $$
$$=\sum_{i=0}^{|y'|}\left( \frac{f\left(y',i,z'\right)}{|1y'|-i}\prod_{j=1}^{d(x)}\left(g\left(x,j\right)-i\right)\left(\prod_{j=1}^{s\left(\Delta'\right)}\left(e\left(\Delta',j\right)-i\right)\right)(|1y'|-i)\right)+$$
$$+{f\left(1y',|y'|+1,z\right)}\prod_{j=1}^{d(x)}\left(g\left(x,j\right)-(|y'|+1)\right)\prod_{j=1}^{s\left(\overline{\Delta}\right)}\left(e\left(\overline{\Delta},j\right)-(|y'+1|)\right).$$

    Мы знаем, что $e\left(\overline{\Delta},s(\Delta)\right)=|y|=|y'|+1$, а значит
    $$ {f\left(1y',|y'|+1,z\right)}\prod_{j=1}^{d(x)}\left(g\left(x,j\right)-(|y'|+1)\right)\prod_{j=1}^{s\left(\overline{\Delta}\right)}\left(e\left(\overline{\Delta},j\right)-(|y'+1|)\right)=0.$$

Таким образом, наше выражение равняется следующему выражению:
$$\sum_{i=0}^{|y'|}\left( {f\left(y',i,z'\right)}\prod_{j=1}^{d(x)}\left(g\left(x,j\right)-i\right)\prod_{j=1}^{s\left(\Delta'\right)}\left(e\left(\Delta',j\right)-i\right)\right)=$$
$$=o(y',x,\Delta')=\text{(По предположению)}=d(y',x,\Delta')=d\left(y,x,\overline{\Delta}\right).$$
Что и требовалось.
\end{enumerate}
\end{enumerate}
\end{enumerate}

Все случаи разобраны.

Теорема доказана.

\end{proof}

\begin{Col}
Пусть $x\in\mathbb{YF}$ и $\Delta\in\Omega(0,x)$. Тогда
$$d(\varepsilon,x,\Delta)=\prod_{j=1}^{d(x)}g\left(x,j\right)\prod_{j=1}^{s(\Delta)}e\left(\Delta,j\right).$$
\end{Col} 
\begin{proof}
$$d(\varepsilon,x,\Delta)=(\text{По Теореме \ref{delta}})=$$
$$=\sum_{i=0}^{|\varepsilon|}\left( {f\left(\varepsilon,i,h(\varepsilon,x)\right)}\prod_{j=1}^{d(x)}\left(g\left(x,j\right)-i\right)\prod_{j=1}^{s(\Delta)}\left(e\left(\Delta,j\right)-i\right)\right)=$$
$$=\prod_{j=1}^{d(x)}g\left(x,j\right)\prod_{j=1}^{s(\Delta)}e\left(\Delta,j\right).$$
\end{proof}

\subsection{Теорема о количестве $S$-путей между двумя вершинами в графе Юнга-Фибоначчи}

Итак, пусть $x,y\in \mathbb{YF}$, $|y| \ge |x|$, $S\in \mathbb{N}_0$.

Мы готовы искать число $yx$-$S$-путей в $\mathbb{YF}$, то есть $d(x,y,S)$.

\begin{Oboz}
Пусть $X,Y,S \in \mathbb{N}_0$. Тогда множество $$\left\{\Delta \in \Omega(X,Y): s(\Delta)=S\right\}$$ обозначим за $\Omega(X,Y,S)$.
\end{Oboz}

Из Определения \ref{sopr} (определения $S$-пути) следует
\begin{Zam}
При данных $x,y\in \mathbb{YF}$: $|y| \ge |x|$, $S\in \mathbb{N}_0$ $xy$-$S$-пути -- это в точности $xy$-$\Delta$-пути при всех $\Delta$: $s(\Delta)=S$, в частности,
$$d(x,y,S)=\sum_{\Delta\in\Omega(|x|,|y|,S)}d(x,y,\Delta).$$
\end{Zam}

Таким образом, мы сразу можем заметить, что
\begin{Nab}
$$d(x,y,S)=\sum_{\Delta\in\Omega(|x|,|y|,S)}d(x,y,\Delta)=$$
$$=\sum_{\Delta \in \Omega(|x|,|y|,S)} \sum_{i=0}^{|x|}\left( {f\left(x,i,z\right)}\prod_{j=1}^{d(y)}\left(g\left(y,j\right)-i\right)\prod_{j=1}^{s(\Delta)}\left(e\left(\Delta,j\right)-i\right)\right)=$$
$$= \sum_{i=0}^{|x|}\left(\left( \sum_{\Delta \in \Omega(x,y,S)}  \prod_{j=1}^{S}\left(e\left(\Delta,j\right)-i\right) \right) {f\left(x,i,z\right)}\prod_{j=1}^{d(y)}\left(g\left(y,j\right)-i\right)\right).$$
\end{Nab}

То есть, у нас уже есть  формула для $d(x,y,S)$. Проблема в том, что в ней слишком много слагаемых. Чтобы уменьшить число слагаемых, надо решить следующую задачу:
\begin{Problem} \label{wehavesome}
При $x,y\in\mathbb{YF}$: $|y|\ge|x|$ и $i \in \{0,...,|x|\}$ свернуть следующее выражение:
$$\sum_{\Delta \in \Omega(|x|,|y|,S)}  \prod_{j=1}^{S}\left(e\left(\Delta,j\right)-i\right).$$
\end{Problem}

Итак, давайте решать эту задачу.

Как устроено выражение из Задачи \ref{wehavesome}?

Вспомним, что $e(\Delta,j)$ -- это значение $\Delta_i$, которое достигается после $j$-ого шага ``вверх''. Как может быть утроена последовательность $$e(\Delta,1),e(\Delta,2),...,e(\Delta,S)?$$ 

\begin{Prop} \label{first}
Пусть $X,Y,S\in\mathbb{N}_0$.

Пусть также последовательность
$e(\Delta,1),$ $e(\Delta,2)$,...,$e(\Delta,S)$ удовлетворяет следующим четырём условиям:
\begin{enumerate}
\item $e(\Delta,1) \le Y+1$;

\item $\forall i \in \{1,...,S-1\} \quad e(\Delta,i+1) \le e(\Delta,i)+1 $;

\item $e(\Delta,S) \ge X$;

\item $\forall i \in \{1,...,S\}\quad e(\Delta,i) \ge 1 $.
\end{enumerate}
Тогда существует и единственна траектория $\Delta\in\Omega(X,Y,S)$ с такой средней функцией.
\end{Prop}

\begin{proof}
$\;$

\begin{tikzpicture}

\draw[ultra thick, ->] (0.5,0) -- (0.5,7);
\draw[thin, dashed, ->] (1.5,0) -- (1.5,7);
\draw[thin, dashed, ->] (2.5,0) -- (2.5,7);
\draw[thin, dashed, ->] (3.5,0) -- (3.5,7);
\draw[thin, dashed, ->] (4.5,0) -- (4.5,7);
\draw[thin, dashed, ->] (5.5,0) -- (5.5,7);
\draw[thin, dashed, ->] (6.5,0) -- (6.5,7);
\draw[thin, dashed, ->] (7.5,0) -- (7.5,7);
\draw[thin, dashed, ->] (8.5,0) -- (8.5,7);
\draw[thin, dashed, ->] (9.5,0) -- (9.5,7);
\draw[thin, dashed, ->] (10.5,0) -- (10.5,7);
\draw[thin, dashed, ->] (11.5,0) -- (11.5,7);
\draw[thin, dashed, ->] (12.5,0) -- (12.5,7);
\draw[thin, dashed, ->] (13.5,0) -- (13.5,7);
\draw[thin, dashed, ->] (14.5,0) -- (14.5,7);

\draw[ultra thick, ->] (0,0.5) -- (15,0.5);
\draw[thin, dashed, ->] (0,1.5) -- (15,1.5);
\draw[thin, dashed, ->] (0,2.5) -- (15,2.5);
\draw[thin, dashed, ->] (0,3.5) -- (15,3.5);
\draw[thin, dashed, ->] (0,4.5) -- (15,4.5);
\draw[thin, dashed, ->] (0,5.5) -- (15,5.5);
\draw[thin, dashed, ->] (0,6.5) -- (15,6.5);

\draw[ultra thick, red, dotted, ->] (0.5,3.5) -- (1.5,2.5);
\draw[ultra thick, red, dotted, ->] (1.5,2.5) -- (2.5,1.5);
\draw[ultra thick, red, ->] (2.5,1.5) -- (3.5,2.5);
\draw[ultra thick, red, dotted, ->] (3.5,2.5) -- (4.5,1.5);
\draw[ultra thick, red, dotted, ->] (4.5,1.5) -- (5.5,0.5);
\draw[ultra thick, red, ->] (5.5,0.5) -- (6.5,1.5);
\draw[ultra thick, red, ->] (6.5,1.5) -- (7.5,2.5);
\draw[ultra thick, red, dotted, ->] (7.5,2.5) -- (8.5,1.5);
\draw[ultra thick, red, ->] (8.5,1.5) -- (9.5,2.5);
\draw[ultra thick, red, ->] (9.5,2.5) -- (10.5,3.5);
\draw[ultra thick, red, dotted, ->] (10.5,3.5) -- (11.5,2.5);
\draw[ultra thick, red, dotted, ->] (11.5,2.5) -- (12.5,1.5);
\draw[ultra thick, red, dotted, ->] (12.5,1.5) -- (13.5,0.5);
\draw[ultra thick, red, ->] (13.5,0.5) -- (14.5,1.5);

\path (0.25,0.25) node  {$0$}
-- (1.25,0.25)  node  {$1$}
-- (2.25,0.25)  node  {$2$}
-- (3.25,0.25)  node  {$3$}
-- (4.25,0.25)  node  {$4$}
-- (5.25,0.25)  node  {$5$}
-- (6.25,0.25)  node  {$6$}
-- (7.25,0.25)  node  {$7$}
-- (8.25,0.25)  node  {$8$}
-- (9.25,0.25)  node  {$9$}
-- (10.25,0.25)  node  {$10$}
-- (11.25,0.25)  node  {$11$}
-- (12.25,0.25)  node  {$12$}
-- (13.25,0.25)  node  {$13$}
-- (14.25,0.25)  node  {$14$};

\path (0.25,1.25)  node  {$1$}
-- (0.25,2.25)  node  {$2$}
-- (0.25,3.25)  node  {$3$}
-- (0.25,4.25)  node  {$4$}
-- (0.25,5.25)  node  {$5$}
-- (0.25,6.25)  node  {$6$};

\path (3.25,2.75) node [font=\Large,red]  {$2$}
-- (6.25,1.75) node [font=\Large,red]  {$1$}
-- (7.25,2.75) node [font=\Large,red]  {$2$}
-- (9.25,2.75) node [font=\Large,red]  {$2$}
-- (10.25,3.75) node [font=\Large,red]  {$3$}
-- (14.25,1.75) node [font=\Large,red]  {$1$};

\end{tikzpicture}

\underline{\textbf{Существование}}:

На рисунке изображено построение $\Delta\in\Omega(1,3,6)$ по последовательности
$$e(\Delta,1)=2,\;e(\Delta,2)=1,\;e(\Delta,3)=2,\;e(\Delta,4)=2,\;e(\Delta,5)=3,\;e(\Delta,6)=1$$

Теперь в общем случае:

Построим следующую $\Delta\in\Omega(x,y,S)$.

Перед шагами ``вверх'', между шагами ``вверх'' и после шагов ``вверх'' вставим нужное число шагов ``вниз''. Формально:
\begin{itemize}
    \item $\Delta_0=Y$;
    \item $\Delta_1=Y-1$;
    \item $\Delta_2=Y-2$;
    \item ...
    \item $\Delta_{Y-e(\Delta,1)+1}=e(\Delta,1)-1$;
    \item $\Delta_{Y-e(\Delta,1)+2}=e(\Delta,1)$;
    \item $\Delta_{Y-e(\Delta,1)+3}=e(\Delta,1)-1$;
    \item $\Delta_{Y-e(\Delta,1)+4}=e(\Delta,1)-2$;
    \item ...
    \item $\Delta_{Y-e(\Delta,2)+3}=e(\Delta,2)-1$;
    \item ... ...
    \item $\Delta_{Y-e(\Delta,S)+2S-1}=e(\Delta,S)-1$;
    \item $\Delta_{Y-e(\Delta,S)+2S}=e(\Delta,S)$;
    \item $\Delta_{Y-e(\Delta,S)+2S+1}=e(\Delta,S)-1$;
    \item ...
    \item $\Delta_{Y-X+2S}=X$.
\end{itemize}

Давайте убедимся в том, что мы действительно построили $\Delta\in\Omega(X,Y,S)$:

\begin{itemize}
    \item $\Delta_0=Y$, $\Delta_{|\Delta|}=X$, $|\Delta|=S$ -- очевидно.
    \item Построенная траектория $\Delta$ состоит из $(S+1)$-ого цикла, каждый из которых состоит какого-то (может быть, нулевого) количества шагов ``вниз'', причём между каждыми двумя соседними циклами есть ровно один шаг ``вверх''.
    \item Первый цикл шагов ``вниз'' состоит из $Y-e(\Delta,1)+1$ шагов ``вниз''. По первому условию $Y-e(\Delta,1)+1 \ge 0$, а значит построение первого цикла шагов ``вниз'' корректно.
    \item Цикл шагов ``вниз'' между $j$-ым и $(j+1)$-ым шагом ``вверх'' ($j\in\{1,...,S-1\}$) состоит из $e(\Delta,j)-e(\Delta,j+1)+1$ шагов ``вниз''. По второму условию $e(\Delta,j)-e(\Delta,j+1)+1 \ge 0$,  а значит построение этих циклов шагов ``вниз'' корректно.
    \item Последний цикл шагов ``вниз'' состоит из $e(\Delta,S)-X$ шагов ``вниз''. По третьему условию $e(\Delta,S)-X \ge 0$, а значит построение последнего цикла шагов ``вниз'' корректно.
    \item Каждый цикл шагов ``вниз'' заканчивается на $\Delta_i=e(\Delta,j)-1$ при $j\in\{1,...,S-1\}$ или на $X-1$, а значит по четвёртому условию $\forall i\in \{0,...,|\Delta|\}$ $\Delta_i \ge 0$. 
\end{itemize}

Убедились.

\underline{\textbf{Единственность}}:

Пусть у нас есть другая траектория $\Delta \in \Omega(x,y,S)$ с данной последовательностью $e(\Delta,1)$,..., $e(\Delta,1)$. Рассмотрим её в следующем виде:
\begin{itemize}
    \item $a_0$ шагов ``вниз'';
    \item шаг ``вверх'';
    \item $a_1$ шагов ``вниз'';
    \item шаг ``вверх'';
    \item ...
    \item $a_{S-1}$ шагов ``вниз'';
    \item шаг ``вверх'';
    \item $a_{S}$ шагов ``вниз''.
\end{itemize}

Последовательно замечаем, что если
\begin{itemize}
    \item $a_0 \ne Y-e(\Delta,1)+1$, то $e(\Delta,1)$ получается неправильным;
    \item (при $j\in\{1,...,S-1\}$) $a_j\ne e(\Delta,j)-e(\Delta,j+1)+1$, то $e(\Delta,j+1)$ получается неправильным;
    \item $a_S \ne e(\Delta,S)-X$, то $\Delta_{|\Delta|}\ne X$.
\end{itemize}

Таким образом, траектория, построенная выше, является единственной.
\end{proof}

\begin{Prop} \label{second}
Пусть $X,Y,S\in\mathbb{N}_0$.

Тогда любая траектория $\Delta\in\Omega(X,Y,S)$ удовлетворяет следующим четырём условиям:
\begin{enumerate}
\item $e(\Delta,1) \le Y+1$;

\item $\forall i \in \{1,...,S-1\}\quad e(\Delta,i+1) \le e(\Delta,i)+1 $;

\item $e(\Delta,S) \ge X$;

\item $\forall i \in \{1,...,S\}\quad e(\Delta,i) \ge 1 $.
\end{enumerate}
\end{Prop}
\begin{proof}
Рассмотрим траекторию $\Delta$ в следующем виде:
\begin{itemize}
    \item $a_0$ шагов ``вниз'';
    \item шаг ``вверх'';
    \item $a_1$ шагов ``вниз'';
    \item шаг ``вверх'';
    \item ...
    \item $a_{S-1}$ шагов ``вниз'';
    \item шаг ``вверх'';
    \item $a_{S}$ шагов ``вниз''.
\end{itemize}

Ясно, что $a_0\ge0,...,\;
a_S\ge0$.

Кроме того, ясно, что:
\begin{itemize}
    \item $e(\Delta,1)=Y-a_0+1\le Y+1$;
    \item (при
    $j\in\{1,...,S-1\}$) $e(\Delta,j+1)=e(\Delta,j)-a_j+1\le e(\Delta,j)+1$;
    \item $X=e(\Delta,S)-a_S\le e(\Delta,1)$;
    \item (при $j\in\{1,...,S\}$) $\exists i_j:$ $e(\Delta,j)=\Delta_i+1\ge 1$.
\end{itemize}
\end{proof}

\begin{Def}
Пусть $X,Y,S\in\mathbb{N}_0$.

Тогда последовательность $\lambda=\{\lambda_1,\lambda_2,...,\lambda_S\}$: $\forall i\in\{1,...,S\}$ $\lambda_i \in \mathbb{Z}$  будем называть хорошей $(X,Y,S)$-последовательностью, если она удовлетворяет следующим условиям:
\begin{enumerate}
\item $\lambda_1 \le Y+1$;

\item $\forall i \in \{1,...,S-1\}\quad \lambda_{i+1} \le \lambda_{i}+1 $;

\item $\lambda_{S} \ge X$;

\item $\forall i \in \{1,...,S\}\quad\lambda_{i} \ge 1 $.
\end{enumerate}

\begin{Oboz} 
$\;$
\begin{itemize}
    \item Пусть $X,Y,S\in \mathbb{N}_0$. Тогда множество хороших $(X,Y,S)$-последовательностей обозначим за $\Lambda(X,Y,S)$;

    \item Множество всех хорших $(X,Y,S)$-последовательностей при каких-то $X,Y,S\in \mathbb{N}_0$ обозначим за $\Lambda$.
\end{itemize}

\end{Oboz}

\begin{Col}[Из Утверждений \ref{first} и \ref{second}] \label{Kesha}

Пусть $X,Y,S\in\mathbb{N}_0$.

Тогда существует биекция $B(X,Y,S):\Omega(X,Y,S)\leftrightarrow\Lambda(X,Y,S)$, такая, что если $B(\Delta)=\lambda$, то $\forall j\in\{1,...,S\}$ 
$$e(\Delta,j)=\lambda_j.$$
В частности,
$$\sum_{\Delta \in \Omega(X,Y,S)}  \prod_{j=1}^{S}e\left(\Delta,j\right)=\sum_{\lambda \in \Lambda(X,Y,S)}  \prod_{j=1}^{S}\lambda_j.$$
\end{Col}
\end{Def}

\begin{Def}
Пусть $X,Y\in \mathbb{N}_0$: $Y\ge X$. Тогда последовательность $\Delta=\{\Delta_i: i \in \{0,1,...,k\}\}$ при некотором $k \in \mathbb{N}_0$ назовём $(Y;X)$-псевдотраекторией, если выполняются следующие условия:
\begin{enumerate}
    \item $ \forall i \in \{1,...,k\} \quad \Delta_i-\Delta_{i-1} = \pm 1$;
    \item$ \Delta_0=Y $;
    \item$ \Delta_k=X$.
\end{enumerate}
\end{Def}

\begin{Zam}
Траектории -- это такие псевдотраектории, что $\Delta_i$ всегда неотрицательны.
\end{Zam}

\begin{Oboz}
$\;$
\begin{itemize}
    \item Пусть $X,Y\in \mathbb{N}_0$. Тогда множество $(Y; X)$-псевдотраекторий обозначим за $\Omega'(X,Y)$.
    \item Множество всех $(Y;X)$-псевдотраекторий при каких-то $X,Y\in \mathbb{N}_0$ обозначим за $\Omega'$.
    \item Кроме того, давайте для псевдотраекторий использовать все определения и обозначения, которые были даны для траекторий.
\end{itemize}
\end{Oboz}

\begin{Oboz} При $X,Y,S,i\in\mathbb{N}_0$: $Y  \ge X \ge i \ge 0$ введём следующие обозначения:
    \begin{itemize}
        \item $$\Xi(X,Y,S,i):=\sum_{\Delta \in \Omega(X,Y,S)}  \prod_{j=1}^{S}\left(e\left(\Delta,j\right)-i\right);$$

        \item $$\Xi'(X,Y,S,i)=\sum_{\Delta \in \Omega'(X,Y,S)}  \prod_{j=1}^{S}\left(e\left(\Delta,j\right)-i\right).$$
    \end{itemize}
\end{Oboz}

\begin{Prop} \label{strih}
Пусть $X,Y,S,i\in\mathbb{N}_0$: $Y  \ge X \ge i \ge 0$. Тогда
$$\Xi(X,Y,S,i)=\Xi'(X,Y,S,i).$$
\end{Prop}
\begin{proof}

$\;$

Мы знаем, что если $\Delta \in \Omega(X,Y,S)$, то $\Delta \in \Omega'(X,Y,S)$, поэтому нам достаточно доказать, что 
$$\sum_{\Delta \in \Omega'(X,Y,S)\textbackslash\Omega(X,Y,S)}  \prod_{j=1}^{S}\left(e\left(\Delta,j\right)-i\right).$$

Рассмотрим $\Delta \in \Omega'(X,Y,S) \textbackslash \Omega(X,Y,S)$.

Ясно, что в данном случае $\exists i\in\{0,...,|\Delta|\}: \Delta_i < 0$.

Пусть после $i$-ого шага траектория $\Delta$ перед ближайшим шагом ``вверх'' состоит из $a$ шагов ``вниз'', и пусть ближайший шаг ``вверх'' -- $j$-ый шаг ``вверх'. Тогда ясно, что $a\ge0$, и поэтому $e(\Delta,j)=\Delta_i-a+1\le \Delta_i+1 \le 0 \le i$. 

Пусть $e(\Delta,j)<i$. Помним, что $\forall k \in \{1,...,S-1\}\quad e(\Delta,k+1) \le e(\Delta,k)+1 $, а поэтому $e(\Delta,j+1) \le e(\Delta,j)+1 \le i$.

Повторим это рассуждение несколько раз (после рассуждения с $e(\Delta,j)$ проделаем рассуждение с $e(\Delta,j+1)$, $e(\Delta,j+2)$,..., $e(\Delta,S-1)$). Так мы либо найдём $l\in\{j,...,S\}:e(\Delta,l) = i$, либо при каждом $l\in\{j,...,S\}$ будем получать, что $e(\Delta,l) \le i$, и в итоге придём к тому, что $e(\Delta,S) \le i$.

В таком случае вспомним, что $i\le X$, а значит $e(\Delta,S) \le X$. По Утверждению \ref{second} $e(\Delta,S) \le X$, поэтому $e(\Delta,S) = X$, а таким образом, из того, что $e(\Delta,S) \le i\le X$, следует, что $e(\Delta,S)= i = X$.

Таким образом, в любом случае мы нашли $l\in\{1,...,S\}:e(\Delta,l) = i$, а значит  для любой траектории $\Delta \in \Omega'(X,Y,S) \textbackslash \Omega(X,Y,S)$ 
$$\prod_{j=1}^{S}\left(e\left(\Delta,j\right)-i\right)=0,$$
а поэтому
$$\sum_{\Delta \in \Omega'(X,Y,S)\textbackslash\Omega(X,Y,S)}  \prod_{j=1}^{S}\left(e\left(\Delta,j\right)-i\right)=0.$$
Что и требовалось.
\end{proof}

\begin{Prop} \label{sdvig}
Пусть $X,Y,S,i\in\mathbb{N}_0$: $Y  \ge X \ge i \ge 0$. Тогда
$$\Xi(X,Y,S,i)=\Xi(X-i,Y-i,S,0).$$
\end{Prop}
\begin{proof}
Из Утверждения \ref{strih} мнговенно следует, что нам достаточно доказать, что $$\Xi'(X,Y,S,i)=\Xi'(X-i,Y-i,S,0).$$

Давайте это докажем.

Введём отображение $M:\Omega'(X,Y,S) \rightarrow \Omega'(X-i,Y-i,S)$ следующим образом:

\begin{tikzpicture}

\draw[ultra thick, ->] (0.5,0) -- (0.5,7);
\draw[thin, dashed, ->] (1.5,0) -- (1.5,7);
\draw[thin, dashed, ->] (2.5,0) -- (2.5,7);
\draw[thin, dashed, ->] (3.5,0) -- (3.5,7);
\draw[thin, dashed, ->] (4.5,0) -- (4.5,7);
\draw[thin, dashed, ->] (5.5,0) -- (5.5,7);
\draw[thin, dashed, ->] (6.5,0) -- (6.5,7);
\draw[thin, dashed, ->] (7.5,0) -- (7.5,7);
\draw[thin, dashed, ->] (8.5,0) -- (8.5,7);
\draw[thin, dashed, ->] (9.5,0) -- (9.5,7);
\draw[thin, dashed, ->] (10.5,0) -- (10.5,7);
\draw[thin, dashed, ->] (11.5,0) -- (11.5,7);
\draw[thin, dashed, ->] (12.5,0) -- (12.5,7);
\draw[thin, dashed, ->] (13.5,0) -- (13.5,7);
\draw[thin, dashed, ->] (14.5,0) -- (14.5,7);

\draw[ultra thick, ->] (0,0.5) -- (15,0.5);
\draw[thin, dashed, ->] (0,1.5) -- (15,1.5);
\draw[thin, dashed, ->] (0,2.5) -- (15,2.5);
\draw[thin, dashed, ->] (0,3.5) -- (15,3.5);
\draw[thin, dashed, ->] (0,4.5) -- (15,4.5);
\draw[thin, dashed, ->] (0,5.5) -- (15,5.5);
\draw[thin, dashed, ->] (0,6.5) -- (15,6.5);

\draw[ultra thick, red, ->] (0.5,5.5) -- (1.5,4.5);
\draw[ultra thick, red, ->] (1.5,4.5) -- (2.5,3.5);
\draw[ultra thick, red, ->] (2.5,3.5) -- (3.5,2.5);
\draw[ultra thick, red, ->] (3.5,2.5) -- (4.5,1.5);
\draw[ultra thick, red, ->] (4.5,1.5) -- (5.5,0.5);
\draw[ultra thick, red, ->] (5.5,0.5) -- (6.5,1.5);
\draw[ultra thick, red, ->] (6.5,1.5) -- (7.5,2.5);
\draw[ultra thick, red, ->] (7.5,2.5) -- (8.5,1.5);
\draw[ultra thick, red, ->] (8.5,1.5) -- (9.5,2.5);
\draw[ultra thick, red, ->] (9.5,2.5) -- (10.5,3.5);
\draw[ultra thick, red, ->] (10.5,3.5) -- (11.5,2.5);
\draw[ultra thick, red, ->] (11.5,2.5) -- (12.5,1.5);
\draw[ultra thick, red, ->] (12.5,1.5) -- (13.5,2.5);
\draw[ultra thick, red, ->] (13.5,2.5) -- (14.5,3.5);

\path (0.25,0.25) node  {$0$}
-- (1.25,0.25)  node  {$1$}
-- (2.25,0.25)  node  {$2$}
-- (3.25,0.25)  node  {$3$}
-- (4.25,0.25)  node  {$4$}
-- (5.25,0.25)  node  {$5$}
-- (6.25,0.25)  node  {$6$}
-- (7.25,0.25)  node  {$7$}
-- (8.25,0.25)  node  {$8$}
-- (9.25,0.25)  node  {$9$}
-- (10.25,0.25)  node  {$10$}
-- (11.25,0.25)  node  {$11$}
-- (12.25,0.25)  node  {$12$}
-- (13.25,0.25)  node  {$13$}
-- (14.25,0.25)  node  {$14$};

\path (0.25,1.25)  node  {$1$}
-- (0.25,2.25)  node  {$2$}
-- (0.25,3.25)  node  {$3$}
-- (0.25,4.25)  node  {$4$}
-- (0.25,5.25)  node  {$5$}
-- (0.25,6.25)  node  {$6$};

\end{tikzpicture}

$$$$

\begin{tikzpicture}

\draw[ultra thick, ->] (0.5,0) -- (0.5,7);
\draw[thin, dashed, ->] (1.5,0) -- (1.5,7);
\draw[thin, dashed, ->] (2.5,0) -- (2.5,7);
\draw[thin, dashed, ->] (3.5,0) -- (3.5,7);
\draw[thin, dashed, ->] (4.5,0) -- (4.5,7);
\draw[thin, dashed, ->] (5.5,0) -- (5.5,7);
\draw[thin, dashed, ->] (6.5,0) -- (6.5,7);
\draw[thin, dashed, ->] (7.5,0) -- (7.5,7);
\draw[thin, dashed, ->] (8.5,0) -- (8.5,7);
\draw[thin, dashed, ->] (9.5,0) -- (9.5,7);
\draw[thin, dashed, ->] (10.5,0) -- (10.5,7);
\draw[thin, dashed, ->] (11.5,0) -- (11.5,7);
\draw[thin, dashed, ->] (12.5,0) -- (12.5,7);
\draw[thin, dashed, ->] (13.5,0) -- (13.5,7);
\draw[thin, dashed, ->] (14.5,0) -- (14.5,7);

\draw[ultra thick, ->] (0,2.5) -- (15,2.5);
\draw[thin, dashed, ->] (0,1.5) -- (15,1.5);
\draw[thin, dashed, ->] (0,0.5) -- (15,0.5);
\draw[thin, dashed, ->] (0,3.5) -- (15,3.5);
\draw[thin, dashed, ->] (0,4.5) -- (15,4.5);
\draw[thin, dashed, ->] (0,5.5) -- (15,5.5);
\draw[thin, dashed, ->] (0,6.5) -- (15,6.5);

\draw[ultra thick, blue, ->] (0.5,5.5) -- (1.5,4.5);
\draw[ultra thick, blue, ->] (1.5,4.5) -- (2.5,3.5);
\draw[ultra thick, blue, ->] (2.5,3.5) -- (3.5,2.5);
\draw[ultra thick, blue, ->] (3.5,2.5) -- (4.5,1.5);
\draw[ultra thick, blue, ->] (4.5,1.5) -- (5.5,0.5);
\draw[ultra thick, blue, ->] (5.5,0.5) -- (6.5,1.5);
\draw[ultra thick, blue, ->] (6.5,1.5) -- (7.5,2.5);
\draw[ultra thick, blue, ->] (7.5,2.5) -- (8.5,1.5);
\draw[ultra thick, blue, ->] (8.5,1.5) -- (9.5,2.5);
\draw[ultra thick, blue, ->] (9.5,2.5) -- (10.5,3.5);
\draw[ultra thick, blue, ->] (10.5,3.5) -- (11.5,2.5);
\draw[ultra thick, blue, ->] (11.5,2.5) -- (12.5,1.5);
\draw[ultra thick, blue, ->] (12.5,1.5) -- (13.5,2.5);
\draw[ultra thick, blue, ->] (13.5,2.5) -- (14.5,3.5);

\path (0.25,0.25) node  {$-2$}
-- (1.25,2.25)  node  {$1$}
-- (2.25,2.25)  node  {$2$}
-- (3.25,2.25)  node  {$3$}
-- (4.25,2.25)  node  {$4$}
-- (5.25,2.25)  node  {$5$}
-- (6.25,2.25)  node  {$6$}
-- (7.25,2.25)  node  {$7$}
-- (8.25,2.25)  node  {$8$}
-- (9.25,2.25)  node  {$9$}
-- (10.25,2.25)  node  {$10$}
-- (11.25,2.25)  node  {$11$}
-- (12.25,2.25)  node  {$12$}
-- (13.25,2.25)  node  {$13$}
-- (14.25,2.25)  node  {$14$};

\path (0.25,1.25)  node  {$-1$}
-- (0.25,2.25)  node  {$0$}
-- (0.25,3.25)  node  {$1$}
-- (0.25,4.25)  node  {$2$}
-- (0.25,5.25)  node  {$3$}
-- (0.25,6.25)  node  {$4$};
\end{tikzpicture}

На рисунках изображено, как по псевдотраектории $\Delta$ (верхний рисунок) строится псевдотраектория $М(\Delta)$ (нижний рисунок). В данном случае операция $M$ -- это просто ``сдвиг'' псевдотраектории $\Delta$ на $i$ ``уровней'' ``вниз''. Формально:
$$|M(\Delta)|=|{\Delta}|; \;   \forall j \in \left\{ 0,...,|M(\Delta)| \right\} \;  M(\Delta)_{j}=\Delta_{j}-i.$$

Несложно заметить, что 
\begin{itemize}
    \item $ \Delta'\in \Omega'(X-i,Y-i,S)$;
    \item $\forall j \in \{1,...,S\} \quad e(M(\Delta),j)=e(\Delta,j)-i$;
    \item $|M(\Delta)|=|\Delta|$;
    \item обратное отображение к $M$ -- это отображение $M':\Omega'(X-i,Y-i,S) \rightarrow \Omega'(X,Y,S)$. Оно является просто ``сдвигом'' данной псевдотраектории на $i$ ``уровней'' ``вверх''. Формально:
    $$|M'(\Delta)|=|{\Delta}|; \;   \forall j \in \left\{ 0,...,|M(\Delta)| \right\} \;  M(\Delta)_{j}=\Delta_{j}+i;$$
    \item M -- биекция между $\Omega'(X,Y,S)$ и $\Omega'(X-i,Y-i,S)$.
\end{itemize}

Таким образом,
$$\prod_{j=1}^{S}\left(e\left(\Delta,j\right)-i\right)=\prod_{j=1}^{S}e\left(M(\Delta),j\right).$$

Просуммировав это равенство по $\Omega'(X,Y,S)$, мы получим, что
$$\sum_{\Delta \in \Omega'(X,Y,S)}  \prod_{j=1}^{S}\left(e\left(\Delta,j\right)-i\right)=\sum_{\Delta \in \Omega'(X,Y,S)}  \prod_{j=1}^{S}e\left(M(\Delta),j\right) \Longleftrightarrow$$
$$\Longleftrightarrow\text{(Так как $M$ -- биекция)}\Longleftrightarrow$$
$$\Longleftrightarrow\sum_{\Delta \in \Omega'(X,Y,S)}  \prod_{j=1}^{S}\left(e\left(\Delta,j\right)-i\right)=\sum_{\Delta \in \Omega'(X-i,Y-i,S)}  \prod_{j=1}^{S}e\left(\Delta,j\right) \Longleftrightarrow$$
$$\Longleftrightarrow \Xi'(X,Y,S,i)=\Xi'(X-i,Y-i,S,0)\Longleftrightarrow$$
$$\Longleftrightarrow \text{(по Утверждению \ref{strih})}\Longleftrightarrow$$
$$\Longleftrightarrow \Xi(X,Y,S,i)=\Xi(X-i,Y-i,S,0).$$
\end{proof}

Таким образом, нам достаточно найти только $\Xi(X,Y,S,0)$ при $X,Y,S\in\mathbb{N}_0$: $Y \ge X \ge 0$. 

\begin{Oboz}
Обозначим $\Xi(X,Y,S,0)$ за $\Xi(X,Y,S)$.
\end{Oboz}

\begin{Prop} \label{s0}
Пусть $X,Y\in\mathbb{N}_0$: $Y \ge X \ge 0$. Тогда
$$ \Xi(X,Y,0)=1.$$
\end{Prop}
\begin{proof}
 $$\Xi(X,Y,0) = \sum_{\Delta \in \Omega(X,Y,0)}  \prod_{j=1}^{0}e\left(\Delta,j\right)=\sum_{\Delta \in \Omega(X,Y,0)} 1.$$
    
    Нам осталось посчитать количество траекторий $\Delta \in \Omega(X,Y,0)$, то есть траекторий с началом в $Y$, концом в $X$ и нулём шагов ''вверх``. Ясно, что единственная такая траектория -- траектория, состоящая только из $(Y-X)$ шагов ``вниз'' (помним, что $Y\ge X$), поэтому
    $$\Xi(X,Y,0) = \sum_{\Delta \in \Omega(X,Y,0)}  \prod_{j=1}^{0}e\left(\Delta,j\right)=\sum_{\Delta \in \Omega(X,Y,0)} 1=1.$$
    Что и требовалось.
\end{proof}

\begin{Prop} \label{vazhno}
Пусть $X,Y,S\in\mathbb{N}_0$: $Y \ge X \ge 0$. Тогда:

\begin{enumerate}
    \item если $ X \ge 2$ и $S\ge 1$, то $$\Xi(X,Y,S)=\Xi(X-1,Y,S)-\Xi(X-2,Y,S-1)\cdot (X-1);$$
    \item $$ \Xi(1,Y,S)=\Xi(0,Y,S).$$
\end{enumerate}
\end{Prop}
\begin{proof}

По формуле для $\Xi$ мы хотим доказать следующее:
\begin{enumerate}
    \item $$\Xi(X,Y,S)=\Xi(X-1,Y,S)-\Xi(X-2,Y,S-1)\cdot (X-1) \Longleftrightarrow$$ 
    $$\Longleftrightarrow\sum_{\Delta \in \Omega(X-1,Y,S)}  \prod_{j=1}^{S}e\left(\Delta,j\right)-\sum_{\Delta \in \Omega(X,Y,S)}  \prod_{j=1}^{S}e\left(\Delta,j\right)=(X-1)\sum_{\Delta \in \Omega(X-2,Y,S-1)}  \prod_{j=1}^{S}e\left(\Delta,j\right)\Longleftrightarrow$$
    $$\Longleftrightarrow \text{(По Следствию \ref{Kesha})}\Longleftrightarrow$$
    $$\Longleftrightarrow\sum_{\lambda \in \Lambda(X-1,Y,S)}  \prod_{j=1}^{S}\lambda_j-\sum_{\lambda \in \Lambda(X,Y,S)}  \prod_{j=1}^{S}\lambda_j=(X-1)\sum_{\lambda \in \Lambda(X-2,Y,S-1)}  \prod_{j=1}^{S}\lambda_j;$$
    
    \item $$\Xi(1,Y,S)=\Xi(0,Y,S) \Longleftrightarrow \sum_{\Delta \in \Omega(1,Y,S)}  \prod_{j=1}^{S}e\left(\Delta,j\right)=\sum_{\Delta \in \Omega(0,Y,S)}  \prod_{j=1}^{S}e\left(\Delta,j\right)\Longleftrightarrow \text{(По Следствию \ref{Kesha})}\Longleftrightarrow$$
    $$\Longleftrightarrow\sum_{\lambda \in \Lambda(1,Y,S)}  \prod_{j=1}^{S}\lambda_j=\sum_{\lambda \in \Lambda(0,Y,S)}  \prod_{j=1}^{S}\lambda_j.$$
\end{enumerate}
\renewcommand{\labelenumi}{\arabic{enumi}$)$}
\renewcommand{\labelenumii}{\arabic{enumii}$)$}

Вспомним, что $\lambda \in \Lambda(X-1,Y,S)$ тогда и только тогда, когда выполняются следующие 4 условия:

\begin{enumerate}
\item $\lambda_1 \le Y+1$;

\item $\forall i \in \{1,...,S-1\} \quad \lambda_{i+1} \le \lambda_{i}+1 $;

\item $\lambda_{S} \ge X-1$;

\item $\forall i \in \{1,...,S\}\quad\lambda_{i} \ge 1 $.
\end{enumerate}

Также вспомним, что $\lambda \in \Lambda(X,Y,S)$ тогда и только тогда, когда выполняются следующие 4 условия:

\begin{enumerate}
\item $\lambda_1 \le Y+1$;

\item $\forall i \in \{1,...,S-1\}\quad\lambda_{i+1} \le \lambda_{i}+1 $;

\item $\lambda_{S} \ge X$;

\item $\forall i \in \{1,...,S\}\quad\lambda_{i} \ge 1 $.
\end{enumerate}

Из этого несложно сделать вывод о том, что $\Lambda(X,Y,S)\subset\Lambda(X-1,Y,S)$, а значит
$$\sum_{\lambda \in \Lambda(X-1,Y,S)}  \prod_{j=1}^{S}\lambda_j-\sum_{\lambda \in \Lambda(X,Y,S)}  \prod_{j=1}^{S}\lambda_j=\sum_{\lambda \in \Lambda(X-1,Y,S)\textbackslash\Lambda(X,Y,S)}  \prod_{j=1}^{S}\lambda_j.$$

Из определений $\Lambda(X-1,Y,S)$ и $\Lambda(X,Y,S)$ несложно понять, что $\lambda \in \Lambda(X-1,Y,S)\textbackslash\Lambda(X,Y,S)$, если выполняются следующие 4 условия:

\begin{enumerate}
\item $\lambda_1 \le Y+1$;

\item $\forall i \in \{1,...,S-1\}\quad\lambda_{i+1} \le \lambda_{i}+1 $;

\item $\lambda_{S} = X-1$;

\item $\forall i \in \{1,...,S\}\quad\lambda_{i} \ge 1 $.
\end{enumerate}

Теперь будем доказывать оба пункта по отдельности:
\begin{enumerate}
    \item Несложно заметить, что эти четыре условия равносильны следующим пяти:

\begin{enumerate}
\item $\lambda_1 \le Y+1$;

\item $\forall i \in \{1,...,S-2\}\quad\lambda_{i+1} \le \lambda_{i}+1 $;

\item $\lambda_{S-1} \ge X-2$;

\item $\forall i \in \{1,...,S-1\}\quad\lambda_{i} \ge 1 $;

\item $\lambda_{S} = X-1$.
\end{enumerate}

А это четыре условия для $\lambda\in\Lambda(X-2,Y,S-1)$ вместе с условием о том, что $\lambda_{S} = X-1$. А отсюда уже нетрудно понять, что
$$\sum_{\lambda \in \Lambda(X-1,Y,S)\textbackslash\Lambda(X,Y,S)}  \prod_{j=1}^{S}\lambda_j=(X-1)\sum_{\lambda \in \Lambda(X-2,Y,S-1)}  \prod_{j=1}^{S}\lambda_j.$$
Что и требовалось.

\item
В данном случае из условий для $\lambda \in \Lambda(X-1,Y,S)\textbackslash\Lambda(X,Y,S)$ ясно, что $\lambda \in \Lambda(0,Y,S)\textbackslash\Lambda(1,Y,S)$ тогда и только тогда, когда выполняются следующие четыре условия:
\begin{enumerate}
\item $\lambda_1 \le Y+1$;

\item $\forall i \in \{1,...,S-1\}\quad\lambda_{i+1} \le \lambda_{i}+1 $;

\item $\lambda_{S} = 0$;

\item $\forall i \in \{1,...,S\}\quad\lambda_{i} \ge 1 $.
\end{enumerate}

Но одновременное выполнение условий $\lambda_{|S|}=0$ и $\lambda_{|S|}\ge 1$ невозможно, поэтому 
$$\sum_{\lambda \in \Lambda(0,Y,S)}  \prod_{j=1}^{S}\lambda_j-\sum_{\lambda \in \Lambda(1,Y,S)}  \prod_{j=1}^{S}\lambda_j=\sum_{\lambda \in \Lambda(0,Y,S)\textbackslash\Lambda(1,Y,S)}  \prod_{j=1}^{S}\lambda_j=0.$$
\end{enumerate}
\end{proof}

Теперь давайте посчитаем $\Xi(0,Y,S)=\Xi(1,Y,S)$.

Но для начала докажем вспомогательное утверждение.

\begin{Prop} \label{zaebalo}
Пусть $Y,S\in\mathbb{N}$. Тогда:
\begin{enumerate}
    \item
$$\Xi(0,Y,S)=\sum_{e=1}^{Y+1}\left(e \cdot \Xi(0,e,S-1)\right)=\Xi(0,Y-1,S)+(Y+1)\cdot\Xi(0,Y+1,S-1);$$

    \item 
$$\Xi(0,0,S)=\Xi(0,1,S-1).$$
\end{enumerate}
\end{Prop}
\begin{proof}
\renewcommand{\labelenumi}{\arabic{enumi}$)$}
\renewcommand{\labelenumii}{\arabic{enumii}$)$}
$\;$

\begin{enumerate}
    \item По Следствию \ref{Kesha}
    $$\Xi(0,Y,S)=\sum_{e=1}^{Y+1}\left(e \cdot \Xi(0,e,S-1)\right)\Longleftrightarrow\sum_{\lambda \in \Lambda(0,Y,S)}  \prod_{j=1}^{S}\lambda_j=\sum_{e=1}^{Y+1}\left(e \cdot\sum_{\lambda \in \Lambda(0,e,S-1)} \right).$$
    
    Вспомним, что $\lambda \in \Lambda(0,Y,S)$ тогда и только тогда, когда выполняются следующие 4 условия:
    \begin{enumerate}
    \item $\lambda_1 \le Y+1$;

    \item $\forall i \in \{1,...,S-1\}\quad\lambda_{i+1} \le \lambda_{i}+1 $;

    \item $\lambda_{S} \ge 0$;

    \item $\forall i \in \{1,...,S\}\quad\lambda_{i} \ge 1 $.
    \end{enumerate}

    Ясно, что $\lambda_1$ может принимать значения от $0$ до $y+1$.
    
    При условии, что $\lambda_1=e$, четыре условия для $\lambda \in \Lambda(0,Y,S)$ превращаются в следующие пять:
    \begin{enumerate}
    \item $\lambda_1 = e$;
    
    \item $\lambda_2 \le e+1$;

    \item $\forall i \in \{2,...,S-1\}\quad\lambda_{i+1} \le \lambda_{i}+1 $;

    \item $\lambda_{S} \ge 0$;

    \item $\forall i \in \{2,...,S\}\quad\lambda_{i} \ge 1 $.
    \end{enumerate}
    
    Несложно заметить, что все эти условия, кроме первого -- это условия того, что $\lambda'=\{\lambda_2,...,\lambda_S\}\in\Lambda(0,e,S+1)$.
    
    Из этого следует первое равенство.
    
    Теперь давайте докажем второе равенство:
    $$\sum_{e=1}^{Y+1}\left(e \cdot \Xi(0,e,S-1)\right)=\sum_{e=1}^{Y}\left(e \cdot \Xi(0,e,S-1)\right)+(Y+1)\cdot\Xi(0,Y+1,S-1)=$$
    $$\text{(По первому равенству)}=\Xi(0,Y-1,S)+(Y+1)\cdot\Xi(0,Y+1,S-1).$$
    
    \item По Следствию \ref{Kesha}
    $$\Xi(0,0,S)=\Xi(0,1,S-1)\Longleftrightarrow\sum_{\lambda \in \Lambda(0,0,S)}  \prod_{j=1}^{S}\lambda_j=\sum_{\lambda \in \Lambda(0,1,S-1)}  \prod_{j=1}^{S}\lambda_j.$$
    
    Вспомним, что $\lambda \in \Lambda(0,Y,S)$ тогда и только тогда, когда выполняются следующие 4 условия:
    \begin{enumerate}
    \item $\lambda_1 \le 1$;

    \item $\forall i \in \{1,...,S-1\}\quad\lambda_{i+1} \le \lambda_{i}+1 $;

    \item $\lambda_{S} \ge 0$;

    \item $\forall i \in \{1,...,S\}\quad\lambda_{i} \ge 1 $.
    \end{enumerate}

    Из первого и четвёртого условия ясно, что $\lambda_1$ может принимать только значение $1$.
    
    А из того, что $\lambda_1=1$ следует, что четыре условия для $\lambda \in \Lambda(0,Y,S)$ превращаются в следующие пять:
    \begin{enumerate}
    \item $\lambda_1 = 1$;
    
    \item $\lambda_2 \le 2$;

    \item $\forall i \in \{2,...,S-1\}\quad\lambda_{i+1} \le \lambda_{i}+1 $;

    \item $\lambda_{S} \ge 0$;

    \item $\forall i \in \{2,...,S\}\quad\lambda_{i} \ge 1 $.
    \end{enumerate}
    
    Несложно заметить, что все эти условия, кроме первого -- это условия того, что $\lambda'=\{\lambda_2,...,\lambda_S\}\in\Lambda(0,1,S-1)$.
    
    Из этого следует нужное равенство.
\end{enumerate}
\end{proof}

\begin{Lemma} \label{x0}
Пусть $Y\in\mathbb{N}_0$ и $S\in\mathbb{N}$. Тогда
     $$\Xi(0,Y,S)=\binom{Y+2S}{Y}(2S-1)!!.$$
\end{Lemma}
\begin{proof}
$\;$

Давайте доказывать это по индукции по $S$, а при данном $S$ по индукции по $Y$:

База: $S=1$.

Рассмотрим два случая:
\renewcommand{\labelenumi}{\arabic{enumi}$^\circ$}
\begin{enumerate}
    \item Пусть $Y=0$.

В данном случае
$$\Xi(0,0,1)=\text{(По Утверждению \ref{zaebalo})}=\Xi(0,1,0)=\text{(По Утверждению \ref{s0})}=1=\binom{1}{1}(1)!!.$$
Что и требовалось.

\item Пусть $Y>0$.

В данном случае
$$\Xi(0,Y,1)=\text{(По Утверждению \ref{zaebalo})}=\sum_{e=1}^{Y+1}\left(e \cdot \Xi(0,e,0)\right)=$$
$$=\text{(По Утверждению \ref{s0})}=\sum_{e=1}^{Y+1}e=\frac{(Y+1)(Y+2)}{2}=\binom{Y+2}{Y}(1)!!.$$
Что и требовалось.
\end{enumerate}

Переход: $S\ge2$ 

Разберём те же случаи.
\begin{enumerate}
    \item Пусть $Y=0$.

В данном случае
$$\Xi(0,0,S+1)=\text{(По Утверждению \ref{zaebalo})}=\Xi(0,1,S)=\text{(По предположению)}=$$
$$=\binom{1+2S}{2S}(2S-1)!!=(2S+1)!!=\binom{2S+2}{0}(2S+1)!!.$$

\item Пусть $Y>0$.

В данном случае по Утверждению \ref{zaebalo}
$$\Xi(0,Y,S+1)=\sum_{e=1}^{Y+1}\left(e \cdot \Xi(0,e,S)\right)=\Xi(0,Y-1,S+1)+(Y+1)\cdot\Xi(0,Y+1,S)=$$
$$=\text{(По предположению)}=$$
$$=\binom{Y+2S+1}{Y-1}(2S+1)!!+(Y+1)\cdot\binom{Y+2S+1}{Y+1}(2S-1)!!=$$
$$=(2S-1)!!\left(\frac{(Y+2S+1)!(2S+1)}{(Y-1)!(2S+2)!}+\frac{(Y+2S+1)!(Y+1)}{(Y+1)!(2S)!}\right)=$$
$$=(2S-1)!!\left(\frac{(Y+2S+1)!}{(Y-1)!(2S)!(2S+2)}+\frac{(Y+2S+1)!}{(Y)!(2S)!}\right).$$

Хотим доказать, что это выражение равняется
$$\binom{Y+2S+2}{Y}(2S+1)!!=\frac{(Y+2S+2)!}{(Y)!(2S+2)!}(2S+1)!!=$$
$$=\frac{(Y+2S+2)!}{(Y)!(2S)!(2S+2)}(2S-1)!!.$$

Сократим обо выражения на
$$\frac{(Y+2S+1)!}{(Y-1)!(2S)!}(2S-1)!!$$
и поймём, что мы хотим доказать, что
$$\frac{1}{2S+2}+\frac{1}{Y}=\frac{Y+2S+2}{(Y)(2S+2)}.$$
А это очевидно.
\end{enumerate}

Лемма доказана.
\end{proof}

\begin{Col} \label{pzd}
Пусть $Y\in\mathbb{N}_0$ и $S\in\mathbb{N}$. Также считаем, что $(-1)!!=1$. Тогда
     $$\Xi(0,Y,S)=\binom{Y+2S}{Y}(2S-1)!!.$$
\end{Col}
\begin{proof}
Рассмотрим два случая:
\begin{enumerate}
    \item Пусть $S>0$.
    
    В данном случае это Лемма \ref{x0}.
    \item Пусть $S=0$.
    
    В данном случае
    $$\Xi(0,Y,0)=(\text{по Утверждению \ref{s0}})=1=\binom{Y}{Y}(-1)!!.$$
\end{enumerate}
\end{proof}

Итак, $\Xi(0,Y,S)=\Xi(1,Y,S)$ при $Y,S\in\mathbb{N}_0$ мы научились считать. Осталось посчитать $\Xi(X,Y,S)$ при $X \ge 2$:

\begin{Lemma} \label{final}
Пусть $X,Y,S\in\mathbb{N}_0$ и $m(X,S)=\min\{\lfloor\frac{X}{2}\rfloor,S\}$. Тогда
$$\Xi(X,Y,S)= \Xi(0,Y,S)+\sum_{i=1}^{m(X,S)}\left((-1)^i\binom{X}{2i}(2i-1)!! \cdot \Xi(0,Y,S-i)\right).$$
\end{Lemma}
\begin{proof}
Давайте докажем лемму по индукции по $X$.

База: $X=0$ и $X=1$.

В данном случае ясно, что $m=\min\{\lfloor\frac{X}{2}\rfloor,S\}=\min\{\lfloor0,S\}=0$, а значит мы хотим доказать следующие два равенства:
$$\Xi(0,Y,S)= \Xi(0,Y,S);\; \Xi(1,Y,S)= \Xi(0,Y,S).$$

Первое равенство является тавтологией, а второе верно по Утверждению \ref{vazhno}.

Переход к $X\ge 2$:

Рассмотрим два случая:
\renewcommand{\labelenumi}{\arabic{enumi}$^\circ$}
\renewcommand{\labelenumii}{\arabic{enumi}.\arabic{enumii}$^\circ$}
\renewcommand{\labelenumiii}{\arabic{enumi}.\arabic{enumii}.\arabic{enumiii}$^\circ$}
\begin{enumerate}
    \item Пусть $S=0$.
    
    В данном случае ясно, что $m(X,S)=\min\{\lfloor\frac{X}{2}\rfloor,S\}=\min\{\lfloor\frac{X}{2}\rfloor,0\}=0$, а значит мы хотим доказать, что
    $$\Xi(X,Y,0)= \Xi(0,Y,0).$$
    
А по Утверждению \ref{s0} мы знаем, что 
$$\Xi(X,Y,0)=1= \Xi(0,Y,0).$$

    \item Пусть $S\ge 1$.
    
По Утверждению \ref{vazhno} мы знаем, что при $X\ge2$ и $S\ge1$
$$\Xi(X,Y,S)=\Xi(X-1,Y,S)-(X-1)\cdot\Xi(X-2,Y,S-1) =\text{(По предположению)}=$$
$$=\Xi(0,Y,S)+\sum_{i=1}^{m(X-1,S)}\left((-1)^i\binom{X-1}{2i}(2i-1)!! \cdot \Xi(0,Y,S-i)\right)-$$
$$-(X-1)\Xi(0,Y,S-1)-(X-1)\sum_{i=1}^{m(X-2,S-1)}\left((-1)^i\binom{X-2}{2i}(2i-1)!! \cdot \Xi(0,Y,S-i-1)\right)=$$
$$=\Xi(0,Y,S)+\sum_{i=1}^{m(X-1,S)}\left((-1)^i\binom{X-1}{2i}(2i-1)!! \cdot \Xi(0,Y,S-i)\right)-$$
$$-(X-1)\Xi(0,Y,S-1)-(X-1)\sum_{i=2}^{m(X-2,S-1)+1}\left((-1)^(i-1)\binom{X-2}{2i-2}(2i-3)!! \cdot \Xi(0,Y,S-i)\right).$$
Заметим, что $m(X-2,S-1)+1=\min\{\lfloor\frac{X-2}{2}\rfloor,S-1\}+1=\min\{\lfloor\frac{X}{2}\rfloor,S\}=m(X,S)$, а значит то выражение равняется следующему:
$$\Xi(0,Y,S)+\sum_{i=1}^{m(X-1,S)}\left((-1)^i\binom{X-1}{2i}(2i-1)!! \cdot \Xi(0,Y,S-i)\right)-$$
$$-(X-1)\Xi(0,Y,S-1)-(X-1)\sum_{i=2}^{m(X,S)}\left((-1)^{i-1}\binom{X-2}{2i-2}(2i-3)!! \cdot \Xi(0,Y,S-i)\right).$$

Хотим доказать, что это выражение равняется 
$$\Xi(0,Y,S)+\sum_{i=1}^{m(X,S)}\left((-1)^i\binom{X}{2i}(2i-1)!! \cdot \Xi(0,Y,S-i)\right).$$

То есть мы хотим доказать следующее равенство:
$$\sum_{i=1}^{m(X,S)}\left((-1)^i\binom{X}{2i}(2i-1)!! \cdot \Xi(0,Y,S-i)\right)=$$
$$=\sum_{i=1}^{m(X-1,S)}\left((-1)^i\binom{X-1}{2i}(2i-1)!! \cdot \Xi(0,Y,S-i)\right)-$$
$$-(X-1)\Xi(0,Y,S-1)-(X-1)\sum_{i=2}^{m(X,S)}\left((-1)^{i-1}\binom{X-2}{2i-2}(2i-3)!! \cdot \Xi(0,Y,S-i)\right).$$

Рассмотрим следующие подслучаи:

\begin{enumerate}
    \item Пусть $X\underset{2}{\equiv} 1$. 
    
    В данном случае ясно, что $m(X-1,S)=\min\{\lfloor\frac{X-1}{2}\rfloor,S\}=\min\{\lfloor\frac{X}{2}\rfloor,S\}=m(X,S)$, а значит несложно заметить, что нам достаточно доказать следующие равенства:
    \begin{itemize}
        \item $$(-1)^1\binom{X}{2}(1)!!=(-1)^1\binom{X-1}{2}(1)!!-(X-1);$$
        \item (Только если $m(X,S)\ge2$) $\forall i\in\{2,...,m(X,S)\}$ $$(-1)^i\binom{X}{2i}(2i-1)!!=(-1)^i\binom{X-1}{2i}(2i-1)!!-(X-1)(-1)^{i-1}\binom{X-2}{2i-2}(2i-3)!!.$$
    \end{itemize}
    
    Давайте докажем эти равенства:
    
    \begin{itemize}
        \item $$(-1)^1\binom{X}{2}(1)!!=(-1)^1\binom{X-1}{2}(1)!!-(X-1)\Longleftrightarrow$$
        $$\Longleftrightarrow-\frac{X(X-1)}{2}=-\frac{(X-1)(X-2)}{2}-(X-1)\Longleftrightarrow\frac{X}{2}=\frac{(X-2)}{2}+1.$$
        Что и требовалось.
        
        \item Пусть $i\in\{2,...,m(X,S)\}$. Тогда $$(-1)^i\binom{X}{2i}(2i-1)!!=$$
        $$=(-1)^i\binom{X-1}{2i}(2i-1)!!-(X-1)(-1)^{i-1}\binom{X-2}{2i-2}(2i-3)!!\Longleftrightarrow$$
        $$\Longleftrightarrow(-1)^i\frac{(X)!}{(2i)!(X-2i)!}(2i-1)!!=$$
        $$=(-1)^i\frac{(X-1)!}{(2i)!(X-2i-1)!}(2i-1)!! -(X-1)(-1)^{i-1}\frac{(X-2)!}{(2i-2)!(X-2i)!}(2i-3)!! $$
        
        Разделим обе части равенства на $$(-1)^i\frac{(X-1)!}{(2i-2)!(X-2i-1)!}(2i-3)!!$$
        и получим следующее:
        $$\frac{X}{2i(X-2i)}=\frac{1}{2i} +\frac{1}{(X-2i)}.$$
        А это равенство очевидно.
    \end{itemize}

    \item Пусть $X\underset{2}{\equiv} 0$.
    
    Давайте рассмотрим ``подподслучаи'':
    \begin{enumerate}
        \item Пусть $S\le \frac{X}{2}-1$.
        
        В данном случае ясно, что 
    
        $m(X-1,S)=\min\{\lfloor\frac{X-1}{2}\rfloor,S\}=\min\{\lfloor\frac{X}{2}\rfloor-1,S\}=S=m(X,S)$.
        
        А это означает, что мы можем проделать то же доказательство, что и в прошлом случае.
        
        \item Пусть $S\ge \frac{X}{2}$.
        
        В данном случае ясно, что 
    
        $m(X-1,S)=\min\{\lfloor\frac{X-1}{2}\rfloor,S\}=\min\{\lfloor\frac{X}{2}\rfloor-1,S\}=\frac{X}{2}-1=\min\{\lfloor\frac{X}{2}\rfloor,S\}-1=m(X,S)-1$.

    а значит, несложно заметить, что нам достаточно доказать следующие равенства:
    \begin{itemize}
        \item $$(-1)^1\binom{X}{2}(1)!!=(-1)^1\binom{X-1}{2}(1)!!-(X-1);$$
        \item (Только если $m(X,S)\ge3$) $\forall i\in\{2,...,m(X,S)-1\}$ $$(-1)^i\binom{X}{2i}(2i-1)!!=(-1)^i\binom{X-1}{2i}(2i-1)!!-(X-1)(-1)^{i-1}\binom{X-2}{2i-2}(2i-3)!!;$$
        \item (Только если $m(X,S)\ge2$)
        $$(-1)^{\frac{X}{2}}\binom{X}{X}(X-1)!!=-(X-1)(-1)^{\frac{X}{2}-1}\binom{X-2}{X-2}(X-3)!!.$$
    \end{itemize}
    
    Давайте докажем эти равенства:
    \begin{itemize}
        \item Нетрудно убедиться в том, что это доказывается абсолютно аналогично предыдущему случаю.
        \item Также нетрудно убедиться в том, что это доказывается абсолютно аналогично предыдущему случаю.
        \item  Понятно, что $$(-1)^{\frac{X}{2}}\binom{X}{X}(X-1)!!=-(X-1)(-1)^{\frac{X}{2}-1}\binom{X-2}{X-2}(X-3)!!\Longleftrightarrow(X-1)!!=(X-1)!!.$$
        Что и требовалось
    \end{itemize}
    \end{enumerate}
\end{enumerate}
\end{enumerate}

Все случаи разобраны.

Переход доказан.

Лемма доказана.
\end{proof}

\begin{Zam} \label{final2}
Пусть $X,Y,S\in\mathbb{N}_0$ и $m(X,S)=\min\{\lfloor\frac{X}{2}\rfloor,S\}$.

Тогда если считать, что $(-1)!!=1$, то Лемма \ref{final} принимает следующий вид:
$$\Xi(X,Y,S)=\sum_{i=0}^{m(X,S)}\left((-1)^i\binom{X}{2i}(2i-1)!! \cdot \Xi(0,Y,S-i)\right).$$
 
\end{Zam}

Итак, пусть $X,Y,S\in\mathbb{N}_0:Y\ge X$. Тогда $\Xi(X,Y,S)$ мы нашли. А значит по Утверждению \ref{sdvig} мы нашли и $\Xi(X,Y,S,i)$ при $X,Y,S\in\mathbb{N}_0:Y\ge X\ge i$.

Таким образом, мы решили Задачу \ref{wehavesome} и готовы получить формулу для $d(x,y,S)$ с небольшим числом слагаемых. Но для начала надо ввести следующее обозначение:

\begin{Oboz} \label{krasota}
Пусть $x,y \in \mathbb{YF}$: $|y| \ge |x|$ и $z=h(x,y)$. Тогда 
$$F(x,y,i)={f\left(x,i,z\right)}\prod_{j=1}^{d(y)}\left(g\left(y,j\right)-i\right).$$ 
\end{Oboz}

\begin{Zam} \label{finalfinal}
Пусть $x,y \in \mathbb{YF}$: $|y| \ge |x|$ и $z=h(x,y)$. Тогда 
\begin{itemize}
    \item $$d(x,y)=\sum_{i=0}^{|x|}F(x,y,i);$$
    \item $$d(x,y,\Delta)=\sum_{i=0}^{|x|}\left(F(x,y,i) \prod_{j=1}^{s(\Delta)}\left(e\left(\Delta,j\right)-i\right)\right);$$
    \item Из Замечания \ref{yavf} следует, что у нас есть явная формула для $F(x,y,i)$, и в ней  $O(|x|)$ слагаемых; 
    \item На самом деле Обозначение \ref{krasota} нужно только для того, чтобы формула для $d(x,y,S)$ поместилась на одной строчке.
\end{itemize}
\end{Zam}

\begin{theorem} \label{theend}
Пусть $x,y\in\mathbb{YF}:$ $|y|\ge |x|$, $z=h(x,y)$, $S\in\mathbb{N}_0$ и $m(X,S)=\min\{\lfloor\frac{X-1}{2}\rfloor,S\}$. Также считаем, что $(-1)!!=1$. Тогда
$$d(x,y,S)=\sum_{i=0}^{|x|}\left(\left(\sum_{k=0}^{m(|x|-i,S)}\left((-1)^k\binom{|x|-i}{2k}(2k-1)!! \cdot \binom{|y|-i+2S-2k}{|y|-i}(2S-2k-1)!!  \right)\right)  F(x,y,i)\right).$$
\end{theorem}

\begin{proof}
Как мы знаем,
$$d(x,y,S)=$$
$$=\sum_{\Delta \in \Omega(|x|,|y|,S)} \sum_{i=0}^{|x|}\left( {f\left(x,i,z\right)}\prod_{j=1}^{d(y)}\left(g\left(y,j\right)-i\right)\prod_{j=1}^{s(\Delta)}\left(e\left(\Delta,j\right)-i\right)\right)=$$
$$= \sum_{i=0}^{|x|}\left(\left( \sum_{\Delta \in \Omega(x,y,S)}  \prod_{j=1}^{S}\left(e\left(\Delta,j\right)-i\right) \right) {f\left(x,i,z\right)}\prod_{j=1}^{d(y)}\left(g\left(y,j\right)-i\right)\right)=$$
$$=\sum_{i=0}^{|x|}\left(\Xi(|x|,|y|,S,i)\cdot  F(x,y,i)\right)=$$
$$=\text{(По Утверждению \ref{sdvig})}=$$
$$=\sum_{i=0}^{|x|}\left(\Xi(|x|-i,|y|-i,S,0)\cdot  F(x,y,i\right)=$$
$$=\sum_{i=0}^{|x|}\left(\Xi(|x|-i,|y|-i,S)\cdot  F(x,y,i\right)=$$
$$= \text{(По Замечанию \ref{final2})}=$$
$$=\sum_{i=0}^{|x|}\left(\left(\sum_{k=0}^{m(|x|-i,S)}\left((-1)^k\binom{|x|-i}{2k}(2k-1)!! \cdot \Xi(0,|y|-i,S-k)\right)\right)\cdot  F(x,y,i)\right)=$$
$$=\text{(По Следствию \ref{pzd})}=$$
$$=\sum_{i=0}^{|x|}\left(\left(\sum_{k=0}^{m(|x|-i,S)}\left((-1)^k\binom{|x|-i}{2k}(2k-1)!! \cdot \binom{|y|-i+2S-2k}{|y|-i}(2S-2k-1)!!  \right)\right)  {F\left(x,i,z\right)}\right).$$

Теорема доказана.
\end{proof}

\begin{Zam}
$\;$
\begin{enumerate}
    \item У нас есть явная формула для $d(x,y,S)$;
    \item В ней $O(|x|^3)$ слагаемых.
\end{enumerate}
\end{Zam}
\begin{proof}
Из Замечания \ref{finalfinal} и того, что 
$$d(x,y,S)=\sum_{i=0}^{|x|}\left(\left(\sum_{k=0}^{m(|x|-i,S)}\left((-1)^k\binom{|x|-i}{2k}(2k-1)!! \cdot \binom{|y|-i+2S-2k}{|y|-i}(2S-2k-1)!!  \right)\right)  {F\left(x,i,z\right)}\right),$$
мгновенно следует первая часть.

Также из Замечания \ref{finalfinal} следует, что количество слагаемых в 
$$\left(\sum_{k=0}^{m(|x|-i,S)}\left((-1)^k\binom{|x|-i}{2k}(2k-1)!! \cdot \binom{|y|-i+2S-2k}{|y|-i}(2S-2k-1)!!  \right)\right)  {F\left(x,i,z\right)}$$
при каком-то $i\in\left\{0,...,|x|\right\}$ равно $${m(|x|-i,S)}O(|x|)=\min\left\{\left\lfloor\frac{|x|-i}{2}\right\rfloor,S\right\}O(|x|)\le |x|O(|x|)=O(|x|^2), $$
а значит несложно заметить, что количество слагаемых в $d(x,y,S)$
равно $(|x|+1)O(|x|^2)=O(|x|^3)$.

Что и требовалось.
\end{proof}

\begin{Col}
Пусть $x\in\mathbb{YF}$, $S\in\mathbb{N}_0$. Также считаем, что $(-1)!!=1$. Тогда
$$d(\varepsilon,x,S)=\binom{|x|+2S}{|x|}(2S-1)!! \prod_{j=1}^{d(x)}g\left(x,j\right).$$
\end{Col}
\begin{proof}
$$d(\varepsilon,x,S)=\text{(По Теореме \ref{theend}, если считать $(-1)!!=1$})=$$
$$=\sum_{i=0}^{|\varepsilon|}\left(\left(\sum_{k=0}^{m(|\varepsilon|-i,S)}\left((-1)^k\binom{|\varepsilon|-i}{2k}(2k-1)!! \cdot \binom{|x|-i+2S-2k}{|x|-i}(2S-2k-1)!!  \right)\right)  F(\varepsilon,x,i)\right)=$$
$$=\left(\sum_{k=0}^{m(|\varepsilon|,S)}\left((-1)^k\binom{|\varepsilon|}{2k}(2k-1)!! \cdot \binom{|x|+2S-2k}{|x|}(2S-2k-1)!!  \right)\right)  \cdot {f\left(\varepsilon,0,h(\varepsilon,x)\right)}\prod_{j=1}^{d(x)}g\left(x,j\right)=$$
$$= \binom{|x|+2S}{|x|}(2S-1)!! \prod_{j=1}^{d(x)}g\left(x,j\right).$$ 

\end{proof}

\end{document}